\documentclass[10pt,oneside]{amsart}
\usepackage[a4paper, textwidth=15.6cm]{geometry}
\usepackage[T1]{fontenc}
\usepackage[english]{babel}
\usepackage{aliascnt}

\usepackage{xcolor} 
\usepackage{tikz}
\usepackage[draft=false]{hyperref} 
\usepackage{mathtools}
\usepackage{amsmath,amsthm,amssymb}
\usepackage[all,cmtip]{xy}

\usepackage[capitalise]{cleveref}

\newtheorem{theorem}{Theorem}[subsection]

\newaliascnt{lemma}{theorem}
\newtheorem{lemma}[lemma]{Lemma}
\aliascntresetthe{lemma}
\crefname{Lemma}{Lemma}{Lemmas}
\Crefname{Lemma}{Lemma}{Lemmas}

\newaliascnt{proposition}{theorem}
\newtheorem{proposition}[proposition]{Proposition}
\aliascntresetthe{proposition}
\crefname{Proposition}{Proposition}{Propositions}
\Crefname{Proposition}{Proposition}{Propositions}

\newaliascnt{corollary}{theorem}
\newtheorem{corollary}[corollary]{Corollary}
\aliascntresetthe{corollary}
\crefname{Corollary}{Corollary}{Corollaries}
\Crefname{Corollary}{Corollary}{Corollaries}

\theoremstyle{definition}
\newaliascnt{definition}{theorem}
\newtheorem{definition}[definition]{Definition}
\aliascntresetthe{definition}
\crefname{Definition}{Definition}{Definitions}
\Crefname{Definition}{Definition}{Definitions}

\newaliascnt{notation}{theorem}
\newtheorem{notation}[notation]{Notation}
\aliascntresetthe{notation}
\crefname{Notation}{notation}{Notations}
\Crefname{Notation}{notation}{Notations}

\newaliascnt{construction}{theorem}
\newtheorem{construction}[construction]{Construction}
\aliascntresetthe{construction}
\crefname{Construction}{construction}{Constructions}
\Crefname{Construction}{construction}{Constructions}

\newaliascnt{example}{theorem}
\newtheorem{example}[example]{Example}
\aliascntresetthe{example}
\crefname{Example}{example}{Examples}
\Crefname{Example}{example}{Examples}

\newaliascnt{remark}{theorem}
\newtheorem{remark}[remark]{Remark}
\aliascntresetthe{remark}
\crefname{Remark}{Remark}{Remarks}
\Crefname{Remark}{Remark}{Remarks}

\usepackage{amsmath,amsthm,amssymb,amsfonts,extarrows}
\usepackage{enumerate,amscd,amsxtra,MnSymbol}
\usepackage{mathrsfs}
\usepackage{color}
\usepackage[cmtip,all]{xy}
\usepackage{eucal}
\usepackage{bbold}
\usepackage{upgreek}
\usepackage{paralist}
\usepackage{tikz-cd}
\usepackage{dsfont}
\usepackage{bbm}
\usepackage[bbgreekl]{mathbbol}

\DeclareMathSymbol\bbDelta \mathord{bbold}{"01}
\DeclareMathSymbol\bDelta \mathord{bbold}{"01}

\newcommand{\bD}{{\mathbb D}}
\newcommand{\bE}{{\mathbb E}}

\newcommand{\bH}{{\mathbb H}}
\newcommand{\bN}{{\mathbb N}}

\newcommand{\mB}{{\mathcal B}}
\newcommand{\mC}{{\mathcal C}}
\newcommand{\mD}{{\mathcal D}}
\newcommand{\mE}{{\mathcal E}}
\newcommand{\mF}{{\mathcal F}}

\newcommand{\mJ}{{\mathcal J}}
\newcommand{\mK}{{\mathcal K}}

\newcommand{\mO}{{\mathcal O}}

\newcommand{\mQ}{{\mathcal Q}}

\newcommand{\mU}{{\mathcal U}}
\newcommand{\mV}{{\mathcal V}}
\newcommand{\mW}{{\mathcal W}}

\newcommand{\A}{{\mathrm A}}

\newcommand{\C}{{\mathrm C}}
\newcommand{\D}{{\mathrm D}}
\newcommand{\E}{{\mathrm E}}
\newcommand{\F}{{\mathrm F}}
\newcommand{\G}{{\mathrm G}}

\renewcommand{\L}{{\mathrm L}}

\newcommand{\N}{{\mathrm N}}

\newcommand{\R}{{\mathrm R}}

\newcommand{\Y}{{\mathrm Y}}

\newcommand{\ra}{\mathrm{a}}

\newcommand{\rc}{\mathrm{c}}

\newcommand{\bi}{\mathrm{i}}
\newcommand{\m}{\mathrm{m}}
\newcommand{\bk}{\mathrm{k}}

\newcommand{\rt}{\mathrm{t}}
\newcommand{\s}{\mathrm{s}}

\newcommand{\n}{\mathrm{n}}
\newcommand{\br}{\mathrm{r}}

\newcommand{\op}{\mathrm{op}}

\newcommand{\Ch}{\mathsf{Ch}}

\newcommand{\colim}{\mathrm{colim}}

\newcommand{\Mod}{{\mathrm{Mod}}}

\newcommand{\rev}{{\mathrm{rev}}}

\newcommand{\Cmon}{{\mathrm{Cmon}}}

\newcommand{\Rig}{{\mathrm{Rig}}}

\newcommand{\ot}{\otimes}

\newcommand{\co}{\mathrm{co}}

\newcommand{\strict}{\mathrm{strict}}

\newcommand{\Pre}{\mathrm{Pre}}

\newcommand{\id}{\mathrm{id}}

\newcommand{\Cat}{\mathrm{Cat}}

\newcommand{\Set}{\mathrm{Set}}
\newcommand{\sSet}{\mathrm{sSet}}

\renewcommand{\Pr}{\mathrm{Pr}}

\newcommand{\Alg}{\mathrm{Alg}}

\newcommand{\Mon}{\mathrm{Mon}}
\newcommand{\Fun}{\mathrm{Fun}}

\newcommand{\lax}{{\mathrm{lax}}}
\newcommand{\oplax}{{\mathrm{oplax}}}

\newcommand{\tu}{{\mathbb 1}}

\newcommand{\ev}{{\mathrm{ev}}}

\newcommand{\Map}{{\mathrm{Map}}}

\newcommand{\bZ}{{\mathbb{Z}}}

\newcommand{\Mor}{{\mathrm{Mor}}} 
 
\newcommand{\inj}{{\mathrm{inj}}}

\newcommand{\fin}{\mathrm{fin}}
\newcommand{\Free}{\mathrm{Free}}

\newcommand{\Steiner}{\mathrm{Steiner}}
\newcommand{\cop}{\mathrm{cop}}

\newcommand{\Ab}{{\mathrm{Ab}}} 
\newcommand{\Sp}{{\mathrm{Sp}}} 

\newcommand{\Fin}{{\mathrm{Fin}}}

\newcommand{\mon}{{\mathrm{mon}}} 
\newcommand{\Sym}{{\mathrm{Sym}}} 
\newcommand{\Strat}{{\mathrm{Strat}}} 
\newcommand{\Sing}{{\mathrm{Sing}}} 
\newcommand{\precomplicial}{{\mathrm{precomplicial}}}

\newcommand{\Grp}{{\mathrm{Grp}}}

\newcommand{\cube}{{\,\vline\negmedspace\square}}

\begin{document}

\title{Homology of higher categories}


\author{Hadrian Heine, \\ Max Planck Institute Bonn, Germany, \\ heine@mpim-bonn.mpg.de}
\maketitle

\begin{abstract}

The classical Dold--Kan correspondence identifies simplicial abelian groups with connective chain complexes.
On the level of homotopy theory, these provide models for strict grouplike $\bE_\infty$-spaces and connective $H\bZ$-module spectra.
We establish a higher-categorical Dold--Kan correspondence identifying stratified simplicial commutative monoids with strict symmetric monoidal $\omega$-categories and with connective categorical $H\bN$-module spectra. This provides a combinatorial, algebraic and stable homotopy-theoretic model for higher-categorical homological algebra.

As a consequence, the free categorical $H\bN$-module spectrum provides a natural notion of homology for higher categories. Through the categorical Dold--Kan correspondence, it is modeled by the free stratified simplicial commutative monoid on the Street nerve, the higher-categorical analogue of singular chains.
We obtain a categorical Dold--Thom theorem, which leads to explicit computations of the categorical homology of the globes and shows that categorical homology detects genuinely higher-categorical information invisible to classical homology.



\end{abstract}

\tableofcontents

\section{Introduction}

\subsubsection{Motivation}

Classical homological algebra is fundamentally additive.
As a consequence, many naturally occurring algebraic and higher-categorical structures -- such as commutative monoids and higher analogues -- lie outside its direct scope.
This raises a fundamental structural question:
\begin{center}
\emph{What is the intrinsic form of homological algebra in higher category theory?}
\end{center}

We give an answer by proving a higher-categorical Dold--Kan correspondence, which identifies the algebraic, combinatorial, and stable homotopy-theoretic forms of homological algebra in higher category theory. 

\subsubsection*{The categorical Dold--Kan correspondence}

The classical Dold--Kan correspondence identifies connective chain complexes with simplicial abelian groups.
On the level of homotopy theory, this gives an equivalence between connective $H\bZ$-module spectra and strict grouplike $\bE_\infty$-spaces.

We establish an analogous correspondence in higher category theory, reflecting the principle that semiadditivity plays the role of additivity in higher category theory.

A priori, neither side of such a correspondence is evident:
it is not clear what should replace simplicial abelian groups, nor what should replace connective chain complexes or connective $H\bZ$-modules.
Our main theorem identifies these structures simultaneously.
The stable side is given by categorical spectra equipped with an action of the categorical Eilenberg--MacLane spectrum $H\bN$ of the natural numbers.
The algebraic side is given by strict symmetric monoidal $\infty$-categories.
The combinatorial side is given by stratified simplicial commutative monoids.

Thus the categorical Dold–Kan correspondence identifies the algebraic, combinatorial, and stable homotopy-theoretic manifestations of homological algebra of higher categories.

\subsubsection*{Categorical homology}

As a fundamental consequence, the categorical Dold–Kan correspondence singles out a canonical intrinsic notion of homology for higher categories.
Just as ordinary homology is obtained by linearizing spaces through the free $H\bZ$-module spectrum, categorical homology is obtained by the corresponding semiadditive linearization through the free categorical $H\bN$-module spectrum.
It satisfies a categorical version of Eilenberg–Steenrod axioms \cite[6.2.]{heine2026stable}, and hence defines a categorical homology theory, which we call categorical homology.
The categorical Dold–Kan correspondence turns this intrinsic homology theory into a concrete theory of categorical chains.

The singular complex presents a space combinatorially in terms of maps from standard simplices.
In higher category theory, the corresponding role is played by the Street nerve.
For every $n\geq 0$, Street \cite{street1972two} constructs an $n$-category $\bDelta^n$, the oriented $n$-simplex.
It is a higher-categorical refinement of the totally ordered set
\[
[n]=\{0<\cdots<n\}
\]
and plays the role of the topological $n$-simplex.
These oriented simplices give rise to the Street nerve
\[
\N(X):=\Fun(\bDelta^\bullet,X)
\]
of an $\infty$-category $X$, generalizing the ordinary nerve of a category, and playing the role of the singular complex of a space.

Unlike the ordinary nerve, the Street nerve is not fully faithful when regarded as a functor to simplicial sets.
It becomes only fully faithful after passing to stratified simplicial sets, namely simplicial sets equipped with a distinguished class of simplices containing all degenerate simplices
\cite{verity2008complicial,verity2008weak,loubaton2022complicial}.
Stratified simplicial sets therefore provide a natural combinatorial model for higher categories and play a role analogous to simplicial sets in classical homotopy theory.

The categorical Dold--Kan correspondence turns this combinatorial model into a model for categorical homology.
Classically, homology is modeled by the free simplicial abelian group
\[
\bZ[\Sing(X)]
\]
on the singular complex of a space.
In higher category theory, categorical homology is modeled by the free stratified simplicial commutative monoid
\[
\bN[\N(X)]
\]
on the Street nerve.

Thus the Street nerve plays the role of the singular complex, while the free stratified simplicial commutative monoid provides the higher-categorical analogue of the simplicial chain construction.

Concretely, the objects of categorical homology of an $\infty$-category $X$ are represented by finite formal sums of objects of $X$, identified up to invertible morphisms in a fibrant replacement of the free stratified simplicial commutative monoid on the Street nerve of $X$.
The fibrant replacement coherently adds those higher-dimensional compositions which do not already exist.
This description makes categorical homology accessible to explicit constructions and computations.

Another consequence of the categorical Dold-Kan correspondence is that 
categorical homology is canonically multiplicative: it refines to a monoidal functor with values in categorical $H\bN$-modules. Consequently, monoidal $\infty$-categories induce categorical $H\bN$-algebras on categorical homology, and dually, categorical cohomology of the oriented simplices carries a categorical analogue of the cup product.

\subsubsection{A categorical Dold-Thom theorem}

The Dold-Kan correspondence yields a categorical version of the Dold--Thom theorem \cite{dold1958quasifaserungen}.
For any $\infty$-category $X$ with distinguished object, we consider the sequential diagram of strict symmetric powers
\[
X\longrightarrow
X^{\times 2}_{\Sigma_2}\longrightarrow
X^{\times 3}_{\Sigma_3}\longrightarrow\cdots .
\]

We identify its colimit with the reduced free strict symmetric monoidal $\infty$-category generated by $X$. 
By the categorical Dold–Kan correspondence, this colimit computes the reduced categorical homology of $X$.
Thus categorical homology admits an infinite symmetric-power description, paralleling the classical Dold–Thom theorem.

These models lead to concrete calculations for globes, the walking higher morphisms, and Street’s oriented simplices. Categorical homology distinguishes globes of different dimensions even though their classifying spaces are contractible. 
Thus categorical homology detects genuinely higher-categorical information that is invisible to ordinary homology.

Taken together, these results identify the mechanism of homological algebra in the higher-categorical semiadditive world.
The categorical Dold–Kan correspondence identifies the algebraic, combinatorial, and stable manifestations of categorical homological algebra, determines an intrinsic homology theory of higher categories, and realizes this theory concretely through the Street nerve and strict symmetric powers.
Together with the stable homotopy theory of higher categories developed in \cite{heine2026stable}, this yields two complementary foundations for an intrinsic homological algebra of higher categories: the stable homotopy-theoretic structure established in \cite{heine2026stable} and the algebraic-combinatorial structure identified here.

\section{Main results}

\subsubsection*{Categorical Dold-Kan correspondence}

We now state the categorical Dold–Kan correspondence precisely. 
Our main result identifies the structures that play the role of connective $H\bZ$-module spectra, strict grouplike $\bE_\infty$-spaces, and simplicial abelian groups:

\begin{theorem}\label{the22}(\cref{theuta2}) 
There are canonical equivalences between the following structures:

\begin{enumerate}
\item Connective left (right) $\bH \bN$-module categorical spectra.

\item Strict symmetric monoidal $\infty$-categories.

\item Homotopy types of stratified simplicial commutative monoids.

\end{enumerate}

\end{theorem}

\cref{the22} identifies the stable, algebraic, and combinatorial forms of categorical homological algebra. The following results are consequences showing that 
this identification is reflected in the corresponding theory of homology.

\subsubsection*{Combinatorial presentation of categorical homology}

\cref{the22} provides the following combinatorial presentation of categorical homology:

\begin{corollary}\label{the4}(\cref{commodel})
The categorical homology of an $\infty$-category $X$
is modeled by the free stratified simplicial commutative monoid $$\bN[\N(X)]$$ on the Street nerve of $X.$ 
	
\end{corollary}

This description is an analogue of the classical presentation of homology, in which the homology of a space $X$ is modeled by the free simplicial abelian group $\bZ[\Sing(X)]$ on its singular complex.

\subsubsection*{Categorical Dold–Thom theorem}

By the Dold-Thom theorem \cite{dold1958quasifaserungen} the reduced homology of a connected space $X$ is the homotopy of the infinite symmetric product associated to $X$.
We construct for every $\infty$-category with distinguished object $X$ a sequential diagram of strict symmetric powers
$$
X \to X^{\times 2}_{\Sigma_2} \to X^{\times 3}_{\Sigma_3} \to ... $$

We identify its colimit with the free reduced strict symmetric monoidal $\infty$-category generated by $X$.
\cref{the22} implies the following categorical Dold-Thom theorem:

\begin{corollary}\label{klt}(\cref{free})
The categorical homology of an $\infty$-category $X$ with distinguished object is the colimit of the sequential diagram 
$$
X \to X^{\times 2}_{\Sigma_2} \to X^{\times 3}_{\Sigma_3} \to ... $$
	
\end{corollary}

\subsubsection*{Categorical homology of the globes}

We compute the categorical homology of the globes, the walking higher morphisms:

\begin{corollary}\label{the3}(\cref{disk1}, \cref{disk2})
Let $\bk \geq 0.$ 
	
\begin{enumerate}

\item The non-reduced categorical homology of $\bD^\bk$ is the coproduct $$ \coprod_{\n \geq 0} (\bD^\bk)^{\vee \n}.$$

\item The reduced categorical homology of $\bD^\bk$ is the sequential colimit of the diagram $$ \bD^0 \to \bD^\bk \to ... \to (\bD^\bk)^{\vee \n} \to ... $$

\end{enumerate}
	
\end{corollary}

In particular, categorical homology distinguishes globes of different dimensions, even though their classifying spaces are contractible.

\subsubsection*{Categorical homology of the oriented simplices}

We obtain a model for the categorical homology of the oriented simplices:

\begin{corollary}(\cref{Orientmodel})
Let $n \geq 0.$
The free stratified simplicial commutative monoid
$$\bN[\Delta^n]$$ generated by the stratified simplicial set $\Delta^n$ endowed with minimal stratification, is a model for the categorical homology of $\bDelta^n.$

\end{corollary}

\subsubsection{Multiplicativity of categorical homology}

Classical homology is multiplicative and so takes loop spaces to graded rings. Categorical homology is multiplicative, too:

\begin{corollary}\label{the5} 
Categorical homology canonically refines to a monoidal functor from $\infty$-categories with distinguished objects to left $H\bN$-module categorical spectra.
	
\end{corollary}

\cref{the5} implies that categorical homology sends monoidal $\infty$-categories to categorical $H\bN$-algebras.
\cref{the5} also implies the following:

\begin{corollary}\label{homolalg}

Let $n \geq 0.$ The non-reduced categorical homology of the $n$-th oriented simplex $\bDelta^n$ refines to a coalgebra in $ H(\bN)$-module categorical spectra.

\end{corollary}

Taking the linear dual, the categorical cohomology carries a categorical cup product:

\begin{corollary}\label{homolcoalg}

Let $n \geq 0.$ The non-reduced categorical cohomology of the $n$-th oriented simplex $\bDelta^n$ refines to a categorical $H(\bN)$-algebra. 
\end{corollary}

\subsubsection*{Relation to other work}

Recently, several authors have studied spectra of higher categories, the analogue of spectra in higher category theory \cite{kern2024categoricalspectrapointedinftymathbbzcategories}, \cite{masuda2026algebra}, \cite[\S 13.3]{HigherQuasi}, \cite{Scholze2025}.

A different direction toward categorical homological algebra is provided by Dyckerhoff \cite{dyckerhoff2021categorified} who proves a categorified Dold–Kan correspondence for chain complexes of stable $(\infty,1)$-categories in the sense of \cite{christ2023complexes}. Dyckerhoff identifies such chain complexes with suitably enhanced simplicial objects in stable $(\infty,1)$-categories. Conceptually, this categorifies the classical Dold–Kan correspondence by replacing the abelian category of abelian groups with the 2-stable $(\infty,2)$-category of stable $(\infty,1)$-categories. Thus the classical correspondence between connective chain complexes of abelian groups and simplicial abelian groups is lifted one categorical level.

The correspondence established here is of a different nature. Rather than replacing abelian groups by stable $(\infty,1)$-categories, we ask what plays the roles of simplicial abelian groups and connective chain complexes intrinsically in higher category theory. On the combinatorial side, simplicial abelian groups are replaced by stratified simplicial commutative monoids; on the stable side, connective chain complexes, viewed homotopically as connective $H\bZ$-module spectra, are replaced by connective categorical $H\bN$-module spectra. 
The resulting equivalence therefore identifies the semiadditive homological algebra intrinsic to higher category theory.

Since the higher algebra of higher categories contains the theory of commutative monoids, the problem of identifying homological algebra in higher category theory in particular addresses the question of what replaces homological algebra in the semiadditive setting.
This latter question has been considered in several forms within non-additive homological algebra \cite{janelidze2002semi}, \cite{goedecke2009satellites}, \cite{abuhlail2014exact}, \cite{flores2015homological}. 






\subsubsection*{Acknowledgements}

We thank David Gepner, Markus Spitzweck and David White for helpful discussions.

\subsubsection*{Notation and terminology}

We fix a hierarchy of Grothendieck universes whose objects we call small, large, very large, etc.
We call a space small, large, etc. if its set of path components and its homotopy groups are small, large, etc. for any choice of base point. 

We refer to weak $(\infty,n)$-categories for $0 \leq \n \leq \infty$ as $\n$-categories and refer to weak $(\n,\n)$-categories as $(\n,\n)$-categories.
In particular, we refer to $(\infty,1)$-categories as 1-categories or simply categories.

\vspace{1mm}
We write 
\begin{itemize}
\item $\Set$ for the category of small sets.
\item $\Delta$ for (a skeleton of) the category of finite, non-empty, partially ordered sets and order preserving maps, whose objects we denote by $[n] = \{0 < ... < n\}$ for $\n \geq 0$.
\item $\infty\Grp$ for the $\infty$-category of small homotopy types.
\item $ \Cat$ for the $\infty$-category of small $\infty$-categories.	
\item $\Map_\mC(A,B)$ for the space of maps $A \to B$ in $\mC$ for any category $\mC$ containing objects $A,B.$
\item $\Fun(\mC,\mD)$ for the category of functors $\mC \to \mD$ between two categories $\mC,\mD$, the internal hom of $\Cat$ for the cartesian product.

\item We write $\ast$ for the final space.

\item We write $\bD^1$ for the walking arrow, the category with two objects and a unique non-identity arrow.

\item We write $\partial\bD^1$ and $S^0$ for the maximal subspace in $\bD^1$, the set with two elements.

\end{itemize}

\vspace{1mm}

We indicate $\infty$-categories of large objects by $\widehat{(-)}$, for example we write $\widehat{\infty\Grp}, \widehat{\Cat}$ for the $\infty$-categories of large spaces, large categories.

We often call a fully faithful functor $\mC \to \mD$ an embedding.
We call a functor $\mC \to \mD$ an inclusion if it induces an embedding on maximal subspaces and on all mapping spaces. The latter is equivalent to asking that for any category $\mB$ the induced map
$\Map_{\Cat}(\mB,\mC) \to \Map_{\Cat}(\mB,\mD)$ is an embedding.

\section{Recollection on higher categories}

In this section we recall the basic theory of (homotopy coherent) enriched categories, $\infty$-categories, and the Gray tensor product needed in our work.


\subsection{Enriched categories}

We recall the theory of (homotopy coherent) enriched categories \cite{GEPNER2015575}, \cite{Heine2023AnEB}, \cite{heine2025equivalence}, \cite{HINICH2020107129}.


For every presentably monoidal category $\mV$ there is a presentable 2-category $ \mV\mathrm{-}\Cat$ of $\mV$-enriched categories and $\mV$-enriched functors and a forgetful functor $ \iota:\mV\mathrm{-}\Cat \to \Cat$ to the presentable 2-category $\Cat$ of categories.
For $\mV= \infty\Grp$ the category of homotopy types, the forgetful functor $ \iota:\infty\Grp\mathrm{-}\Cat \to \Cat$ is an equivalence by \cite[Corollary 3.23.]{heine2024bienrichedinftycategories}.

\begin{notation}Let $\mV$ be a presentably monoidal category.
A $\mV$-enriched category $\mC$ provides an underlying category $\iota(\mC)$, for all objects $X,Y \in \iota(\mC)$ a morphism object 
$$\Mor_\mC(X, Y) \in \mV$$
and for all objects $X,Y,Z \in \iota(\mC)$ a composition morphism in $\mV:$
$$\Mor_\mC(Y, Z) \ot \Mor_\mC(X, Y) \to \Mor_\mC(X, Z).$$

We write $ X \in \mC$ for $ X \in \iota(\mC)$ and usually notationally identify $\mC$ with $\iota(\mC).$





	
\end{notation}


\begin{notation}
Let $\mV$ be a presentably monoidal category and $\mC, \mD$ be $\mV$-enriched categories. Let $ \mV \mathrm{-}\Fun(\mC,\mD)$ denote the category of $\mV$-enriched functors $\mC \to \mD$ using that $\mV \mathrm{-}\Cat$ is a 2-category.

\end{notation}

\subsubsection{Opposite enrichment}

For every enriched category there is an opposite one:

The following is \cite[Proposition 4.60.]{heine2024bienrichedinftycategories}

\begin{proposition}
There is an involution $$(-)^\circ: \mV\mathrm{-}\Cat \simeq \mV^\rev\mathrm{-}\Cat$$ forming the opposite enriched category.

For every $\mC \in \mV\mathrm{-}\Cat$ and $X,Y \in \mC$
there are canonical equivalences
$$ \iota(\mC^\circ) \simeq \iota(\mC)^\op$$
and $$ \Mor_{\mC^\circ}(X,Y) \simeq \Mor_{\mC}(Y,X).$$
\end{proposition}

\subsubsection{Transfer of enrichment}

The following is \cite[Theorem 3.67., Corollary 3.71., Corollary 3.74.]{heine2024bienrichedinftycategories}:

\begin{proposition}

Let $\mV, \mW$ be presentably monoidal categories and $\phi: \mV \to \mW$ a lax monoidal functor.

\begin{enumerate}

\item There is a canonical functor $\phi_!: \mV\mathrm{-}\Cat \to \mW\mathrm{-}\Cat$
such that for every $\mV$-enriched category $\mC$ there are a canonical essentially surjective map $\kappa: \iota(\mC) \to \iota(\phi_!(\mC)) $ and 
for every $X, Y \in \mC$ a canonical equivalence
$$ \Mor_{\phi_!(\mC)}(\kappa(X), \kappa(Y)) \simeq \Mor_\mC(X,Y).$$

\item Let $\phi: \mV \to \mW$ be a monoidal functor that admits a right adjoint $\gamma.$ There is an adjunction $\phi_!: \mV\mathrm{-}\Cat \rightleftarrows \mW\mathrm{-}\Cat: \gamma_!.$
For every presentably left $\mW$-tensored category $\mD$ the restriction of the left $\mW$-action along $\phi$ identifies with $\gamma_!(\mD).$

\vspace{1mm}

\item If $\phi: \mV \to \mW$ is a lax monoidal embedding, the functor
$\phi_!: \mV\mathrm{-}\Cat \to \mW\mathrm{-}\Cat$ is fully faithful.

\end{enumerate}

\end{proposition}

\begin{remark}\label{indubi} Let $\mU, \mV, \mW $ be presentably monoidal categories and $F: \mV \rightleftarrows \mW :G $ a left adjoint monoidal functor.
The induced left adjoint monoidal functor
$\mU \ot \mV \to \mU \ot \mW $ identifies with the left adjoint
of the induced functor
\begin{equation}\label{homok}
\Fun^R(\mU^\op, \mW) \to \Fun^R(\mU^\op,\mV).\end{equation}

So if $G: \mW \to \mV$ is fully faithful, the right adjoint functor
(\ref{homok}) is fully faithful. 
If the functor $F: \mV \to \mW$ preserves small limits, the left adjoint of (\ref{homok}) identifies with the induced functor $ \Fun^R(\mU^\op, \mV) \to \Fun^R(\mU^\op,\mW)$, which preserves small limits.
The latter is fully faithful if $F$ is fully faithful.

\end{remark}

\subsubsection{Tensors and cotensors}

\begin{definition}Let $\mV$ be a presentably monoidal category, $\mC$ a $\mV$-enriched category and $X \in \mC, V \in \mV.$	 
	
\begin{enumerate}
	
\item The tensor of $V$ and $X$ in $\mC$ is the object $V \ot X \in \mC $ such that there is a morphism
$V \to \Mor_\mC(X, V \ot X) $ in $\mV$ that induces for every $Y \in \mC$ an equivalence
$$ \Mor_\mC(V \ot X,Y) \to \Mor_\mV(V, \Mor_\mC(X,Y)). $$ 

\item The cotensor of $V$ and $X$ in $\mC$ is the object ${^V X} \in \mC $ that is the tensor of $V $ and $X$ in the opposite $\mV^\rev$-enriched category $\mC^\circ.$

\end{enumerate}
	
\end{definition}

\subsubsection{Initial and final objects}

\begin{definition}\label{initi}
Let $\mV$ be a presentably monoidal category and $\mC$ a $\mV$-enriched category.

\begin{enumerate}
\item An object $X $ of $\mC$ is initial if for every object $Y \in \mC$
the morphism object $\Mor_\mC(X,Y)$ is a final object in $\mV.$

\vspace{1mm}

\item An object $X $ of $\mC$ is final if it is initial in $\mC^\circ$ or equivalently if for every object $Y \in \mC$
the morphism object $\Mor_\mC(Y,X)$ is a final object in $\mV.$

\vspace{1mm}

\item An object $X $ of $\mC$ is a zero object if it is initial and final.

\end{enumerate}

\end{definition}

\subsubsection{Enriched adjunctions}

\begin{definition}Let $\mV$ be a presentably monoidal category.
A $\mV$-enriched functor $\mC \to \mD$ admits a left (right) adjoint if
it admits a left (right) adjoint in the 2-category $\mV \mathrm{-}\Cat.$

\end{definition}

\begin{notation}
Let $\mV$ be a presentably monoidal category and $\mC, \mD$ be $\mV$-enriched categories. Let $$ \mV \mathrm{-}\L\Fun(\mC,\mD) \subset \mV \mathrm{-}\Fun(\mC,\mD)$$ be the full subcategory of $\mV$-enriched functors $\mC \to \mD$ that admit a $\mV$-enriched right adjoint.

\end{notation}

The following is \cite[Remark 2.75.]{heine2024bienrichedinftycategories}:

\begin{proposition}

Let $\mV$ be a presentably monoidal category.

\begin{enumerate}

\item A $\mV$-enriched functor $\phi: \mC \to \mD$ admits a $\mV$-enriched right adjoint if and only if for every $\Y \in \mD$ the $\mV^\rev$-enriched functor
$\Mor_\mD(\phi(-),\Y): \mC^\circ \to \mV$ is representable.

\vspace{1mm}

\item A $\mV$-enriched functor $\phi: \mC \to \mD$ admits a $\mV$-enriched left adjoint if and only if the opposite $\mV^\rev$-enriched functor $\phi^\circ: \mC^\circ \to \mD^\circ$ admits a $\mV^\rev$-enriched right adjoint. By (1) this holds if and only if for every $\Y \in \mD$ the $\mV$-enriched functor
$\Mor_\mD(\Y,\phi(-)): \mC \to \mV$ is representable.

\end{enumerate}
\end{proposition}

The following is \cite[Lemma 2.76.]{heine2024bienrichedinftycategories}:

\begin{proposition}\label{adj}\label{adj2}Let $\mV$ be a presentably monoidal category.

\begin{enumerate}
 
\item A $\mV$-enriched functor $\mC \to \mD$ admits a right adjoint if and only if it preserves tensors and the underlying functor admits a right adjoint.

\item A $\mV$-enriched functor $\mC \to \mD$ admits a left adjoint if and only if it preserves cotensors and the underlying functor admits a left adjoint.

\end{enumerate}

\end{proposition}

\subsubsection{Presentable enriched categories}

\begin{notation}
Let $\Pr^L \subset \widehat{\Cat}$ be the subcategory of presentable categories and left adjoint functors.

\end{notation}

\begin{remark}
The category $\Pr^L$ carries a canonical closed symmetric monoidal structure such that the inclusion $$\Pr^L \subset \widehat{\Cat}$$ is lax symmetric monoidal, where $\widehat{\Cat}$ carries the cartesian structure \cite[Proposition 4.8.1.15.]{lurie.higheralgebra}.

We write $\ot $ for the tensor product of $\Pr^L.$

\end{remark}

\begin{definition}\label{present} Let $\mV$ be a presentably monoidal category.
A $\mV$-enriched category $\mC$ is presentable if it admits tensors
and the underlying category $\iota(\mC)$ is presentable.
\end{definition}




\begin{notation}Let $\mV$ be a presentably monoidal category.
Let $$\mV \mathrm{-}\Pr^L \subset \mV \mathrm{-}\Cat$$ 
be the subcategory of presentable $\mV$-enriched categories and left adjoint $\mV$-enriched functors.

\end{notation}

The following is \cite[Theorem 1.2.]{Heine2023AnEB}:

\begin{theorem}Let $\mV$ be a presentably monoidal category.
There is a canonical equivalence $${_\mV\Mod(\Pr^L)} \simeq \mV \mathrm{-}\Pr^L,$$ where the left-hand side is the 2-category of left $\mV$-modules in $\Pr^L$ and $\mV$-linear left adjoint functors.

\end{theorem}

The following is \cite[Corollary 3.95.]{heine2024bienrichedinftycategories}:

\begin{proposition}\label{adjuiti}Let $\mV, \mW$ be presentably monoidal categories.
There is a canonical equivalence $${_\mV\Mod(\Pr^L)_\mW} \simeq {_{\mV \ot \mW^\rev}\Mod(\Pr^L)}.$$

\end{proposition}

\begin{corollary}

Let $\mV, \mW$ be presentably monoidal categories.
There is a canonical equivalence $${_\mV\Mod(\Pr^L)_\mW} \subset {\mV \ot \mW^\rev} \mathrm{-}\Pr^L.$$

\end{corollary}

\vspace{1mm}

The following is \cite[Lemma 2.79.]{heine2024bienrichedinftycategories}:

\begin{proposition}\label{adjunc}Let $\mV$ be a presentably monoidal category.
There is a canonical equivalence $$\mV \mathrm{-}\Pr^L \simeq (\mV \mathrm{-}\Pr^R)^\op$$ sending left to right adjoints.

\end{proposition}

\subsubsection{Enriched slice categories}

Next we introduce enriched slice categories.


For the next notation we use that the presentable 2-category $\mV \mathrm{-}\Cat$ admits cotensors.

\begin{notation}\label{slices} Let $\mV$ be a presentably monoidal category
whose tensor unit is final.
Let $\mC$ be a $\mV$-enriched category and $X \in \mC.$

Let $\mC_{/X} $ be the fiber over $X$ of the $\mV$-enriched functor $\mC^{\bD^1} \to \mC^{\{1\}}$ evaluating at $1$
in the presentable 2-category $\mV \mathrm{-}\Cat.$
There is a forgetful $\mV$-enriched functor $\mC_{X/} \to \mC^{\bD^1} \to \mC^{\{0\}}$ evaluating at $0.$	
	
\end{notation}

The following is \cite[Remark 2.4.25.]{gepner2025oriented}:

\begin{remark}\label{bien} Let $\mV$ be a presentably monoidal category
whose tensor unit is final.
Let $\mC$ be a $\mV$-enriched category and $X \in \mC.$
Let $\alpha: X \to Y, \beta: X \to Z$ be morphisms in $\mC.$
There is a canonical equivalence 	
$$ \Mor_{\mC_{X/}}(Y,Z) \simeq \{\beta\} \times_{\Mor_{\mC}(X,Z)} \Mor_{\mC}(Y,Z).$$

\end{remark}

The following is \cite[Proposition 3.7.1.]{gepner2025oriented}:

\begin{lemma}\label{left adjoint slice} Let $\mV$ be a presentably monoidal category whose tensor unit is final.
Let $\mC$ be a $\mV$-enriched category and $X \in \mC.$
Every morphism $X \to Y$ in $\mC$ gives rise to a $\mV$-enriched functor
$ \mC_{Y/}  \to \mC_{X/} $ over $\mC.$
If $\mC$ admits conical pushouts, the latter admits a left adjoint.

\end{lemma}

\subsection{$\infty$-categories}

Next we recall the basic theory of $\infty$-categories following \cite[\S 3]{gepner2025oriented}. 


	


\begin{definition}
For every $\n \geq 0$ we inductively define the presentable cartesian closed category $\n\Cat$ of small (univalent) $\n$-categories by setting:
$$0\Cat :=\infty\Grp,\ (\n+1)\Cat:= \n\Cat\mathrm{-}\Cat.$$
	
\end{definition}

\begin{notation}
For every $\n \geq 0$ we inductively define colocalizations
$\n\Cat \rightleftarrows (\n+1)\Cat: \iota_\n,$
where both adjoints  preserve finite products and filtered colimits.
Let
$$0\Cat= \infty\Grp \rightleftarrows 1\Cat = \infty\Grp\mathrm{-}\Cat : \iota_0 $$ be the unique colocalization whose left adjoint preserves finite products. Let $$(\n+1)\Cat= \n\Cat\mathrm{-}\Cat \rightleftarrows (\n+2)\Cat= (\n+1)\Cat \mathrm{-}\Cat:\iota_{\n+1}:= (\iota_\n)_!. $$

\end{notation}

\begin{definition}The presentable cartesian closed category $\infty\Cat$ of small (univalent) $\infty$-categories is the limit
$$\infty\Cat:= \lim(...\xrightarrow{\iota_{\n}} \n\Cat \xrightarrow{\iota_{\n-1}} ... \xrightarrow{\iota_0} 0 \Cat) $$
of presentable cartesian closed categories and right adjoint functors. 
	
\end{definition}

Since forming category objects preserves limits, we obtain the following:

\begin{remark}\label{fix}
	
There is a canonical equivalence
$ \infty\Cat \simeq \infty\Cat\mathrm{-}\Cat. $
	
\end{remark}

\begin{notation}\label{mormoro}

Let $\Mor: \infty\Cat_{\partial\bD^1/} \to \infty\Cat$
be the functor $$\infty\Cat_{\partial\bD^1/} \simeq \Cat_{\partial\bD^1/} \times_{\Cat} \infty\Cat\mathrm{-}\Cat \to \infty\Cat $$ sending $(\mC,X,Y)$ to $\Mor_\mC(X,Y).$

\end{notation}

\begin{remark}\label{homfil}
The functor $\Mor: \infty\Cat_{\partial\bD^1/} \to \infty\Cat$ preserves small limits and small filtered colimits.

\end{remark}

\begin{notation}
For every $0 \leq \n \leq \m$ the left adjoint embeddings $\n\Cat \leftrightarrows \m\Cat$ preserve small limits and thus induce a left adjoint embedding $\n\Cat \leftrightarrows \infty\Cat: \iota_\n$ that preserves small limits and so admits a left adjoint $\tau_\n: \infty\Cat \to \n\Cat$ by presentability.

\end{notation}

By the limit definition of $\infty\Cat$ we have the following filtration \cite[Lemma 2.3.10.]{gepner2025oriented}:

\begin{lemma}\label{decom}
Let $\mC$ be an $\infty$-category.
The sequential diagram $$\iota_0(\mC) \to ... \to \iota_\n(\mC) \to \iota_{\n+1}(\mC) \to ... \to \mC$$ exhibits $\mC$ as the colimit in $\infty\Cat$ of the diagram $\iota_0(\mC) \to ... \to \iota_\n(\mC) \to \iota_{\n+1}(\mC) \to ....$
\end{lemma}

\begin{notation}
	
Let $\partial\bD^1:=S^0:=*\coprod*$ be the set with two elements.	
	
\end{notation}

Next we construct the Gray tensor product.


\begin{notation}

By \cite[Definition 3.5.19.]{gepner2025oriented} for every $\n \geq 0 $ there is an $\n$-category $\cube^\n$,
the oriented $\n$-cube, whose 1-truncation is the $\n$-fold product $(\bD^1)^{\times \n}.$

Let $$\cube \subset \infty\Cat $$ be the full subcategory of oriented cubes.
	
\end{notation}

\vspace{1mm}
By \cite[Remark 3.5.18.]{gepner2025oriented} the category $\cube$ carries a monoidal structure $\boxtimes$ whose tensor unit is $\cube^0$ and such that for every $\n,\m \geq 0$ there is an equivalence $\cube^\n \boxtimes \cube^\m \simeq \cube^{\n+\m}.$

\vspace{1mm}

The next is \cite[Theorem 3.14., Example 3.16.]{campion2023gray}
(see also \cite[Corollary 3.6.2.]{gepner2025oriented}):

\begin{theorem}\label{dense}
	
There is a unique presentably monoidal structure on $\infty\Cat$ such that the embedding
$$\cube \subset \infty\Cat$$ is monoidal.

\end{theorem}

We call the monoidal structure on $\infty\Cat$ of Theorem \ref{dense}
the Gray monoidal structure and write $\boxtimes$ for the Gray tensor product.

\begin{notation}

Since the Gray tensor product defines a presentably monoidal structure on $\infty\Cat$, it is biclosed: for every $\infty$-category $\mC$ the functor $$ \mC \boxtimes (-): \infty\Cat \to \infty\Cat$$ admits a right adjoint $\Fun^\lax(\mC,-): \infty\Cat \to \infty\Cat$ and the functor $$ (-) \boxtimes \mC : \infty\Cat \to \infty\Cat$$ admits a right adjoint $ \Fun^\oplax(\mC,-): \infty\Cat \to \infty\Cat.$
\end{notation}

\begin{definition}Let $\F,\G:\mC \to \mD $ be functors of $\infty$-categories.
An (op)lax natural transformation $\F \to \G$ is a morphism in $\Fun^{(\op)\lax}(\mC,\mD)$.

\end{definition}

\begin{remark}\label{lao}

If $\mC$ is an $\n$-category and $\mD$ an $\m$-category for $\n,\m \geq 0$,
then $\mC \boxtimes \mD$ is an $\n+\m$-category.
This holds since $\n\Cat$ is closed under small colimits in $\infty\Cat$ and $\n\Cat$ is generated under small colimits by the oriented $\ell$-cubes for $1 \leq \ell \leq \n$. 
\end{remark}

The following is \cite[Remark 3.12.7.]{gepner2025oriented}:

\begin{remark}\label{grayspace}

The full subcategory $\infty\Grp \subset \infty\Cat$ is closed under the Gray tensor product and the restricted Gray monoidal structure on $\infty\Grp$ is the cartesian structure.
The left adjoint $\tau_0: \infty\Cat \to \infty\Grp$ of the resulting monoidal embedding $\infty\Grp \subset \infty\Cat$ is monoidal.

\end{remark}


The next is \cite[Proposition 3.8.8.]{gepner2025oriented}:

\begin{proposition}\label{dua}
	
There are monoidal involutions
$$(-)^\op, (-)^\co: (\infty\Cat, \boxtimes^\rev) \simeq (\infty\Cat, \boxtimes),$$
which reverse the odd dimensional morphisms, even dimensional morphisms, respectively.
	
\end{proposition}

\begin{corollary}\label{grayhoms}
Let $\mC,\mD \in \infty\Cat$.
There are canonical equivalences $$ \Fun^\oplax(\mC,\mD)^\op \simeq \Fun^\lax(\mC^\op,\mD^\op), \ \Fun^\oplax(\mC,\mD)^\co \simeq \Fun^\lax(\mC^\co,\mD^\co).$$

\end{corollary}

Next we consider a variant of the Gray tensor product for $\infty$-categories with distinguished object that plays the role of the smash product in homotopy theory.

\begin{definition}
The category $\infty\Cat_*$ of small $\infty$-categories with distinguished object is the full subcategory of $\Fun(\bD^1, \infty\Cat)$ spanned by the functors whose source is final.
	
\end{definition}

\begin{definition}

The Gray smash monoidal structure on $\infty\Cat_{*}$,
denoted by $\wedge$, is the smash monoidal structure of \cite[Lemma 3.69.]{gepner2025oriented}
applied to $(\infty\Cat, \boxtimes).$
For every $X, Y \in \infty\Cat_*$ the Gray smash product $X \wedge Y$ is the cofiber of the canonical functor $X \vee Y \to X \boxtimes Y. $

\end{definition}

\begin{notation}
Let $\mC,\mD \in \infty\Cat_*.$
Let $ \Fun_*^{(\op)\lax}(\mC,\mD) $ be the fiber of the functor $$ \Fun^{(\op)\lax}(\mC,\mD) \to \Fun^{(\op)\lax}(\ast,\mD)\simeq \mD $$ over the base point.

\end{notation}

\cite[Lemma 2.6.11.]{gepner2025oriented} gives the following:

\begin{lemma}
Let $\mC \in \infty\Cat_*$. 
\begin{enumerate}
\item The functor $$ \mC \wedge (-): \infty\Cat_* \to \infty\Cat_*$$ is left adjoint to the functor $\Fun_*^\lax(\mC,-): \infty\Cat_* \to \infty\Cat_*$. 

\vspace{1mm}
\item The functor $$ (-) \wedge \mC : \infty\Cat_* \to \infty\Cat_*$$ is left adjoint to the functor $\Fun_*^\oplax(\mC,-): \infty\Cat_* \to \infty\Cat_*.$

\end{enumerate}

\end{lemma}

By functoriality of the smash monoidal structure, \cref{dua} gives the following:

\begin{corollary}\label{dua2}
	
The monoidal involutions $$(-)^\op, (-)^\co: (\infty\Cat, \boxtimes)^\rev \simeq (\infty\Cat, \boxtimes)$$
of \cref{dua} induce monoidal involutions
$$(-)^\op, (-)^\co: (\infty\Cat_*, \wedge)^\rev \simeq (\infty\Cat_*, \wedge).$$
	
\end{corollary}




Next we define (reduced) categorical suspension and categorical spheres.

\vspace{1mm}

By \cite[Definition 3.4.2.]{gepner2025oriented} there is an $\infty$-category satisfying the following:  

\begin{notation}\label{suspi}

The functor $\Mor: \infty\Cat_{\partial\bD^1/} \to \infty\Cat$
of (\cref{mormoro}) admits a left adjoint $S: \infty\Cat \to \infty\Cat_{\partial\bD^1/}$, the categorical suspension.

Let $\mC$ be an $\infty$-category. The $\infty$-category $S(\mC)$ has two objects $0,1$ and morphism $\infty$-categories: $$\Mor_{S(\mC)}(1,0)\simeq \emptyset, \ \Mor_{S(\mC)}(0,1)\simeq \mC, \ \Mor_{S(\mC)}(0,0)\simeq \Mor_{S(\mC)}(1,1) \simeq *. $$ 

	
\end{notation}

	




	


	
	

The next proposition is \cite[Proposition 3.16.3.]{gepner2025oriented}:

\begin{proposition}\label{thas2}

For every $\infty$-category $\mC$ there is a canonical equivalence
$$ S(\mC) \simeq  \mC \boxtimes \bD^1 +_{ \mC \boxtimes \partial\bD^1} \partial\bD^1.$$

\end{proposition}

\begin{notation}
Let $\mC, \mD$ be $\infty$-categories and $X,Y \in \mC, X',Y' \in \mD.$
Let $$ \Fun_{\partial\bD^1/}^\oplax((\mC,X,Y), (\mD,X',Y'))$$ be the fiber over $(X',Y') \in \mD \times \mD$ of the following functor evaluating at $(X,Y) \in \mC \times \mC$:
$$ \Fun^\oplax(\mC,\mD) \to \mD \times \mD.$$ 

\end{notation}




We obtain the following immediate corollary:

\begin{corollary}\label{hom}
	
For every $\infty$-category $\mC$ and $X,Y \in \mC$ there is a canonical 
equivalence
$$ \Mor_\mC(X,Y) \simeq \{X \} \times_\mC \Fun^\oplax(\bD^1,\mC) \times_\mC \{Y\}.$$

Moreover for every $\infty$-category $\mB$ there is a canonical equivalence
$$ \Fun^\oplax(\mB, \Mor_\mC(X,Y)) \simeq \Fun^\oplax_{\partial\bD^1/}(S(\mB),\mC).$$
	
\end{corollary}




\begin{definition}Let $\n \geq 0$.
The $\n$-disk (or walking $\n$-morphism) is $\bD^\n:= S^{n}(*).$
	
\end{definition}

\begin{definition}Let $\n \geq 0$.
The boundary of the $\n$-disk is $\partial\bD^\n:= S^{n}(\emptyset).$
	
\end{definition}

\begin{remark}
The functor $\emptyset \subset *$ induces inclusions $\partial\bD^\n \subset \bD^\n$ for every $\n \geq 0.$
	
\end{remark}

\begin{example}
Then $\partial\bD^0=\ast, \partial\bD^1=S(\emptyset)=*\coprod*$ is the set with two elements.	
\end{example}
	
\begin{definition}
Let $\infty\Cat^\fin \subset \infty\Cat$ be the full subcategory generated by the disks under finite colimits. We call objects of $\infty\Cat^\fin$ finite $\infty$-categories.
	
\end{definition}	

	
	

	








\begin{notation}Let $\mC\in \infty\Cat_*$.
The suspension of $\mC$ is
$$\Sigma(\mC):= S(\mC) +_{\bD^1} \bD^0.$$ 	
	
\end{notation}

	
		
	

	
	


	
	



\begin{definition}
Let $\Omega: \infty\Cat_{*} \to \infty\Cat$
be the composition $$ \infty\Cat_{*} \to \infty\Cat_{\partial\bD^1/} \xrightarrow{\Mor}\infty\Cat, \ (\mC,X) \mapsto \Mor_\mC(X,X), $$
where the first functor is induced by the functor $\partial\bD^1 \to *.$

The functor $\Omega$ preserves the final object and so lifts to a functor $\Omega: \infty\Cat_{*} \to \infty\Cat_*.$

\end{definition}

\begin{remark}

There is an adjunction $ \Sigma : \infty\Cat_* \rightleftarrows \infty\Cat_*: \Omega.$

\end{remark}

\cref{homfil} implies the following:

\begin{remark}\label{endfil}
The functor $\Omega: \infty\Cat_{*} \to \infty\Cat_*$
preserves small filtered colimits.

\end{remark}

\begin{definition}\label{conn} Let $\n \geq 0$.
An $\infty$-category $\mC$ is $\n$-connected if the underlying space of $\mC$ is $\n$-connected and all morphism $\infty$-categories are $\n$-1-connected.
Every $\infty$-category is -1-connected. 

\end{definition}

\begin{remark}
Let $\n \geq 0$. By induction an $\infty$-category is $\n$-connected if it is $\n$+1-connected.

\end{remark}




\begin{notation}
For every $\n \geq 0$ let $ \Mon_{\bE_\n}(\infty\Cat)$
be the category of $\bE_\n$-monoids \cite[Definition 5.1.0.4.]{lurie.higheralgebra} in $\infty\Cat$.

\end{notation}

\begin{remark}

By \cite[Theorem 5.1.2.2]{lurie.higheralgebra} for every $\n \geq 1$ there is a canonical equivalence $$ \Mon_{\bE_\n}(\infty\Cat) \simeq \Mon_{\bE_{\n-1}}(\Mon(\infty\Cat)).$$	

\end{remark}

The next is \cite[Corollary 4.1.24.]{heine2026stable}:

\begin{corollary}\label{deloop} Let $\n \geq 1.$ There is an adjunction
$$B^\n: \Mon_{\bE_\n}(\infty\Cat) \rightleftarrows \infty\Cat_*: \Omega^\n$$
whose right adjoint lifts $\Omega^\n: \infty\Cat_* \to \infty\Cat_*$ and whose left adjoint preserves finite products and induces an equivalence to the full subcategory of $\n-1$-connected $\infty$-categories with distinguished object.

\end{corollary}

	


Next we define categorical spheres.

\begin{definition} 
Let $\n \geq 0$. The categorical $\n$-sphere $S^\n$ is the cofiber of the functor $\partial\bD^\n \subset \bD^\n$ in $\infty\Cat$.
\end{definition}

\begin{remark}By definition the categorical $\n$-sphere is an $\n$-category with distinguished object.

\end{remark}

\begin{notation}Let $\n \geq 0$. Let $ \Free_{\bE_\n}: \infty\Cat \to \Mon_{\bE_\n}(\infty\Cat)$ be the free functor.
 
\end{notation}

The next is \cite[Proposition 4.1.29.]{heine2026stable}:

\begin{proposition}\label{freesp} Let $\n \geq 0$. There is an equivalence of $\n$-categories with distinguished object $$S^\n \simeq B^\n \Free_{\bE_\n}(*).$$ In particular, there is an equivalence $S^1 \simeq B\bN$ of categories with distinguished object. 

\end{proposition}







The next is \cite[Lemma 4.1.31.]{heine2026stable}:

\begin{lemma}\label{smarem}\label{uuupo}
Let $\mC\in \infty\Cat_*$.
There is a canonical equivalence
$$\Sigma(\mC) \simeq \mC \wedge S^1. $$ 		

\end{lemma}

The next is \cite[Corollary 4.1.32.]{heine2026stable}:

\begin{corollary}\label{freespar} Let $\n \geq 0$.	
There is a canonical equivalence of $\n$-categories with distinguished object $$(-)_+\wedge (S^1)^{\wedge n} \simeq B^\n \circ \Free_{\bE_\n}.$$

\end{corollary}

\cref{freesp} and \cref{freespar} imply the following:

\begin{corollary}Let $\n \geq 0$. There is a canonical equivalence of $\n$-categories $(S^1)^{\wedge \n} \simeq S^\n$.

\end{corollary}	






%

	

\subsection{Oriented categories}

In this section we consider categories enriched in the Gray tensor product, which are also studied in \cite{gepner2025oriented}. 









\begin{definition}\emph{}
\begin{enumerate}
\item An oriented category is a category enriched in $(\infty\Cat,\boxtimes^\rev).$

\item An oriented functor is a functor enriched in $(\infty\Cat,\boxtimes^\rev).$

\item An antioriented category is a category enriched in $(\infty\Cat,\boxtimes).$

\item An antioriented functor is a functor enriched in $(\infty\Cat,\boxtimes).$

\item A bioriented category is a category enriched in $(\infty\Cat,\boxtimes) \ot (\infty\Cat,\boxtimes^\rev).$

\item A bioriented functor is a functor enriched in $(\infty\Cat,\boxtimes) \ot (\infty\Cat,\boxtimes^\rev).$

\end{enumerate}

\end{definition}

\begin{example}

The Gray monoidal structure on $\infty\Cat$ is biclosed and so exhibits $\infty\Cat$ as enriched in $(\infty\Cat,\boxtimes) \ot (\infty\Cat,\boxtimes^\rev).$ This way we see
$\infty\Cat$ as a large bioriented category, which we denote by the same name.

\end{example}	

\begin{notation}\emph{}

\begin{itemize}

\item Let $$\Cat\boxtimes:= (\infty\Cat,\boxtimes^\rev)\mathrm{-}\Cat $$
be the 2-category of oriented categories.

\item Let $$\boxtimes\Cat:= (\infty\Cat,\boxtimes)\mathrm{-}\Cat$$
be the 2-category of antioriented categories.

\item Let $$\boxtimes\Cat\boxtimes:= (\infty\Cat,\boxtimes) \ot (\infty\Cat,\boxtimes^\rev) \mathrm{-}\Cat$$
be the 2-category of bioriented categories.

\end{itemize}
\end{notation}

\begin{remark}
By \cref{indubi} the unique left adjoint monoidal embedding
$(\infty\Grp, \times) \to (\infty\Cat, \boxtimes), $ which preserves small limits, induces left adjoint monoidal embeddings
$$ (\infty\Cat, \boxtimes) \ot (\infty\Grp, \times) \to (\infty\Cat, \boxtimes) \ot (\infty\Cat, \boxtimes^\rev),$$$$ (\infty\Grp, \times) \ot (\infty\Cat, \boxtimes^\rev) \to (\infty\Cat, \boxtimes) \ot (\infty\Cat, \boxtimes^\rev),$$
which preserve small limits.
The latter induce left adjoint embeddings of 2-categories
$$ \Cat\boxtimes \subset {\boxtimes\Cat\boxtimes}, \boxtimes\Cat \subset {\boxtimes\Cat\boxtimes}, $$
which preserve small limits.

\end{remark}

\begin{notation}\emph{}

\begin{itemize}
\item Let $\mC$ be an oriented category and $X,Y \in \mC$.
We write $\R\Mor_\mC(X,Y)$ for the $\infty$-category of morphisms $X \to Y$ in $\mC$.

\item Let $\mC$ be an antioriented category and $X,Y \in \mC$.
We write $\L\Mor_\mC(X,Y)$ for the $\infty$-category of morphisms $X \to Y$ in $\mC$.

\item Let $\mC$ be a bioriented category and $X,Y \in \mC$.
We write $\L\Mor_\mC(X,Y)$ for the $\infty$-category of morphisms $X \to Y$ in the underlying antioriented category of $\mC$ and 
$\R\Mor_\mC(X,Y)$ for the $\infty$-category of morphisms $X \to Y$ in the underlying oriented category of $\mC$.
    
\end{itemize}
    
\end{notation}

We refer to adjunctions of oriented, antioriented, bioriented categories as oriented, antioriented, bioriented adjunctions.









We are mainly interested in the reduced version:

\begin{definition}\emph{}
\begin{enumerate}

\item A weakly reduced oriented category is a category enriched in $(\infty\Cat_*,\wedge^\rev).$

\item A weakly reduced oriented functor is a functor enriched in $(\infty\Cat_*,\wedge^\rev).$

\item A weakly reduced antioriented category is a category enriched in $(\infty\Cat_*,\wedge).$

\item A weakly reduced antioriented functor is a functor enriched in $(\infty\Cat_*,\wedge).$

\item A weakly reduced bioriented category is a category enriched in $(\infty\Cat_*,\wedge) \ot (\infty\Cat_*,\wedge^\rev).$

\item A weakly reduced bioriented functor is a functor enriched in $(\infty\Cat_*,\wedge) \ot (\infty\Cat_*,\wedge^\rev).$

\end{enumerate}

\end{definition}










		 

		
			

		
		
	

\begin{example}

The Gray smash monoidal structure on $\infty\Cat_*$ is biclosed and so exhibits $\infty\Cat_*$ as enriched in $(\infty\Cat_*,\wedge) \ot (\infty\Cat_*,\wedge^\rev).$ This way we see
$\infty\Cat_*$ as a large weakly reduced bioriented category, which we denote by the same name.

\end{example}	

\begin{notation}\emph{}

\begin{itemize}

\item Let $$\Cat\wedge:= (\infty\Cat,\wedge^\rev)\mathrm{-}\Cat$$
be the 2-category of weakly reduced oriented categories.	

\item Let $$\wedge\Cat:= (\infty\Cat,\wedge)\mathrm{-}\Cat$$
be the 2-category of weakly reduced antioriented categories.

\item Let $$\wedge\Cat\wedge:= (\infty\Cat,\smash) \ot (\infty\Cat,\wedge^\rev) \mathrm{-}\Cat$$
be the 2-category of weakly reduced bioriented categories.

\end{itemize}
\end{notation}

\begin{remark}

By \cref{indubi} the unique left adjoint monoidal embedding
$(\infty\Grp, \times) \to (\infty\Cat, \wedge), $ which preserves small limits, induces left adjoint monoidal embeddings
$$ (\infty\Cat, \wedge) \ot (\infty\Grp, \times) \to (\infty\Cat, \wedge) \ot (\infty\Cat, \wedge^\rev),$$$$ (\infty\Grp, \times) \ot (\infty\Cat, \wedge^\rev) \to (\infty\Cat, \wedge) \ot (\infty\Cat, \wedge^\rev),$$
which preserve small limits.
The latter induce left adjoint embeddings of 2-categories
$$ \Cat\wedge \subset {\wedge\Cat\wedge}, \wedge\Cat \subset {\wedge\Cat\wedge}, $$
which preserve small limits.

\end{remark}

We refer to adjunctions of weakly reduced oriented, antioriented, bioriented categories as weakly reduced oriented, antioriented, bioriented adjunctions.

\begin{notation}We fix the following notation:

\begin{itemize}
\item Let $\mC$ be an antioriented category, $X \in \mC$ and $K \in \infty\Cat.$
We refer to tensors of $K$ and $X$ by left tensors and write $K \ot X$
for the left tensor of $K$ and $X$.
We refer to cotensors of $K$ and $X$ by left cotensors and write ${^K X}$
for the left cotensor of $K$ and $X$.

\item Let $\mC$ be an oriented category, $X \in \mC$ and $K \in \infty\Cat.$
We refer to tensors of $K$ and $X$ by right tensors and write $X \ot K$
for the right tensor of $K$ and $X$.
We refer to cotensors of $K$ and $X$ by right cotensors and write ${X^K}$
for the right cotensor of $K$ and $X$.

\item Let $\mC$ be a bioriented category, $X \in \mC$ and $K \in \infty\Cat.$
We refer to tensors of $K \ot * $ and $X$ by left tensors and write $K \ot X$
for the left tensor of $K$ and $X$.
We refer to tensors of $* \ot K$ and $X$ by right tensors and write $X \ot K$
for the right tensor of $K$ and $X$.
We refer to cotensors of $K \ot *$ and $X$ by left cotensors and write ${^K X}$
for the left cotensor of $K$ and $X$.
We refer to cotensors of $* \ot K$ and $X$ by right cotensors and write ${X^K}$
for the right cotensor of $K$ and $X$.


\item Let $\mC$ be a weakly reduced antioriented category, $X \in \mC$ and $K \in \infty\Cat_*.$
We refer to tensors of $K$ and $X$ by left tensors and write $K \wedge X$
for the left tensor of $K$ and $X$.
We refer to cotensors of $K$ and $X$ by left cotensors and write ${^K X_*}$
for the left cotensor of $K$ and $X$.

\item Let $\mC$ be a weakly reduced oriented category, $X \in \mC$ and $K \in \infty\Cat_*.$
We refer to tensors of $K$ and $X$ by right tensors and write $X \wedge K$
for the right tensor of $K$ and $X$.
We refer to cotensors of $K$ and $X$ by right cotensors and write ${X^K_*}$ for the right cotensor of $K$ and $X$.

\item Let $\mC$ be a weakly reduced bioriented category, $X \in \mC$ and $K \in \infty\Cat_*.$
We refer to tensors of $K \ot * $ and $X$ by left tensors and write $K \wedge X$
for the left tensor of $K$ and $X$.
We refer to tensors of $* \ot K$ and $X$ by right tensors and write $X \wedge K$
for the right tensor of $K$ and $X$.
We refer to cotensors of $K \ot * $ and $X$ by left cotensors and write ${^K X_*}$
for the left cotensor of $K$ and $X$.
We refer to cotensors of $* \ot K$ and $X$ by right cotensors and write ${X^K_*}$ for the right cotensor of $K$ and $X$.




\end{itemize}
	
\end{notation}

An initial object, final object, zero object of a (weakly reduced) oriented, antioriented, bioriented category is an initial object, final object, zero object in the sense of \cref{initi}.

\begin{definition}
An oriented, antioriented, bioriented category is reduced if it admits a zero object.
An oriented, antioriented, bioriented functor is reduced if it preserves the zero object.
\end{definition}

\vspace{1mm}

The next is \cite[Proposition 3.8.]{gepner2025oriented}:

\begin{proposition}\label{reduu}
The forgetful functors $${\wedge\Cat} \to \boxtimes\Cat, \ {\Cat\wedge} \to \Cat\boxtimes, \ {\wedge\Cat \wedge} \to {\boxtimes\Cat \boxtimes}$$
restrict to equivalences between the full subcategories of weakly reduced oriented, antioriented, bioriented categories 
that admit an initial or final object and the subcategories of reduced oriented, antioriented, bioriented categories and reduced oriented, antioriented, bioriented functors, respectively.

\end{proposition}

In view of \cref{reduu} we identify reduced oriented, antioriented, bioriented categories with weakly reduced 
oriented, antioriented, bioriented categories that admit an initial or final object. 
Similarly, we identify reduced oriented, antioriented, bioriented functors between reduced oriented, antioriented, bioriented categories with weakly reduced oriented, antioriented, bioriented functors.

\begin{example}\label{Grayspaces} By \cref{grayspace} there is a monoidal localization
$\tau_0: \infty\Cat \rightleftarrows \infty\Grp$ that by \cref{adj} induces a bioriented localization
$$  \tau_0: \infty\Cat \rightleftarrows \tau_0^*(\infty\Grp).$$

\end{example}

\begin{definition}\emph{}
A weakly (reduced) oriented, antioriented, bioriented category is presentable if it is presentable in the sense of \cref{present}.

\end{definition}

\begin{notation}\label{presol}
Let $${\wedge\Pr^L}, \ {\wedge\Pr^R} \subset \wedge\widehat{\Cat}, \ {\Pr^L\wedge}, \ {\Pr^R\wedge} \subset \widehat{\Cat}\wedge, \ {\wedge\Pr^L\wedge}, \ {\wedge\Pr^R\wedge} \subset \wedge\widehat{\Cat}\wedge$$ be the respective subcategories of presentable reduced antioriented, oriented, bioriented categories and left (right) adjoint reduced antioriented, oriented, bioriented functors.

\end{notation}

\cref{adjunc} gives the following:

\begin{proposition}

There are canonical equivalences sending left to right adjoints: $$	(\wedge\Pr^L)^\op \simeq \wedge\Pr^R, $$$$ (\Pr^L\wedge)^\op \simeq \Pr^R\wedge, $$$$ (\wedge\Pr^L\wedge)^\op \simeq {\wedge\Pr^R\wedge}.$$
\end{proposition}

We often use (weakly) (reduced) oriented, antioriented and bioriented slice categories, which we define as enriched slice categories in the sense of \cref{slices}.

	

 






\subsubsection{Opposite and conjugate oriented categories}

\begin{notation}
The equivalences of \cref{dua} give rise to the following equivalences
$$(-)^\co:= (-)^\op_!: {\boxtimes\Cat} \simeq {\Cat\boxtimes}, $$$$ (-)^\co:= (-)^\op_!: {\Cat\boxtimes} \simeq {\boxtimes\Cat}, $$
$$ (-)^\co:= ((-)^\op, (-)^\op)_! : {\boxtimes\Cat\boxtimes} \simeq {\boxtimes\Cat\boxtimes},$$
where the first two equivalences are inverse to each other and the third equivalence is an involution.
\end{notation}

\begin{notation}
Let $$(-)^\circ: {\Cat\boxtimes} \simeq {\boxtimes\Cat}, $$$$ (-)^\circ: {\boxtimes\Cat} \simeq {\Cat\boxtimes} $$$$ (-)^\circ:  {\boxtimes\Cat\boxtimes} \simeq {\boxtimes\Cat\boxtimes}$$ be the opposite enriched category involutions.	

\end{notation}

\begin{notation}The equivalences of \cref{dua} give rise to the following equivalences

$$ (-)^\op:= (-)^\circ\circ (-)^\co_!: {\boxtimes\Cat} \simeq {\boxtimes\Cat}, $$$$ (-)^\op:= (-)^\circ\circ (-)^\co_!: {\Cat\boxtimes} \simeq {\Cat\boxtimes},$$
$$  (-)^\op :=(-)^\circ \circ ((-)^\co, (-)^\co)_!: {\boxtimes\Cat\boxtimes} \simeq {\boxtimes\Cat\boxtimes}, $$
where the first two equivalences are inverse to each other and the third equivalence is an involution.
\end{notation}

\begin{notation}

Moreover we set
$$ {(-)^{\co\op}}:=  {(-)^{\co}} \circ {(-)^{\op}} \simeq {(-)^{\op}} \circ {(-)^{\co}}:{\boxtimes\Cat} \simeq {\Cat\boxtimes}, {\Cat\boxtimes} \simeq {\boxtimes\Cat}, {\boxtimes\Cat\boxtimes} \simeq {\boxtimes\Cat\boxtimes},$$
$$ {(-)^{\cop}}:=  {(-)^{\co\op}} \circ {(-)^{\circ}} \simeq {(-)^{\circ}} \circ {(-)^{\co\op}} \simeq (-)^{\co\op}_!: {\boxtimes\Cat} \simeq {\boxtimes\Cat}, {\Cat\boxtimes} \simeq {\Cat\boxtimes},$$
$${^{\cop}(-)}:= ((-)^{\co\op}, \id)_!, \ (-)^\cop:= (\id, (-)^{\co\op})_!: {\boxtimes\Cat\boxtimes} \simeq {\boxtimes\Cat\boxtimes},$$
$${^{\cop}(-)}^\cop:=(-)^\cop \circ  {^{\cop}(-)} \simeq $$$$ {^{\cop}(-)}\circ (-)^\cop \simeq ((-)^{\co\op}, (-)^{\co\op})_! \simeq$$$$ {(-)^{\co\op}} \circ {(-)^{\circ}} \simeq {(-)^{\circ}} \circ {(-)^{\co\op}}: {\boxtimes\Cat\boxtimes} \simeq {\boxtimes\Cat\boxtimes}.$$

\end{notation}

The equivalences of \cref{dua} give rise to the following equivalences of bioriented categories:
\begin{corollary}\label{duae}There are canonical equivalences of bioriented categories:
$$(-)^\op: \infty\Cat^\co \simeq \infty\Cat, $$$$ (-)^\co: \infty\Cat^\op \simeq \infty\Cat^\circ, $$$$ (-)^{\co\op} : {^\cop\infty\Cat^{\cop}} \simeq \infty\Cat. $$	

\end{corollary}

Next we define oriented pushouts and oriented pullbacks following \cite[4.8.]{gepner2025oriented}.

\begin{definition}
Let $\mC$ be an oriented category.
The oriented pullback of a diagram $X \to Z\leftarrow Y$ in $\mC$ is a diagram
\[
\xymatrix{& W\ar[rd]\ar[ld] &\\
X\ar[rd] & \Longrightarrow & Y\ar[ld]\\
& Z &}
\]
in $\mC$ such that for all objects $T$ of $\mC$ the induced functor
\begin{equation}\label{indmapl}
\R\Mor_{\mC}(T,W)\to \R\Mor_{\mC}(T,X) \underset{\R\Mor_{\mC}(T,Z)}{\times}
\Fun^\oplax(\bD^1,\R\Mor_{\mC}(T,Z)) \underset{\R\Mor_{\mC}(T,Z)}{\times} \R\Mor_{\mC}(T,Y)
\end{equation}
is an equivalence in $\infty\Cat$.
We write $X\overset{\to}{\times}_Z Y$ or $ Y\overset{\leftarrow}{\times}_Z X$ for $W$ if it exists.

\end{definition}

\begin{definition}Let $\mC$ be an oriented category.
The oriented pushout of a diagram $X \leftarrow Z\to Y$ in $\mC$
is the oriented pullback in the oriented category $\mC^\op$ of the corresponding diagram.
We write $ X\overset{\to}{+}_Z Y $ or $Y\overset{\leftarrow}{+}_Z X$ for the oriented pushout.

\end{definition}

\begin{definition}Let $\mC$ be an antioriented category. 
\begin{enumerate}
\item The antioriented pullback of a diagram $X \to Z \leftarrow Y $ in $\mC$ denoted by $  X \overset{\bar{\to}}{\times}_Z Y $ or
$ Y \overset{\bar{\leftarrow}}{\times}_Z X $ is the oriented pullback of the corresponding diagram in the oriented category $\mC^{\co}$.	

\item The antioriented pushout of a diagram $X \leftarrow Z \to Y $ in $\mC$
denoted by $  X \overset{\bar{\to}}{+}_Z Y $ or $ Y \overset{\bar{\leftarrow}}{+}_Z X $ is the oriented pushout of the corresponding diagram in the oriented category $\mC^\co$.

\end{enumerate}	

\end{definition}

\begin{definition}
The (anti)oriented pushout and pullback in any bioriented category is the (anti)oriented pushout and pullback in the underlying (anti)oriented category.

\end{definition}

The next is \cite[Lemma 4.4.6.]{gepner2025oriented}:

\begin{lemma}\label{0desc}\label{adesc}
Let $\mC$ be an oriented category and $Z \to X, Z \to Y, X \to Z, Y \to Z$ morphisms in $\mC$.

\begin{enumerate}

\item If $\mC$ admits pushouts and right tensors with $\bD^1$, 
there is a canonical equivalence $$ X \overset{\to}{+}_Z Y \simeq X +_{\{0\}\otimes Z} (Z \ot \bD^1) +_{\{1\}\otimes Z} Y.$$	

\item If $\mC$ admits pullbacks and right cotensors with $\bD^1$, 
there is a canonical equivalence $$ X \overset{\to}{\times}_Z Y \simeq X \times_{Z^{\{0\}}} {Z^{\bD^1}} \times_{Z^{\{1\}}} Y.$$	
\end{enumerate}

\end{lemma}









\begin{definition}Let $\mC$ be a reduced oriented category.

\begin{enumerate}
\item Let $\phi: A \to B$ be a morphism in $\mC.$ The oriented left cofiber of $\phi$ is $ 0 \overset{\to}{+}_A B $.

\item Let $\phi: A \to B$ be a morphism in $\mC.$ The oriented right cofiber of $\phi$ is $ B \overset{\to}{+}_A 0 $.

\item Let $\phi: B \to A$ be a morphism in $\mC.$ 
The oriented left fiber of $\phi$ is $ 0 \overset{\to}{\times}_A B  $.

\item Let $\phi: B \to A$ be a morphism in $\mC.$ 
The oriented right fiber of $\phi$ is $ B \overset{\to}{\times}_A 0 $.

\end{enumerate} 
\end{definition}

\begin{definition}Let $\mC$ be a reduced antioriented category.

\begin{enumerate}
\item Let $\phi: A \to B$ be a morphism in $\mC.$ The antioriented left cofiber of $\phi$ is $ 0 \overset{\bar{\to}}{+}_A B $.

\item Let $\phi: A \to B$ be a morphism in $\mC.$ The antioriented right cofiber of $\phi$ is $B \overset{\bar{\to}}{+}_A 0$.

\item Let $\phi: B \to A$ be a morphism in $\mC.$ 
The antioriented left fiber of $\phi$ is $ 0 \overset{\bar{\to}}{\times}_A B $.

\item Let $\phi: B \to A$ be a morphism in $\mC.$ 
The antioriented right fiber of $\phi$ is $ B \overset{\bar{\to}}{\times}_A 0 $.

\end{enumerate} 
\end{definition}

\begin{definition}

A reduced (anti)oriented category admits (anti)oriented (co)fibers if it admits the (anti)oriented left and right (co)fiber of any morphism.

\end{definition}

The next is \cite[Corollary 4.3.12.]{gepner2025oriented}:

\begin{corollary}Let $\mC$ be a reduced oriented category and $X \in \mC$.
The oriented pushout $ {\overset{\overset{\rightarrow}{}}{0 +_C} 0}$ and the right tensor $X \wedge S^1$
both satisfy the same universal property.

\end{corollary}

The next is \cite[Corollary 4.3.13.]{gepner2025oriented}:

\begin{corollary}Let $\mC$ be a reduced antioriented category and $X \in \mC$.
The antioriented pushout $  {\overset{\overline{\rightarrow}}{0 +_C} 0}$ and the left tensor $S^1 \wedge X$ both satisfy the same universal property.
\end{corollary}






Next we refine (reduced) categorical suspension and (reduced) categorical antisuspension to a bioriented functor (\cref{reduc}) and prove that this refinement is functorial (\cref{impol}).
This is crucial to construct a bioriented category of categorical spectra in the next section.

The next is \cite[Theorem 4.4.4.]{heine2026stable}:

\begin{theorem}\label{reduc} Let $\mC$ be a reduced bioriented category.

\begin{enumerate}\item If $\mC$ admits suspensions, the antioriented functor $\Sigma: \mC \to \mC$ refines to a bioriented functor $$\mC^\cop \to \mC.$$

\item If $\mC$ admits endomorphisms, the antioriented functor $\Omega: \mC \to \mC$ refines to a bioriented functor $$\mC \to \mC^\cop.$$

\item If $\mC$ admits suspensions and endomorphisms, there is a bioriented adjunction
$$\Sigma: \mC^\cop \rightleftarrows \mC: \Omega.$$ 




\item If $\mC$ admits small colimits and left and right tensors, then $\Sigma: \mC \to \mC$ factors as
$$\mC \ot_{\infty\Cat_*} \Sigma: \mC \simeq \mC {\ot_{\infty\Cat_*} \infty\Cat_*}  \to \mC {\ot_{\infty\Cat_*} \infty\Cat_*} \simeq \mC.$$ 



\item There is an equivalence of reduced bioriented functors $$\Sigma \circ (-)^{\co\op}\simeq (-)^{\co\op} \circ \Sigma: {^\cop\infty\Cat_*} \to \infty\Cat_*.$$

\end{enumerate}	
\end{theorem}

The next is \cite[Proposition 4.4.5.]{heine2026stable}:

\begin{proposition}\label{impol} Let $\mC, \mD$ be reduced bioriented categories.

\begin{enumerate}
\item If $\mC,\mD$ admit endomorphisms, there is a transformation
$$\F \mapsto \F \circ \Omega_\mC \to \Omega_\mD \circ \F$$
of endofunctors of the category of reduced bioriented functors $\mC \to \mD.$



\item If $\mC,\mD$ admit suspensions, there is a transformation 
$$\F \mapsto \Sigma_\mD \circ \F \to \F \circ \Sigma_\mC $$
of endofunctors of the category of reduced bioriented functors $\mC \to \mD.$


\item If $\mC$ admits suspensions and $\mD$ admits endomorphisms,
there is a transformation 
$$\F \mapsto \F \to\Omega_\mD \circ \F \circ \Sigma_\mC $$
of endofunctors of the category of reduced bioriented functors $\mC \to \mD.$


\end{enumerate}

\end{proposition}








\section{A categorical Dold-Kan correspondence}

In this section we prove a categorical analogue of the Dold–Kan correspondence, which provides a combinatorial presentation of categorical homology in which the Street nerve plays the role of the singular complex.




\subsection{Categorical stability}

In the following we introduce stable oriented categories and their variants, which are higher-categorical analogues of classical stable categories. We follow \cite{heine2026stable}.

\begin{definition}
An oriented category $\mC$ is stable if the following conditions are satisfied:
	
\begin{enumerate}
\item $\mC$ admits a zero object.
		
\item $\mC$ admits oriented fibers and oriented cofibers.
 
\item The endomorphisms functor $\Omega: \mC \to \mC $ is an equivalence. 

\end{enumerate}
	

\end{definition}

	
	
	
	
	













\begin{definition}
A bioriented category is stable if it is reduced and its underlying oriented category is stable.	
	
\end{definition}

\begin{example}

Every stable category viewed as a bioriented category is stable.

\end{example}


	








We prove next that stability implies preadditivity.

\begin{definition}
An oriented, antioriented, bioriented category is preadditive if it is reduced, admits finite coproducts and finite products and for every $A,B \in \mC$ the canonical morphism $A \coprod B \to A \times B$ induced by the identity of $A,B$ and the zero morphisms, is an equivalence.	
	
\end{definition}

\begin{remark}

An oriented, antioriented, bioriented category is preadditive if and only if it is reduced, admits finite coproducts and finite products and the underlying category is preadditive.
    
\end{remark}


The next is \cite[Proposition 5.1.11.]{heine2026stable}:

\begin{proposition}\label{preadd}
Every stable oriented category is preadditive.	
\end{proposition}

\begin{notation}
Let $\Pr^L \subset \widehat{\Cat}$ be the subcategory of presentable categories and left adjoint functors.	

\end{notation}
\begin{notation}
Let $\Pr^L_{\mathrm{preadd}} \subset \Pr^L$ be the full subcategory of preadditive presentable categories.

\end{notation}

By \cite[Theorem 4.6.]{gepner2016universality} there is a symmetric monoidal localization \begin{equation}\label{loo}\Mon_{\bE_\infty}:\Pr^L \rightleftarrows \Pr^L_{\mathrm{preadd}},\end{equation} where $\Mon_{\bE_\infty}$ assigns to any presentable category $\mD$ the category of commutative monoid objects in $\mD$ and the unit is the free functor.
The localization (\ref{loo}) gives rise to a localization on associative algebras whose  local objects are precisely the preadditive presentably monoidal categories:  \begin{equation}\label{adjass}
\Mon_{\bE_\infty}:\Alg(\Pr^L) \rightleftarrows \Alg(\Pr^L_{\mathrm{preadd}}).\end{equation}

Moreover it gives rise to localizations, where the subscript preadd refers to the full subcategories of preadditive presentable reduced oriented, bioriented categories:
$$ \Mon_{\bE_\infty}: \wedge \Pr^L\rightleftarrows \wedge \Pr^L_{\mathrm{preadd}}, \  \Mon_{\bE_\infty}: \wedge \Pr^L\wedge \rightleftarrows \wedge {\Pr^L_{\mathrm{preadd}}\wedge} .$$


The next is \cite[Corollary 5.1.15.]{heine2026stable}:

\begin{corollary}\label{hhoo}\emph{ }
\begin{enumerate}

\item Let $\mC$ be a presentable stable oriented category.
The right adjoint oriented forgetful functor $ \Mon_{\bE_\infty}(\mC) \to \mC $ is an equivalence.
\vspace{1mm}


\vspace{1mm}
\item Let $\mC$ be a presentable stable bioriented category.
The right adjoint bioriented forgetful functor $ \Mon_{\bE_\infty}(\mC) \to \mC$ is an equivalence.
	
\end{enumerate}
\end{corollary}




	



\subsection{Categorical spectra}

Next we define spectrum objects in bioriented categories. The prime example are spectrum objects in the bioriented category $\infty\Cat$, which are known as categorical spectra. 


\begin{definition}Let $\mC$ be a reduced bioriented category that admits endomorphisms.
The bioriented category $\Sp(\mC)$ of spectrum objects in $\mC$ is the limit of the following  diagram of reduced bioriented categories: 
$$...\xrightarrow{\Omega} \mC^\cop \xrightarrow{\Omega} \mC  \xrightarrow{\Omega} \mC^\cop \xrightarrow{\Omega} \mC.$$

\end{definition}

	


\begin{example}
	
The bioriented category $\Sp(\infty\Grp)$ is the usual category of spectra viewed as a bioriented category.	

\end{example}

\begin{definition}\label{catspe}
The bioriented category of categorical spectra is $$\Cat\Sp := \Sp(\infty\Cat_*).$$

\end{definition}

\begin{remark}\label{nnm}
Let $\mC$ be a reduced presentable bioriented category. Then the bioriented functor $\Omega$ admits a left adjoint.
This implies that the bioriented category $\Sp(\mC)$ is presentable.
This follows from the fact that the subcategory $\wedge\Pr^R\wedge $ of presentable bioriented categories and right adjoint bioriented functors admits small limits preserved by the inclusion $\wedge\Pr^R\wedge \subset {\wedge\widehat{\Cat}\wedge}.$
\end{remark}

\begin{remark}\label{nnmt}
Via the equivalence $\wedge\Pr^L\wedge \simeq (\wedge\Pr^R\wedge)^\op$ and \cref{nnm} the bioriented category $\Sp(\mC)$ is also the colimit in $\wedge\Pr^L\wedge$ of the sequential diagram $$\mC \xrightarrow{\Sigma} \mC^\cop \xrightarrow{\Sigma} \mC \xrightarrow{\Sigma}....$$
\end{remark}

\begin{remark}\label{sptens}
Let $\mC$ be a presentable reduced bioriented category.
By \cref{reduc} the reduced bioriented functor $\Sigma: \mC \to \mC$ factors as
$$\mC \ot_{\infty\Cat_*} \Sigma: \mC \simeq \mC \ot_{\infty\Cat_*} \infty\Cat_* \to \mC \ot_{\infty\Cat_*} \infty\Cat_* \simeq \mC.$$ 
Hence there is a canonical equivalence of reduced bioriented categories $\Sp(\mC) \simeq \mC \ot_{\infty\Cat_*} \Cat\Sp.$ 
	
\end{remark}
	

\begin{notation}Let $\mC$ be a reduced bioriented category that admits endomorphisms.
The reduced bioriented functor of infinite endomorphisms $$\Omega^\infty: \Sp(\mC) \to \mC$$ is the canonical bioriented functor projecting to the value of the final object of the diagram.

\end{notation}

\begin{notation}\label{infsuspi}
Let $\mC$ be a reduced presentable bioriented category. The bioriented category $\Sp(\mC)$ is presentable by \cref{nnm} and the bioriented functor $\Omega^\infty: \Sp(\mC) \to \mC$ is accessible, preserves small limits, and left and right cotensors. Thus $\Omega^\infty: \Sp(\mC) \to \mC$ admits a bioriented left adjoint $\Sigma^\infty: \mC \to \Sp(\mC)$
by \cref{adj}, which we call infinite suspension.

\end{notation}

\begin{remark}Let $\mC$ be a reduced presentable bioriented category such that endomorphisms preserve small filtered colimits.
The bioriented functor $\Omega^\infty: \Sp(\mC) \to \mC $ preserves small filtered colimits so that the bioriented left adjoint $\Sigma^\infty: \mC \to  \Sp(\mC) $ preserves compact objects. 	
	
\end{remark}

The next is \cite[Lemma 5.2.14. ]{heine2026stable}:

\begin{lemma}Let $\mC$ be a bioriented category that admits a final object and the coproduct of any object with the final object.
The bioriented forgetful functor $\mC_* \to \mC$ admits a bioriented left adjoint $(-)_+$ that sends $X$ to $ X \coprod *.$ 
	
\end{lemma}

\begin{notation}Let $\mC$ be a presentable bioriented category. Then by \cref{infsuspi} the forgetful bioriented functor $\Sp(\mC) \to \mC_*$ admits a bioriented left adjoint $\Sigma^\infty.$
We define the bioriented adjunction $$\Sigma^\infty_+:= \Sigma^\infty \circ (-)_+: \mC \rightleftarrows  \mC_* \rightleftarrows \Sp(\mC):\Omega^\infty.$$
	
\end{notation}

\begin{remark}\label{comprestri}
Let $\mC$ be a reduced compactly generated bioriented category.
Since the bioriented functor $\Omega: \infty\Cat^\cop \to \infty\Cat $ preserves small filtered colimits by \cref{endfil}, the bioriented functor $\Omega: \mC^\cop \to \mC $ preserves small filtered colimits. Thus the bioriented left adjoint $\Sigma: \mC \to \mC^\cop $ preserves compact objects and so restricts to a functor $\Sigma: \mC^\omega \to (\mC^\omega)^\cop $ on compact objects. 

\end{remark}




	



The next is \cite[Proposition 5.2.19.]{heine2026stable}:

\begin{proposition}\label{rstab}
Let $\mC$ be a reduced bioriented category that admits endomorphisms.
The functor $\Omega: \Sp(\mC) \to \Sp(\mC)$ is equivalent to the evident induced functor on limits $$ \lim(... \xrightarrow{\Omega}\mC \xrightarrow{\Omega} \mC) \to \lim(... \xrightarrow{\Omega}\mC \xrightarrow{\Omega} \mC), \ (X_0, X_1,X_2, ...) \mapsto (X_{-1}, X_0,X_1, ...),$$
which is inverse to the shift functor $$ \lim(... \xrightarrow{\Omega}\mC \xrightarrow{\Omega} \mC) \to \lim(... \xrightarrow{\Omega}\mC \xrightarrow{\Omega} \mC), \ (X_0, X_1,X_2, ...) \mapsto (X_1, X_2,X_3, ...).$$

	
\end{proposition}

	

\begin{corollary}\label{stab}
The presentable bioriented category $\Cat\Sp$ is stable.

\end{corollary}

Next we consider an example of categorical spectrum.
The next example is \cite[Example 13.3.12.]{HigherQuasi}:

\begin{example}

Let $\mC$ be a symmetric monoidal category compatible with
geometric realizations.

Haugseng \cite{haugseng2017} constructs for every $\n \geq 0$ an $(\n+ 2)$-category $\mathrm{Morita}^{n}(\mC)$ of $\bE_{\n+1}$-algebras in $\mC$ whose objects are $\bE_{\n+1}$-algebras in $\mC$ and whose morphisms between
$\bE_{\n+1}$-algebras $A,B$ in $\mC$ are $\bE_{\n}$-algebras in $(A,B)$-bimodules in $\mC$, and so on.
The latter are non-unital after taking $\n+1$-fold morphism objects.

So
$\mathrm{Morita}^0(\mC)$
is the usual Morita 2-category of $\mC$.
We view $\mathrm{Morita}^\n(\mC)$ as $\infty$-category with distinguished object, where the distinguished object is the tensor unit of $\mC.$

By \cite[Corollary 5.51]{haugseng2017} 
for every $\n \geq 0$ there is a canonical equivalence $$ \Omega(\mathrm{Morita}^\n(\mC)) \simeq \mathrm{Morita}^{\n-1}(\mC)$$
of $\infty$-categories with distinguished object.
Hence there is a categorical spectrum
$$ \mathrm{morita}(\mC) := \{\mathrm{Morita}^\n(\mC)\}, $$
which we call the Morita spectrum of $\mC$. 

\end{example}

	

\begin{definition}\label{Brau} Let $R$ be an $\bE_\infty$-ring spectrum.
The Morita spectrum $\mathrm{morita}(R) $ of $R$ is the Morita spectrum of $R\mathrm{-}\Mod(\Sp(\infty\Grp)).$	
	
\end{definition}

\begin{definition}
	
A presentably monoidal reduced bioriented category is stable if its underlying presentable reduced bioriented category is
stable.
\end{definition}




The next is \cite[Corollary 5.2.36.]{heine2026stable}:

\begin{corollary}\label{unicity}

The category $\Cat\Sp$ carries a unique presentably monoidal structure such that the functor $\Sigma^\infty: \infty\Cat \to \Cat\Sp$ refines to a monoidal functor.

\end{corollary}

Next we consider prespectrum objects in any reduced bioriented category. 

\begin{notation}

Let $\mC$ be a reduced bioriented category that admits endomorphisms.
Let $$\mC^\Omega:= \mC^{\bD^1}\times_{\mC^{\{1\}}}\mC^\cop$$ be the pullback in $\wedge\Cat\wedge$ of evaluation at the target along $\Omega: \mC^\cop \to \mC.$

\end{notation}

\begin{remark}

There are canonical bioriented functors $\mC^\Omega \to \mC^{\bD^1} \to \mC^{\{0\}}$ and $ \mC^\Omega \to \mC^\cop$.

\end{remark}

\begin{definition}
The bioriented category of prespectrum objects in $\mC$ is the following limit in $\wedge\Cat\wedge$:
$$\Pre\Sp(\mC):= ... \times_{\mC^\cop} \mC^\Omega \times_{\mC}(\mC^\cop)^\Omega \times_{\mC^\cop} \mC^\Omega. $$

\end{definition}

\begin{remark}
A prespectrum in $\mC$ consists of a sequence $(X_0,X_1,X_2,...)$ of objects in $\mC$ and bonding morphisms $X_\n \to \Omega(X_{\n+1})$ in $\mC$ for every $\n \geq 0$.

\end{remark}

\begin{example}
	
The bioriented category $\Pre\Sp(\infty\Grp)$ is the usual category of prespectra viewed as a bioriented category.	
	
\end{example}

\begin{remark}
	
Since $\mC$ admits a zero object and endomorphisms, also $\mC^\Omega$ and so also the limit $\Pre\Sp(\mC)$ admit a zero object and endomorphisms.	
	
\end{remark}

\begin{definition}
The bioriented category of categorical prespectra is $$\Cat\Pre\Sp:=\Pre\Sp(\infty\Cat_*).$$

\end{definition}






The next is \cite[Lemma 6.1.12.]{heine2026stable}:

\begin{lemma}\label{embes}
Let $\mC$ be a reduced bioriented category that admits endomorphisms.
There is an embedding of bioriented categories 
$$\Sp(\mC) \subset \Pre\Sp(\mC)$$
whose essential image precisely consists of the prespectra whose bonding morphisms are equivalences.

\end{lemma}




\begin{definition}\label{hiul}
Let $\mC$ be a reduced bioriented category that admits endomorphisms and sequential colimits and $X \in \Pre\Sp(\mC)$. The associated spectrum of $X$ is the following spectrum $X'$:
for every $\n \geq 0$ let $$X'_\n := \colim(X_\n \to \Omega(X_{\n+1}) \to \Omega^2(X_{\n+2}) \to ...).$$
The bonding morphism $ X'_\n \to \Omega(X'_{\n+1})$ is the canonical equivalence
$$  \colim(X_\n \to \Omega(X_{\n+1}) \to \Omega^2(X_{\n+2}) \to ...) \simeq \colim(\Omega(X_{\n+1}) \to \Omega^2(X_{\n+2}) \to \Omega^3(X_{\n+3}) \to ...) $$$$\simeq \Omega(\colim(X_{\n+1} \to \Omega(X_{\n+2}) \to \Omega^2(X_{\n+3}) \to ...)).$$

\end{definition}

\begin{remark}

There is a morphism $X \to X'$ in $\Pre\Sp(\mC)$
such that for any $\n \geq 0$ the morphism $X_\n \to X'_\n$ is the morphism to the colimit using the canonical factorization $$X_\n \to \Omega^2(X_{\n+2}) \to \Omega^2(X'_{\n+2}) \simeq X'_\n $$ of the morphism $X_\n \to X'_\n.$

\end{remark}

The next is \cite[Proposition 6.1.16.]{heine2026stable}:

\begin{proposition}\label{spectri}

Let $\mC$ be a reduced bioriented category that admits endomorphisms and sequential colimits and let $X \in \Pre\Sp(\mC)$ and $Y \in \Sp(\mC).$
The following induced functor is an equivalence: $$\Mor_{\Pre\Sp(\mC)}(X',Y) \to \Mor_{\Pre\Sp(\mC)}(X,Y).$$

The bioriented embedding $\Sp(\mC)\subset \Pre\Sp(\mC)$ admits a left adjoint sending $X \in \Pre\Sp(\mC)$ to $X'$.

\end{proposition}

	
	
	
	
	


	

The following is \cite[Corollary 6.1.19.]{heine2026stable}:

\begin{corollary}\label{domonos}

Let $\mC, \mD$ be reduced bioriented categories that admit endomorphisms and sequential colimits and $\phi: \mC \to \mD$ a  reduced bioriented functor that preserves sequential colimits and endomorphisms.
The induced reduced bioriented functor $$\phi_*: \Pre\Sp(\mC) \to \Pre\Sp(\mD)$$ descends to a reduced bioriented functor $\phi_!: \Sp(\mC) \to \Sp(\mD)$, which is the restriction of $\phi_*.$

If $\mC,\mD$ admit suspensions and $\phi$ preserves suspensions, $\phi_*$ and $\phi_!$ preserve suspensions. In this case $\phi_*$ preserves infinite suspension prespectra and $\phi_!$
preserves infinite suspension spectra.

\end{corollary}

	

	

	

\begin{remark}

The bioriented localization $\tau_0: \infty\Cat \rightleftarrows \infty\Grp$ of \cref{Grayspaces}
gives rise to a bioriented localization $$ \Cat\Pre\Sp \rightleftarrows \Pre\Sp(\infty\Grp)$$
that descends to a bioriented localization $\tau_0: \Cat\Sp \rightleftarrows \Sp(\infty\Grp)$ using \cref{spectri},
where the left adjoint takes the spectrification of the degreewise classifying space.

\end{remark}

\subsection{Categorical homology}


We introduce a categorical analogue of excision.


\begin{definition}Let $\mC, \mD$ be reduced oriented categories. 
A reduced oriented functor $\F: \mC \to \mD$ is excisive if for every oriented pushout square
\[
\begin{tikzcd}
X \ar{r}{} \ar{d}[swap]{} & 0 \ar[double]{dl}{} \ar{d}{} \\
0 \ar{r}[swap]{} & Y
\end{tikzcd}
\]
in $\mC$ the induced oriented square 
\[
\begin{tikzcd}
F(X) \ar{r}{} \ar{d}[swap]{} & 0 \ar[double]{dl}{} \ar{d}{} \\
0 \ar{r}[swap]{} & F(Y)
\end{tikzcd}
\]
in $\mD$ is an oriented pullback square.

\end{definition}

\begin{definition}
Let $\mC, \mD$ be reduced antioriented categories. A reduced antioriented functor $\F: \mC \to \mD$ is antiexcisive if the reduced oriented functor $\F^\co: \mC^\co \to \mD^\co$ is excisive.
	
\end{definition}

\begin{definition}
Let $\mC, \mD$ be reduced bioriented categories.  
A reduced bioriented functor $\F: \mC \to \mD$ is excisive if its underlying oriented functor is excisive.



	
	
\end{definition}

Now we are ready to state the categorical Eilenberg-Steenrod axioms.

\begin{construction}
	
Let $\mD$ be a reduced bioriented category that admits endomorphisms, $H: \infty\Cat_* \to \mD$ a reduced excisive oriented functor and $X \in \infty\Cat_*.$ 
For every even $\ell \geq 0$ the canonical morphism
$$ H(S^\ell)\wedge X \to H(S^\ell \wedge X)$$ gives rise to a canonical morphism
$$ \Omega^\ell(H(S^\ell)\wedge X) \to \Omega^\ell(H(S^\ell \wedge X)) \simeq \Omega^\ell(H(X \wedge S^\ell)) \simeq H(X).$$
The family of such morphisms for even $\ell \geq 0$ induces a canonical morphism:
\begin{equation}\label{eqzz}
\Omega^\infty(H(S^\bullet)\wedge X) \simeq \colim(H(S^0) \wedge X \to \Omega(H(S^1)\wedge X) \to \Omega^2(H(S^2)\wedge X) \to...)  \simeq \end{equation}	$$ \colim(H(S^0) \wedge X \to \Omega^2(H(S^2)\wedge X) \to \Omega^4(H(S^4)\wedge X) \to ...) \to H(X),$$
where the first equivalence is by \cref{spectri} and the second equivalence is by cofinality.
	
\end{construction}

\begin{definition}Let $\mD$ be a reduced bioriented category that admits endomorphisms.
A categorical homology theory is a reduced oriented functor $H: \infty\Cat_* \to \mD$ 
that satisfies the following axioms:
\begin{enumerate}
\item $H$ preserves small filtered colimits.

\vspace{1mm}

\item $H$ is excisive.

\vspace{1mm}

\item $H$ is spherical, i.e. for every $X \in \infty\Cat_*$ the morphism (\ref{eqzz}) is an equivalence.

\end{enumerate}
\end{definition}

	



\begin{example}\label{homol} Let $\mD$ be a reduced presentable bioriented category and $\E$ a spectrum in $\mD$. The oriented functor $$\infty\Cat_* \to \mD, X \mapsto \Omega^\infty(\E \wedge X)$$
is a categorical homology theory since it is reduced and excisive by stability (\cref{stab}), preserves small filtered colimits, and the morphism (\ref{eqzz}) is an equivalence, where we use that $\Omega^\infty(\E \wedge S^\n)\simeq \E_\n$ compatible with the bonding maps. We call $ \Omega^\infty(\E \wedge (-))$ the categorical homology theory represented by $\E.$

\end{example}

\begin{definition}

A categorical spectrum $\mC$ is connective if for every $\n \geq 0$ the $\infty$-category
$\Omega^\infty(\Sigma^\n(\mC)) $ is $\n$-1-connected in the sense of \cref{conn}.

\end{definition}

\begin{notation}
Let $\Cat\Sp_{\geq0} \subset \Cat\Sp $ be the full subcategory of connective categorical spectra.	

\end{notation}

\begin{construction}
	
The left adjoint of (\ref{adjass}) sends the presentably monoidal category $(\infty\Cat,\wedge)$
to a presentably monoidal structure on $\Mon_{\bE_\infty}(\infty\Cat)$
such that the free functor $$\infty\Cat_* \to \Mon_{\bE_\infty}(\infty\Cat)$$
is monoidal. The free functor makes $ \Mon_{\bE_\infty}(\infty\Cat)$ a presentably monoidal category under $\infty\Cat$.

\end{construction}

The next is \cite[Lemma 5.3.7.]{heine2026stable}:

\begin{lemma}\label{deoo}
	
The endomorphisms functor $$\Omega: \Mon_{\bE_\infty}(\infty\Cat)^\cop \to \Mon_{\bE_\infty}(\infty\Cat) $$ admits a left adjoint reduced bioriented functor
$$B: \Mon_{\bE_\infty}(\infty\Cat) \to \Mon_{\bE_\infty}(\infty\Cat)^\cop$$
that preserves finite products and induces an equivalence. The essential image of $B$ precisely consists of the connected $\infty$-categories with distinguished object.
The underlying functor of $B$ is induced by the finite products preserving functor
$B: \Mon(\infty\Cat) \to \infty\Cat $ of \cref{deloop}.
	
\end{lemma}

\begin{example}
	
Let $\mC$ be a symmetric monoidal $\infty$-category. The categorical spectrum $$B^\infty(\mC):= (\mC, B\mC, B^2\mC, ...)$$ with the natural bonding maps is connective.
	
\end{example}


The next is \cite[Theorem 5.3.9.]{heine2026stable}:

\begin{theorem}\label{symsp}

There is a left adjoint monoidal functor $$ \Mon_{\bE_\infty}(\infty\Cat) \to \Mon_{\bE_\infty}(\Cat\Sp) \simeq \Cat\Sp $$
that sends a symmetric monoidal $\infty$-category $\mC$ to $B^\infty(\mC)$ and induces a monoidal equivalence
$$ \Mon_{\bE_\infty}(\infty\Cat) \simeq \Cat\Sp_{\geq0}.$$

\end{theorem}	

\begin{corollary}\label{domon}
Let $\mC$ be a reduced bioriented category that admits suspensions, endomorphisms and sequential colimits. Every reduced bioriented functor $\phi: \mC \to \Mon_{\bE_\infty}(\infty\Cat)$ that preserves suspensions, endomorphisms and sequential colimits also preserves oriented pushouts.

\end{corollary}

\begin{proof}By \cref{domonos} the reduced bioriented functor $\phi_* :  \Pre\Sp(\mC) \to \Pre\Sp(\Mon_{\bE_\infty}(\infty\Cat))$ 
restricts to a reduced bioriented functor $$\phi_!:  \Sp(\mC) \to \Sp(\Mon_{\bE_\infty}(\infty\Cat)) \simeq  \Mon_{\bE_\infty}(\Sp) \simeq \Sp,$$ where the last equivalence is by \cref{hhoo},
that preserves infinite suspension spectra.
By \cref{symsp} the infinite suspension spectrum bioriented functor of $  \Mon_{\bE_\infty}(\Sp)$ is the bioriented embedding
$$B^\infty:  \Mon_{\bE_\infty}(\Sp) \to \Sp(\Mon_{\bE_\infty}(\infty\Cat)) \simeq  \Mon_{\bE_\infty}(\Sp) \simeq \Sp, $$ which as a left adjoint bioriented embedding detects oriented pushouts.
In other words there is a canonical equivalence $$\phi_! \circ \Sigma^\infty \simeq B^\infty \circ \phi.$$
Hence $\phi$ preserves oriented pushouts because $ \phi_! \circ \Sigma^\infty$ preserves oriented pushouts as a left adjoint bioriented functor.

\end{proof}

\begin{definition}\emph{}
Let $\mC$ be a presentably monoidal category.
A rig in $\mC$ is an associative algebra in $ \Mon_{\bE_\infty}(\mC).$

\end{definition}

\begin{definition}

A rig is a rig in $\Set.$
A rig space is a rig in $\infty\Grp.$

\end{definition}

For the next definition we use that $\Cat\Sp$ carries a monoidal structure
(\cref{unicity}).

\begin{definition}A categorical rig spectrum is an associative algebra in $\Cat\Sp.$

\end{definition}




\begin{notation}Let $\mC$ be a presentably monoidal category.
Let $$ \Rig(\mC):= \Alg(\Mon_{\bE_\infty}(\mC)).$$ 

\end{notation}

\begin{example}\emph{}
\begin{itemize}

\item The natural numbers $(\bN,+,\bullet)$ are a rig.	

\item The free symmetric monoidal category $\coprod_{\n \geq 0} B\Sigma_\n$ generated by one object is a rig space.

\end{itemize}

\end{example}

\begin{example}
	
By adjunction (\ref{adjass}) the canonical symmetric monoidal localizations $$ \infty\Cat \rightleftarrows \infty\Grp, \ \pi_0: \infty\Grp \rightleftarrows \Set$$ give rise to monoidal localizations $$ \Mon_{\bE_\infty}(\infty\Cat) \rightleftarrows \Mon_{\bE_\infty}(\infty\Grp), \ \Mon_{\bE_\infty}(\infty\Grp) \rightleftarrows \Mon_{\bE_\infty}(\Set), $$
which give rise to localizations $$ \Rig(\infty\Cat) \rightleftarrows \Rig(\infty\Grp), \ \Rig(\infty\Grp) \rightleftarrows \Rig(\Set). $$

\end{example}

\begin{remark}

The left adjoint monoidal embedding $$B^\infty: \Mon_{\bE_\infty}(\infty\Cat) \simeq \Mon_{\bE_\infty}(\infty\Cat_*) \to \Mon_{\bE_\infty}(\Cat\Sp) \simeq \Cat\Sp$$
of \cref{symsp} induces a left adjoint embedding
$$B^\infty: \Rig(\infty\Cat) \to \Rig(\Cat\Sp).$$

\end{remark}

	
	


\begin{definition}\label{rigit}
The categorical Eilenberg-MacLane spectrum functor $H$ is the following composition of lax monoidal embeddings:
$$\Mon_{\bE_\infty}(\Set) \subset \Mon_{\bE_\infty}(\infty\Grp) \xrightarrow{B^\infty} \Cat\Sp,$$
which gives rise to an embedding
$$\Rig(\Set) \subset \Rig(\Cat\Sp).$$

\end{definition}



    






\begin{definition}Let $R$ be a commutative monoid.
Categorical $R$-homology, denoted by $ H(-; R)$, is the categorical homology theory associated to the spectrum $H(R).$	
For $R=\bN$ we skip $R.$
	
\end{definition}

\begin{remark}

By definition, categorical $R$-homology factors as the oriented functor
$H(R) \wedge (-): \infty\Cat_* \to {{H(R)}\mathrm{-}\Mod(\Cat\Sp)} $
followed by the oriented forgetful functor ${{H(R)}\mathrm{-}\Mod(\Cat\Sp)} \to \Cat\Sp $
followed by the oriented functor $\Omega^\infty: \Cat\Sp \to \infty\Cat_*.$

\end{remark}

\begin{example}\label{homosph} Let $R$ be a commutative monoid and $\n \geq 0.$ The categorical $R$-homology of $S^\n$ is 
$$ \Omega^\infty(H(R)\wedge  S^\n) \simeq \Omega^\infty(H(R)\wedge \Sigma^\n(S^0)) \simeq  \Omega^\infty(\Sigma^\n(H(R)\wedge S^0)) \simeq \Omega^\infty(\Sigma^\n(H(R))) \simeq B^\n(R).$$

\end{example}






\subsection{Strict symmetric monoidal $\infty$-categories}

\begin{notation}

Let $ \C\mon^{\mathrm{free},\fin} \subset \C\mon$ be the full subcategory of commutative monoids free on a finite set.

\end{notation}

\begin{remark}\label{rempuy}
The category $\C\mon$ is preadditive and the full subcategory $\C\mon^{\mathrm{free},\fin} \subset \C\mon$ is generated by $\bN$ under finite products. In particular, $$\C\mon^{\mathrm{free},\fin} \subset \C\mon$$ is closed under finite products and therefore is preadditive, too.

\end{remark}

\begin{remark}
Since $\bN$ is the tensor unit of the canonical closed symmetric monoidal structure on $\C\mon$,
the canonical closed symmetric monoidal structure on $\C\mon$ restricts to $\C\mon^{\mathrm{free},\fin}$ and every object of $\C\mon^{\mathrm{free},\fin} $ is dualizable in $ \C\mon$ with respect to the canonical closed symmetric monoidal structure. Hence taking the dual provides a symmetric monoidal equivalence 
$$(\C\mon^{\mathrm{free},\fin})^\op \simeq \C\mon^{\mathrm{free},\fin}.$$

\end{remark}

The category $ (\C\mon^{\mathrm{free},\fin})^\op $ is the Lawvere theory for commutative monoids.
So in view of \cref{rempuy} we make the following definition:

\begin{notation}

Let $\mC$ be a category that admits finite products.
The category of strict commutative monoid objects in $\mC$ is the full subcategory $$\C\mon(\mC) \subset \Fun(\C\mon^{\mathrm{free},\fin} ,\mC)$$ of functors $\C\mon^{\mathrm{free},\fin} \to \mC$ preserving finite products.

\end{notation}

\begin{remark}\label{Preaddd}
The category $\C\mon(\mC)$ is preadditive by \cite[Corollary 2.4.]{gepner2016universality}. 
\end{remark}

\begin{notation}

The forgetful functor $\C\mon(\mC) \to \mC$ is the restriction of the functor  evaluating at $\bN$:
$$ \Fun(\C\mon^{\mathrm{free},\fin},\mC) \to \mC.$$

\end{notation}

\begin{lemma}\label{sift1} Let $\mC$ be a category that admits finite products
and $\mK$ a collection of categories.

\begin{enumerate}
\item If $\mC$ admits $\mK$-indexed limits, also the category $ \C\mon(\mC)$ admits $\mK$-indexed limits and the forgetful functor $\C\mon(\mC) \to \mC$ preserves $\mK$-indexed limits. 

\item If $\mK$ consists of sifted categories and $\mC$ admits $\mK$-indexed colimits, also $ \C\mon(\mC)$ admits $\mK$-indexed colimits and the forgetful functor $\C\mon(\mC) \to \mC$ preserves $\mK$-indexed colimits. 

\end{enumerate}

\end{lemma}

\begin{proof}

(1) If $\mC$ admits $\mK$-indexed limits, the full subcategory $$ \C\mon(\mC)\subset \Fun(\C\mon^{\mathrm{free},\fin},\mC) $$ is closed under $\mK$-indexed limits since $\mK$-indexed limits commute with finite products.
This gives (1).

(2) If $\mK$ consists of sifted categories and $\mC$ admits $\mK$-indexed colimits, the full subcategory $$ \C\mon(\mC)\subset \Fun(\C\mon^{\mathrm{free},\fin},\mC) $$ is closed under $\mK$-indexed colimits since $\mK$-indexed colimits commute with finite products as every category of $\mK$ is sifted. This implies (2).

\end{proof}

\begin{notation}

Let $\mC$ be a category that admits finite products.
The category of $\bE_\infty$-monoid objects in $\mC$ is the full subcategory
$\Mon_{\bE_\infty}(\mC) \subset \Fun(\Fin_* ,\mC)$ of functors $X: \Fin_* \to \mC$ 
such that for every $\n \geq 0$ the following induced morphism is an equivalence: $$ X(\{1,...,\n\}_+) \to \prod_{\bi=1}^\n X(\{1\}_+) $$

\end{notation}

\begin{remark}\label{Preaddd2}
The category $\Mon_{\bE_\infty}(\mC)$ is preadditive by \cite[Proposition 2.3.]{gepner2016universality}. 
\end{remark}

\begin{remark}\label{sift2}

The same lemma as \cref{sift1} with the same proof holds for $\Mon_{\bE_\infty}$ instead of $\Cmon.$

\end{remark}

\begin{remark}\label{locquip}

Let $\mC$ be a $\kappa$-compactly generated category for some regular cardinal $\kappa.$
An object of $ \Fun(\C\mon^{\mathrm{free},\fin},\mC)$ belongs to $\C\mon(\mC)$
if and only if it is local with respect to the set of morphisms
$$\coprod_{\bi=1}^\n \C\mon^{\mathrm{free},\fin}(\bN,-) \ot Z  \to  \C\mon^{\mathrm{free},\fin}(\bN^{\times \n},-) \ot Z $$
for every $Z \in \mC^\kappa$ if and only if it is local with respect to the set of morphisms
$$\coprod_{\ell=1}^\bk \C\mon^{\mathrm{free},\fin}(\bN^{\times \n_\ell},-) \ot Z  \to  \C\mon^{\mathrm{free},\fin}(\bN^{\times{\n_1 + ... + \n_\bk}},-) \ot Z $$
for every $Z \in \mC^\kappa$ and $\bk, \n_1, ..., \n_\bk \geq 0$.
Thus $\C\mon(\mC) \subset  \Fun(\C\mon^{\mathrm{free},\fin},\mC)$ is an accessible localization.

Similarly, an object of $ \Fun(\Fin_*,\mC)$ belongs to $\Mon_{\bE_\infty}(\mC)$
if and only if it is local with respect to the set of morphisms
$$\coprod_{\bi=1}^\n \Fin_*(\{1\}_+) \ot Z  \to \Fin_*(\{1,...,\n\}_+) \ot Z $$
for every $Z \in \mC^\kappa$ if and only if it is local with respect to the set of morphisms
$$\coprod_{\ell=1}^\bk \Fin_*(\{1,..., \n_\ell\},-) \ot Z  \to  \Fin_*(\{1,...,\n_1 + ... + n_\bk\},-) \ot Z $$
for every $Z \in \mC^\kappa$ and $\bk, \n_1, ..., \n_\bk \geq 0$.
Thus $\Mon_{\bE_\infty}(\mC) \subset   \Fun(\Fin_*,\mC)$ is an accessible localization.

\end{remark}

\begin{notation}\label{heuty}
Let $\mC$ be a presentable category. By \cref{locquip} the full subcategory $$\C\mon(\mC) \subset \Fun(\C\mon^{\mathrm{free},\fin},\mC)$$ is an accessible localization.
The functor $ \Fun(\C\mon^{\mathrm{free},\fin},\mC) \to \mC$ evaluating at $\bN$ admits a left adjoint $\phi$ that sends an object $X $ of $\mC$ to $\C\mon^{\mathrm{free},\fin}(\bN,-) \times X.$ 
Consequently, the forgetful functor $ \C\mon(\mC) \to \mC$ admits a left adjoint $$\Sym: \mC \to \C\mon(\mC) $$ that factors as $\phi$ followed by the localization. 
We call $\Sym$ the free strict commutative monoid functor.

Since $\C\mon(\mC)$ is presentable, also the forgetful functor $ \C\mon(\mC) \to \mC_*$ admits a left adjoint $$\widetilde{\Sym}: \mC _* \to \C\mon(\mC).$$ 
We call $\widetilde{\Sym}$ the free reduced strict commutative monoid functor.

\end{notation}

\begin{remark}\label{naturalis}

If $\mC$ is presentably cartesian closed, the functor $(-)\times *: \infty\Grp \to \mC$ preserves finite products.
In this case $\phi(*)\simeq \C\mon^{\mathrm{free},\fin}(\bN,-) \times *$ preserves finite products and so is already local and so agrees with $\Sym(*).$ 
In particular, the image of $\Sym(*) \simeq \widetilde{\Sym}(*\coprod *)$ in $\mC$ under the forgetful functor is $$\C\mon^{\mathrm{free},\fin}(\bN,\bN) \times * \simeq \bN \times *.$$

\end{remark}

\begin{remark}\label{holuk}
The category $\Set_*$ of pointed sets is a symmetric monoidal category via the smash product (\cite[Lemma 3.69.]{gepner2025oriented}).
The smash monoidal structure restricts to $\Fin_*$. By \cite[Theorem 4.6.]{gepner2016universality} the free functor $\Set_* \to \C\mon$ is symmetric monoidal. Hence the symmetric monoidal structure on $\C\mon$ restricts to
$\C\mon^{\mathrm{free},\fin}$, the image of $\Fin_*$ under the free functor, and the symmetric monoidal free functor $\Set_* \to \C\mon$ restricts to a symmetric monoidal functor $\Fin_* \to \C\mon^{\mathrm{free},\fin}. $
\end{remark}

\begin{notation}\label{restrii}

For every category $\mC$ that admits finite products the restricted free functor $\Fin_* \to \C\mon^{\mathrm{free},\fin}$ induces a functor $ \Fun(\C\mon^{\mathrm{free},\fin}, \mC) \to \Fun(\Fin_*,\mC),$ which restricts to a functor \begin{equation}\label{forgeto}
\nu: \C\mon(\mC) \to \Mon_{\bE_\infty}(\mC).
\end{equation}

\end{notation}

\begin{remark}
Since the category $ (\C\mon^{\mathrm{free},\fin})^\op $ is the Lawvere theory for commutative monoids, for every $(1,1)$-category $\mC$ that admits finite products the forgetful functor 
(\ref{forgeto}) is an equivalence.

\end{remark}

\begin{lemma}\label{rerul}
	
Let $\mC$ be a category that admits finite products.

\begin{enumerate}
\item By \cref{sift1}, \cref{sift2} the categories $\C\mon(\mC), \Mon_{\bE_\infty}(\mC)$ admit finite products and the forgetful functors 
\begin{equation*}\label{foeti}
\C\mon(\mC) \to \mC, \ \Mon_{\bE_\infty}(\mC) \to \mC \end{equation*}
and (\ref{forgeto}) preserve finite products and finite coproducts.

\vspace{1mm}

\item If $\mC$ admits small sifted colimits, the functor (\ref{forgeto}) preserves small colimits.

\end{enumerate}
	
\end{lemma}

\begin{proof}

(1): By \cref{sift1}, \cref{sift2} the categories $\C\mon(\mC), \Mon_{\bE_\infty}(\mC)$ admit finite products and the conservative forgetful functors 
\begin{equation}\label{foeti}
\C\mon(\mC) \to \mC, \ \Mon_{\bE_\infty}(\mC) \to \mC 
\end{equation}
preserve finite products. So the functor (\ref{forgeto}) preserves finite products and so also finite coproducts because the categories
$ \C\mon(\mC), \Mon_{\bE_\infty}(\mC)$ are preadditive by \cref{Preaddd}, \cref{Preaddd2}.

(2): If $\mC$ admits small sifted colimits, the functor (\ref{forgeto}) preserves small sifted colimits because the categories $ \C\mon(\mC),$$  \Mon_{\bE_\infty}(\mC)$ admit small sifted colimits and the forgetful functors (\ref{foeti}) preserve small sifted colimits (\cref{sift1}, \cref{sift2}). So the functor (\ref{forgeto}) preserves small colimits since it preserves finite coproducts by (1).

\end{proof}

\begin{proposition}\label{laxcomp}
Let $\mC$ be a presentably monoidal category.
The categories $\C\mon(\mC), \Mon_{\bE_\infty}(\mC)$ carry presentably monoidal structures and the restriction $$\C\mon(\mC) \to \Mon_{\bE_\infty}(\mC)$$ of \cref{restrii} and the 
forgetful functors
$$ \C\mon(\mC) \to \mC, \ \Mon_{\bE_\infty}(\mC) \to \mC $$ are lax monoidal and admit a monoidal left adjoint. 

\end{proposition}

\begin{proof}

The monoidal structures on $\Fin_*$ and $ \C\mon^{\mathrm{free},\fin}$ of \cref{holuk} give rise via Day-convolution to monoidal structures
on $$\Fun(\Fin_*,\mC), \ \Fun(\C\mon^{\mathrm{free},\fin},\mC),$$ 
which are compatible with the localizations of \cref{locquip} by the description of generating local equivalences of \cref{locquip}.
Hence the full subcategories $$ \Mon_{\bE_\infty}(\mC) \subset \Fun(\Fin_*,\mC), \ \C\mon(\mC) \subset \Fun(\C\mon^{\mathrm{free},\fin},\mC)$$ inherit monoidal structures such that
the full subcategory inclusions are lax monoidal and admit monoidal left adjoints.

By properties of the Day-convolution \cite[Recollection 2.2.(3)]{heine2017topological} the functor $$ \Fun(\C\mon^{\mathrm{free},\fin}, \mC) \to \Fun(\Fin_*,\mC)$$ induced by the symmetric monoidal functor $\Fin_* \to \C\mon^{\mathrm{free},\fin}$ and the functors
$$ \Fun(\C\mon^{\mathrm{free},\fin}, \mC) \to \mC, \Fun(\Fin_*,\mC) \to \mC $$ evaluating at 
$\bN, \{0\}_+$, respectively, are lax monoidal and admit a monoidal left adjoint. 
Thus also the restriction $\C\mon(\mC) \to \Mon_{\bE_\infty}(\mC)$ of \cref{restrii} and the 
forgetful functors
$$ \C\mon(\mC) \to \mC, \ \Mon_{\bE_\infty}(\mC) \to \mC $$ are lax monoidal and admit a monoidal left adjoint. 

\end{proof}

\begin{corollary}\label{monoidalos}

The categories $$ \C\mon(\infty\Cat), \ \Mon_{\bE_\infty}(\infty\Cat)$$ carry canonical presentably monoidal structures induced by the Gray-tensor product on $\infty\Cat$.
The free functors $$\infty\Cat \to \C\mon(\infty\Cat), $$$$\infty\Cat \to \Mon_{\bE_\infty}(\infty\Cat), $$$$ \Mon_{\bE_\infty}(\infty\Cat) \to \C\mon(\infty\Cat)$$ are monoidal.

\end{corollary}



\subsection{Strict symmetric powers of $\infty$-categories}

\begin{definition}

Let $\n \geq 0$. The functor $\D^\n: \infty\Grp \to \infty\Grp $ of strict $\n$-th symmetric power is the unique small sifted colimits preserving extension of the functor $$ \Fin \xrightarrow{(-)^{\times \n}_{\Sigma_\n} } \Fin \subset \infty\Grp.$$
\end{definition}

\begin{remark}
Let $\n \geq 0$.
The functor $\D^\n: \infty\Grp \to \infty\Grp$ preserves the final object and so gives rise to a functor
$\infty\Grp_* \to \infty\Grp_*$ that preserves small sifted colimits, and so is uniquely determined by its restriction to $\Fin_*.$
\end{remark}
\begin{construction}
Let $\n \geq 0$ and $\gamma_\n: \Fin_* \to \Fin_*$ the functor sending $X $ to $X^{\times \n}$. 
The canonical natural transformation $* \to  \id $ of functors $\Fin_* \to \Fin_*$, which is trivially $\Sigma_\n$-equivariant, and the identity of $\gamma_\n $ give rise to a $\Sigma_\n$-equivariant natural transformation $ \gamma_\n \to \gamma_\n \times \id \simeq \gamma_{\n+1} $ of functors $\Fin_* \to \Fin_* $. Here we view $\gamma_{\n+1}$ as $\Sigma_\n$-equivariant by restriction along the inclusion $\Sigma_\n \subset \Sigma_{\n+1}$ that extends a bijection to the bijection preserving the maximum.
The natural transformation $ \gamma_\n \to \gamma_{\n+1} $ yields a natural transformation $ \D^\n_{\mid\Fin_*} \to \D^{\n+1}_{|\Fin_*}$ of functors $\Fin_* \to \Fin_* $ that uniquely extends to a natural transformation $$ \D^\n \to \D^{\n+1}$$ of functors $\infty\Grp_* \to \infty\Grp_* $.

\end{construction}

\begin{definition}
The functor $\D^\infty: \infty\Grp \to \infty\Grp $ of strict $\infty$-th symmetric power is the sequential colimit
$$ \D^0 \to \D^1 \to \D^2  \to .... $$

\end{definition}

\begin{lemma}\label{polujk}

\begin{enumerate}
\item There is an equivalence $$ \Sym \simeq \coprod_{\n \geq 0} \D^\n$$
of endofunctors of $\infty\Grp$ under $\id \simeq \D^1. $	

\item There is an equivalence $$ \widetilde{\Sym} \simeq \D^\infty$$ of endofunctors of $\infty\Grp_*$ under $\id \simeq \D^1. $	

\end{enumerate}	

\end{lemma}

\begin{proof}

(1): The functor $\Sym: \infty\Grp \to \infty\Grp$ sends the final space to $\bN$ and so sends any finite set $\{1,...,\n\} $ to $\bN^{\times \n}$ by preadditivity of $\C\mon(\infty\Grp).$ 
So $\Sym$ sends finite sets to sets and therefore sends a finite set to the strict commutative algebra generated by it. Thus by \cite[Corollary 3.1.3.5.]{lurie.higheralgebra} there is a canonical equivalence $$ \coprod_{\n \geq 0} (-)^{\times \n}_{\Sigma_\n} \simeq \Sym_{\mid \Fin}$$ of functors $\Fin \to \Set. $
Since the embedding $\Set \subset \infty\Grp$ preserves small coproducts, the latter equivalence uniquely extends to an equivalence $ \coprod_{\n \geq 0} \D^\n \simeq \Sym$ of functors $\infty\Grp \to \infty\Grp. $

(2): Every finite pointed set is isomorphic to $\{1,...,\n\}_+ $ for some $\n \geq 0$.
Thus the functor $\widetilde{\Sym}: \infty\Grp_* \to \infty\Grp_*$ sends finite pointed sets to sets and so sends a finite pointed set to the strict commutative monoid generated by it. Thus by \cite[Corollary 3.1.3.5.]{lurie.higheralgebra} there is a canonical equivalence $$\D^\infty_{|\Fin_*} \simeq \colim_{\n \geq 0} (-)^{\times \n}_{\Sigma_\n} \simeq \widetilde{\Sym}_{\mid \Fin_*}$$ of functors $\Fin_* \to \Set_*. $
The latter equivalence uniquely extends to an equivalence $\D^\infty \simeq \widetilde{\Sym}$ of functors $\infty\Grp_* \to \infty\Grp_* $ preserving small sifted colimits.

\end{proof}

\begin{definition}
Let $ \n \geq 0 $. The functor $\D^\n: \infty\Cat \to \infty\Cat$ of strict $\n$-th symmetric power is  $$  \infty\Cat \subset \Fun(\Theta^\op, \infty\Grp) \xrightarrow{\D^\n_*} \Fun(\Theta^\op, \infty\Grp)  \to \infty\Cat.$$

\end{definition}

\begin{definition}
Let $ \n \geq 0 $. The functor $\D^\infty: \infty\Cat_* \to \infty\Cat_*$ of strict $\infty$-th symmetric power is $$  \infty\Cat_* \subset \Fun(\Theta^\op, \infty\Grp)_* \simeq \Fun(\Theta^\op, \infty\Grp_*) \xrightarrow{\D^\infty_*}  \Fun(\Theta^\op, \infty\Grp_*) \simeq \Fun(\Theta^\op, \infty\Grp)_*  \to \infty\Cat_*.$$	

\end{definition}

\begin{remark}
Let $\n \geq 0.$ The functor $\D^\n: \infty\Cat \to \infty\Cat$ preserves the final object and so gives rise to a functor
$\D^\n: \infty\Cat_* \to \infty\Cat_*$.
The definition of $\D^\infty: \infty\Cat_* \to \infty\Cat_*$ immediately implies that $\D^\infty: \infty\Cat_* \to \infty\Cat_*$ is the sequential colimit of functors $\D^\n: \infty\Cat_* \to \infty\Cat_*$:
$$ \D^\infty = \colim(\D^0 \to \D^1 \to \D^2 \to ...). $$

\end{remark}

\begin{remark}

Let $\mJ \subset \infty\Cat$ be a dense full subcategory containing $\Theta.$
By density the restricted Yoneda embedding $$\infty\Cat \to \Fun(\mJ^\op, \infty\Cat)$$
factors as the restricted Yoneda embedding $\infty\Cat \to \Fun(\Theta^\op, \infty\Cat)$ 
followed by right-Kan extension along the embedding $\Theta^\op \subset \mJ^\op.$
Hence the restricted Yoneda embedding $$\infty\Cat \to \Fun(\Theta^\op, \infty\Cat)$$
factors as the restricted Yoneda embedding $\infty\Cat \to \Fun(\mJ^\op, \infty\Cat)$ 
followed by restriction along $\Theta^\op \subset \mJ^\op$.
Thus by adjointness
the localization functor $ \Fun(\mJ^\op, \infty\Cat) \to \infty\Cat$ factors as restriction
along $\Theta^\op \subset \mJ^\op$ followed by the localization functor
$ \Fun(\Theta^\op, \infty\Cat) \to \infty\Cat.$ 
Hence the composition $$  \infty\Cat \subset \Fun(\mJ^\op, \infty\Grp) \xrightarrow{\D^\n_*} \Fun(\mJ^\op, \infty\Grp)  \to \infty\Cat$$
factors as $$  \infty\Cat \subset \Fun(\mJ^\op, \infty\Grp) \xrightarrow{\D^\n_*} \Fun(\mJ^\op, \infty\Grp)  \to\Fun(\Theta^\op, \infty\Grp)  \to \infty\Cat, $$
where the second last functor is restriction, and so identifies with $\D^\n: \infty\Cat \to \infty\Cat.$
In other words the definition of $\D^\n$ does not depend on any choice of dense full subcategory of
$\infty\Cat$ containing $\Theta.$
We could for example replace $\Theta$ by the category $\cube$ of oriented cubes in the definition of $\D^\n.$
\end{remark}

\begin{lemma}\label{leoopi}
	
\begin{enumerate}
\item There is a canonical equivalence $$ \Sym \simeq \coprod_{\n \geq 0} \D^\n$$
of endofunctors of $\infty\Cat$ under $\id \simeq \D^1. $	
\item There is a canonical equivalence $$ \widetilde{\Sym} \simeq \D^\infty $$
of endofunctors of $\infty\Cat_*$ under $\id \simeq \D^1. $	
\end{enumerate}	
\end{lemma}

\begin{proof}

By \cite{rezk2010cartesian} (see also \cite[Theorem 3.2.8.]{gepner2025oriented}) there is a localization $ \Fun(\Theta^\op, \infty\Grp) \rightleftarrows \infty\Cat$ 
whose left adjoint preserves finite products, and so gives rise to a localization
$$ \Fun(\Theta^\op, \C\mon(\infty\Grp)) \simeq \C\mon(\Fun(\Theta^\op, \infty\Grp) ) \rightleftarrows \C\mon(\infty\Cat).$$
Consequently, the functor $\Sym: \infty\Cat \to \infty\Cat$ factors as 
$$  \infty\Cat \subset \Fun(\Theta^\op, \infty\Grp) \xrightarrow{\Sym_*} \Fun(\Theta^\op, \infty\Grp)  \to \infty\Cat$$ and the functor $\widetilde{\Sym}: \infty\Cat_* \to \infty\Cat_*$ factors as 
$$ \infty\Cat_* \subset \Fun(\Theta^\op, \infty\Grp)_* \simeq \Fun(\Theta^\op, \infty\Grp_*) \xrightarrow{\widetilde{\Sym}_*}\Fun(\Theta^\op, \infty\Grp_*) \simeq \Fun(\Theta^\op, \infty\Grp)_*  \to \infty\Cat_* .$$

So (1) follows from \cref{polujk} (1) and that the localization functor
$  \Fun(\Theta^\op, \infty\Grp)  \to \infty\Cat$ preserves small coproducts.
And (2) follows from \cref{polujk} (2) and that the localization functor  $  \Fun(\Theta^\op, \infty\Grp_*)  \to \infty\Cat_*$ preserves small filtered colimits.

\end{proof}

\subsection{Connective $H(\bN)$-modules are strict symmetric monoidal $\infty$-categories}

In the following we show that connective $H(\bN)$-module categorical spectra are equivalent to strict symmetric monoidal $\infty$-categories
(\cref{theuta}, \cref{theuta2}). We obtain a categorical Dold-Thom theorem (\cref{free}). Moreover, categorical homology carries a canonical multiplicative structure (\cref{catmonoidal}, \cref{catmonoidali}).

\begin{theorem}\label{theuta}
Let $R$ be an associative algebra in $ \C\mon(\infty\Cat).$

\begin{enumerate}
\item There is a canonical equivalence 
$$ {_R\Mod}(\C\mon(\infty\Cat)) \simeq {_R\Mod}(\Mon_{\bE_\infty}(\infty\Cat)) $$
of presentable reduced oriented categories over $\Mon_{\bE_\infty}(\infty\Cat)$.

\item There is a canonical equivalence 
$$ {\Mod_R}(\C\mon(\infty\Cat)) \simeq {\Mod_R}(\Mon_{\bE_\infty}(\infty\Cat)) $$
of presentable reduced antioriented categories over $\Mon_{\bE_\infty}(\infty\Cat)$.


\end{enumerate}
\end{theorem}

\begin{proof}

(1): By \cref{laxcomp} the free functor $$\lambda: \Mon_{\bE_\infty}(\infty\Cat) \to \C\mon(\infty\Cat)$$ is a left adjoint monoidal reduced bioriented functor and so admits a right adjoint reduced bioriented functor $\nu$
that fits into a commutative triangle of presentable reduced bioriented categories and conservative right adjoint reduced bioriented functors:
$$\begin{xy}
\xymatrix{
\C\mon(\infty\Cat) \ar[rd] \ar[rr]^\nu
&& \Mon_{\bE_\infty}(\infty\Cat) \ar[ld]
\\ 
& \infty\Cat_*.
}
\end{xy}$$
The vertical reduced bioriented functors of the triangle are the forgetful bioriented functors, which preserve and so detect small limits and left and right cotensors as right adjoint reduced bioriented functors.
Thus $\nu$ preserves small limits and left and right cotensors.

In particular, $\nu$ preserves endomorphisms. 
Moreover $\nu$ preserves small colimits by \cref{rerul}.
We prove next that $\nu$ preserves suspensions.
By \cref{deoo} there is an adjunction
$$B:  \Mon_{\bE_\infty}(\infty\Cat) \to  \Mon_{\bE_\infty}(\infty\Cat): \Omega,$$
where the left adjoint preserves finite products and by adjointness identifies with the suspensions
of $\Mon_{\bE_\infty}(\infty\Cat).$
Hence this adjunction induces an adjunction
$$ B: \C\mon(\infty\Cat) \simeq \C\mon(\Mon_{\bE_\infty}(\infty\Cat)) \to \C\mon(\Mon_{\bE_\infty}(\infty\Cat)) \simeq \C\mon(\infty\Cat): \Omega. $$

The right adjoint of the latter adjunction is induced by $\Omega$ and is the endomorphisms of $\C\mon(\infty\Cat)$ so that the left adjoint of the latter adjunction is the suspensions of $\C\mon(\infty\Cat)$.
Consequently, the forgetful functor $$\nu: \C\mon(\infty\Cat) \to \Mon_{\bE_\infty}(\infty\Cat)$$
preserves suspensions.

\cref{domon} implies that $\nu$ preserves oriented pushouts.
Since $\nu$ preserves small colimits and oriented pushouts, $\nu$ preserves right tensors because
$\infty\Cat$ is generated under small colimits by the disks and for every $\n \geq 0$ 
the $\n+1$-disk $\bD^{\n+1}$ is the oriented pushout of $\bD^\n \to *$ with itself by \cref{thas2}.

Hence $\nu$ induces a reduced oriented functor $$ \bar{\nu}: R\mathrm{-}\Mod(\C\mon(\infty\Cat)) \to R\mathrm{-}\Mod(\Mon_{\bE_\infty}(\infty\Cat)), $$
which lifts $\nu$ and so also preserves small colimits and right tensors.
To see that $\bar{\nu}$ is an equivalence, by \cite[Corollary 4.7.3.16.]{lurie.higheralgebra} it remains to see that $\bar{\nu}$ preserves the free object on any small $\infty$-category. Since $\bar{\nu}$ and the free reduced bioriented functors preserve right tensors and $\infty\Cat$ is generated under right tensors by the final $\infty$-category, it suffices to show that $\bar{\nu}$ preserves the free objects on the final $\infty$-category. 
Since the free functors $$ \infty\Cat \to \Mon_{\bE_\infty}(\infty\Cat),  \infty\Cat \to \C\mon(\infty\Cat)$$ are monoidal, the free objects on the final $\infty$-category, the tensor unit of $\infty\Cat,$ both identify with $R$ and the $R$-linear comparison map is the identity. 

The proof of (2) is similar.

\end{proof}

\begin{corollary}\label{theuta2}

\begin{enumerate}
\item There is a canonical equivalence of reduced presentable oriented categories
$$ \C\mon(\infty\Cat) \simeq {_\bN\Mod(\Mon_{\bE_\infty}(\infty\Cat))} \simeq {_{H(\bN)}\Mod(\Cat\Sp_{\geq 0})}$$
over $\Mon_{\bE_\infty}(\infty\Cat) \simeq \Cat\Sp_{\geq 0}$.

\vspace{1mm}

\item There is a canonical equivalence of reduced presentable antioriented categories
$$ \C\mon(\infty\Cat) \simeq \Mod_\bN(\Mon_{\bE_\infty}(\infty\Cat)) \simeq \Mod_{H(\bN)}(\Cat\Sp_{\geq 0})$$
over $\Mon_{\bE_\infty}(\infty\Cat) \simeq \Cat\Sp_{\geq 0}$.

\end{enumerate}

\end{corollary}

\begin{proof}
(1): By \cref{laxcomp} the free functor $\infty\Cat \to \C\mon(\infty\Cat)$ is monoidal.
Hence the tensor unit of $\C\mon(\infty\Cat)$ is free on the final $\infty$-category and so by \cref{naturalis} is $\bN $. 
Thus the forgetful reduced oriented functor
$$_\bN\Mod(\C\mon(\infty\Cat)) \to  \C\mon(\infty\Cat) $$ is an equivalence.

By \cref{symsp} there is a canonical equivalence of presentable bioriented categories
$$ \Mon_{\bE_\infty}(\infty\Cat) \simeq \Cat\Sp_{\geq 0}, $$
which yields an equivalence of presentable reduced oriented categories
$$ {_\bN\Mod}(\Mon_{\bE_\infty}(\infty\Cat)) \simeq {_{H(\bN)}\Mod(\Cat\Sp_{\geq 0})}. $$

So (1) follows from \cref{theuta} (1).

(2): Since the tensor unit of $\C\mon(\infty\Cat)$ is $\bN $,
the forgetful reduced antioriented functor
$$\Mod_\bN(\C\mon(\infty\Cat)) \to  \C\mon(\infty\Cat) $$ is an equivalence.

By \cref{symsp} there is a canonical equivalence of reduced presentable bioriented categories
$$ \Mon_{\bE_\infty}(\infty\Cat) \simeq \Cat\Sp_{\geq 0}, $$
which yields an equivalence of reduced presentable antioriented categories
$$ {\Mod_\bN}(\Mon_{\bE_\infty}(\infty\Cat)) \simeq {\Mod_{H(\bN)}(\Cat\Sp_{\geq 0})}. $$

So (2) follows from \cref{theuta} (2).
	
\end{proof}

\begin{corollary}\label{theutasp0}
Let $\mC$ be a reduced presentable bioriented category and $R$ an associative algebra in $ \C\mon(\mC).$

There is a canonical equivalence of reduced presentable oriented categories over $\mC:$
$$ {_R\Mod(\C\mon(\mC))} \simeq {_R\Mod(\Mon_{\bE_\infty}(\mC))}.$$

	
\end{corollary}

\begin{proof}

The canonical equivalence of \cref{theuta} (1) of presentable reduced oriented categories
$$ {_R\Mod(\C\mon(\infty\Cat))} \simeq {_R\Mod(\Mon_{\bE_\infty}(\infty\Cat))}$$
over $\Mon_{\bE_\infty}(\infty\Cat)$ gives rise to an equivalence of presentable reduced oriented categories
$${_R\Mod(\C\mon(\infty\Cat))} \ot_{\Mon_{\bE_\infty}(\infty\Cat)} \Mon_{\bE_\infty}(\mC) \simeq {_R\Mod(\Mon_{\bE_\infty}(\infty\Cat))}  \ot_{\Mon_{\bE_\infty}(\infty\Cat)} \Mon_{\bE_\infty}(\mC).$$

By \cite[Proposition B.3.]{gepner2016universality} for every presentable reduced bioriented category $\mD$ there are canonical reduced bioriented equivalences 
$$ \C\mon(\mD) \simeq \C\mon(\infty\Grp) \ot \mD, $$$$  \Mon_{\bE_\infty}(\mD) \simeq \Mon_{\bE_\infty}(\infty\Grp) \ot \mD.$$

\cref{Preaddd} implies that $ \C\mon(\infty\Cat)$ is preadditive so that there is a canonical equivalence of presentable reduced bioriented categories
$$  \C\mon(\infty\Cat) \simeq \Mon_{\bE_\infty}(\C\mon(\infty\Cat)) \simeq \C\mon(\Mon_{\bE_\infty}
(\infty\Cat)) \simeq \C\mon(\infty\Grp) \ot \Mon_{\bE_\infty} (\infty\Cat).$$ 
The latter gives rise to an equivalence of presentable reduced bioriented categories
$$ \C\mon(\infty\Cat) \ot_{\Mon_{\bE_\infty}(\infty\Cat)} \Mon_{\bE_\infty}(\mC) \simeq $$$$  \C\mon(\infty\Grp) \ot \Mon_{\bE_\infty} (\infty\Cat) \ot_{\Mon_{\bE_\infty}(\infty\Cat)} \Mon_{\bE_\infty}(\mC) \simeq $$$$ \C\mon(\infty\Grp) \ot \Mon_{\bE_\infty}(\mC) \simeq \C\mon(\Mon_{\bE_\infty}(\mC)).$$
We obtain a canonical equivalence of presentable reduced oriented categories
$$ {_R\Mod(\C\mon(\mC))} \simeq  {_R\Mod(\C\mon(\infty\Cat))} \ot_{\C\mon(\infty\Cat)}\C\mon(\Mon_{\bE_\infty}(\mC)) \simeq $$
$${_R\Mod(\C\mon(\infty\Cat))} \ot_{\C\mon(\infty\Cat)} \C\mon(\infty\Cat) \ot_{\Mon_{\bE_\infty}(\infty\Cat)} \Mon_{\bE_\infty}(\mC) \simeq $$
$$ {_R\Mod(\C\mon(\infty\Cat))} \ot_{\Mon_{\bE_\infty}(\infty\Cat)} \Mon_{\bE_\infty}(\mC) \simeq$$$$ {_R\Mod(\Mon_{\bE_\infty}(\infty\Cat))} \ot_{\Mon_{\bE_\infty}(\infty\Cat)} \Mon_{\bE_\infty}(\mC) \simeq {_R\Mod(\Mon_{\bE_\infty}(\mC))}, $$
where the first and last equivalence are by \cite[Theorem 4.8.4.6.]{lurie.higheralgebra}.

\end{proof}

\begin{corollary}\label{theutasp}

\begin{enumerate}
\item There is a canonical equivalence of presentable reduced oriented categories over $\Cat\Sp:$
$$ \C\mon(\Cat\Sp) \simeq {_{H(\bN)}\Mod}(\Cat\Sp). $$

\item There is a canonical equivalence of presentable reduced antioriented categories over $\Cat\Sp:$
$$ \C\mon(\Cat\Sp) \simeq {\Mod_{H(\bN)}}(\Cat\Sp). $$
	
\end{enumerate}

\end{corollary}

\begin{proof}

(1): The monoidal functor $B^\infty: \infty\Cat  \to \Cat\Sp$ induces a monoidal functor $$ \C\mon(\infty\Cat) \to \C\mon(\Cat\Sp). $$
By \cref{naturalis} the tensor unit of $ \C\mon(\infty\Cat)$ is $\bN.$
Hence the tensor unit of $\C\mon(\Cat\Sp)$ is $H(\bN)$.
Therefore the forgetful oriented functor $$ {_{H(\bN)}\Mod(\C\mon(\Cat\Sp))} \to \C\mon(\Cat\Sp)$$ is an equivalence.

By \cref{theutasp0} (1) there is a canonical equivalence of reduced oriented categories $$ {_{H(\bN)}\Mod(\C\mon(\Cat\Sp))} \simeq {_{H(\bN)}\Mod(\Cat\Sp)}.$$

(2): Since the tensor unit of $\C\mon(\Cat\Sp)$ is $H(\bN)$,
the forgetful antioriented functor $$ {\Mod_{H(\bN)}(\C\mon(\Cat\Sp))} \to \C\mon(\Cat\Sp)$$ is an equivalence.

By \cref{theutasp0} (2) there is a canonical equivalence of reduced antioriented categories $$ {\Mod_{H(\bN)}(\C\mon(\Cat\Sp))} \simeq {\Mod_{H(\bN)}(\Cat\Sp)}.$$


\end{proof}

\begin{remark}
	
There is a commutative square of reduced oriented categories:	
$$\begin{xy}
\xymatrix{
\C\mon(\infty\Cat) \ar[d] \ar[rr]^\simeq
&& {_\bN\Mod(\Mon_{\bE_\infty}(\infty\Cat))} \ar[d]^{B^\infty}
\\ 
\C\mon(\Cat\Sp) \ar[rr]^\simeq  && {_{H(\bN)}\Mod(\Cat\Sp)},
}
\end{xy}$$	
where the left vertical functor is the left adjoint of the functor induced by $\Omega^\infty:\Cat\Sp \to \infty\Cat$ and the right vertical functor is induced by the monoidal embedding $$B^\infty: \Mon_{\bE_\infty}(\infty\Cat) \to \Cat\Sp$$ of \cref{symsp} left adjoint to $\Omega^\infty:\Cat\Sp \to \Mon_{\bE_\infty}(\infty\Cat)$.

Similarly, there is a commutative square of reduced antioriented categories:	
$$\begin{xy}
\xymatrix{
\C\mon(\infty\Cat) \ar[d] \ar[rr]^\simeq
&& {\Mod_\bN(\Mon_{\bE_\infty}(\infty\Cat))} \ar[d]^{B^\infty}
\\ 
\C\mon(\Cat\Sp) \ar[rr]^\simeq  && {\Mod_{H(\bN)}(\Cat\Sp)}.
}
\end{xy}$$	
	
\end{remark}

\cref{leoopi}, \cref{theuta} give the following categorical Dold-Thom theorem:






\begin{corollary}\label{redfree}\label{free}

\begin{enumerate}
\item There are canonical equivalences $$ \Omega^\infty(H(\bN) \wedge (-)) \simeq \D^\infty \simeq \Omega^\infty((-) \wedge H(\bN)).$$ 	

\item There are canonical equivalences $$ \Omega^\infty(H(\bN) \wedge (-)_+)  \simeq \coprod_{\n \geq 0} \D^\n \simeq \Omega^\infty((-)_+ \wedge H(\bN)).$$ 

\end{enumerate}

\end{corollary}

\cref{redfree} implies the following:

\begin{corollary}\label{catmonoidal00}

The categories $$ {_{H(\bN)}\Mod(\Cat\Sp)}, \ {\Mod_{H(\bN)}(\Cat\Sp)} $$ carry presentably monoidal structures, which are canonically equivalent. 

\end{corollary}

\cref{leoopi} and \cref{monoidalos} imply the following:

\begin{corollary}\label{catmonoidal}

Categorical homology $$ H(\bN) \wedge (-): \infty\Cat_* \to H(\bN)\mathrm{-}\Mod(\Cat\Sp) $$ is monoidal.

\end{corollary}

\begin{corollary}\label{catmonoidali}

Categorical homology $$\Omega^\infty(H(\bN) \wedge (-)): \infty\Cat_* \to \infty\Cat_* $$ is lax monoidal.

\end{corollary}

\vspace{2mm}

\subsection{Stratified simplicial commutative monoids}

In the following we introduce stratified simplicial commutative monoids, which will serve as a model for strict symmetric monoidal $\infty$-categories. 

We will prove that stratified simplicial commutative monoids form a combinatorial model category (\cref{Strattrans}).

We first introduce stratified simplicial sets. We refer \cite{verity2008complicial}, \cite{verity2008weak}, \cite{riehl2018complicial}, \cite{ozornova2020model} for background on stratified simplicial sets.
The next definition is \cite[Definition 5]{verity2008weak}:

\begin{definition}
	
A stratified simplicial set is a simplicial set $X$ together with for every $\n \geq 1$ a subset $\mE_\n \subset X_\n$ of distinguished $\n$-simplices that contains all degenerate $\n$-simplices and is called the stratification or the set of thin simplices.
A map of stratified simplicial sets is a map of simplicial sets that preserves the distinguished simplices.

\end{definition}

\begin{definition}
Let $\Strat$ be the category of stratified simplicial sets.	
	
\end{definition}

\begin{example}\label{exammp}
	
The forgetful functor $\Strat \to \sSet$ that forgets the stratification admits a left and right adjoint.
The left adjoint $(-)^\flat: \sSet \to \Strat$ equips a simplicial set with the minimal stratification,
the positive dimensional degenerate simplices.
The  right adjoint $(-)^{\#}: \sSet \to \Strat$ equips a simplicial set with the maximal stratification,
all positive dimensional simplices.
The functor $(-)^{\#}: \sSet \to \Strat$ admits itself a right adjoint that assigns the simplicial subset spanned by all thin simplices.
	
\end{example}

\begin{notation}
For any $\n \geq 1$ let $(\Delta^\n)^\rt$ be the standard $\n$-simplex equipped with all degenerate 
positive dimensional simplices and the unique non-degenerate $\n$-simplex.
By convention we set $(\Delta^0)^\rt := \Delta^0.$

\end{notation}

\begin{notation}
Let $\rt \Delta \subset \Strat$ be the full subcategory spanned by the stratified simplicial sets
$ (\Delta^\n)^\flat, (\Delta^\n)^{t}$ for $\n \geq 0.$ Let $ \Set^{\rt\Delta^\op}:= \Fun(\rt\Delta^\op, \Set).$

\end{notation}

By the next remark $\Strat$ is a canonical localization of a convenient presheaf category.

\begin{remark}\label{Adjfinprod}
The embedding $(-)^{\flat}: \sSet \to \Strat$ restricts to an embedding $\Delta \to t\Delta.$ 
The restricted Yoneda embedding $\Strat \to \Set^{\rt\Delta^\op}$ 
is fully faithful \cite[Observation 11]{verity2008weak} and admits a finite products preserving left adjoint sending $Y$ to $(Y_{|\Delta^\op}, \mathrm{im}(Y((\Delta^\n)^\rt) \to Y((\Delta^\n)^\flat))_{\n \geq 0})$ (\cite{ozornova2020model} after Remark 1.10.).

\end{remark}

\begin{definition}Let $n \geq 0.$
By induction on $n$ we define the $n$-th oriented simplex
$\bDelta^n$ by $\bDelta^0:= \bD^0$ and the oriented cofiber $$\bDelta^n:= \bD^0 \overset{\to}{+}_{\bDelta^{n-1}} \bDelta^{n-1}$$
of the identity.
    
\end{definition}

\begin{remark}

The oriented simplex was studied by \cite{street1987algebra}.
The $n$-th oriented simplex $\bDelta^n$ has a unique non-degenerate $\n$-morphism and the 1-truncation is the totally ordered set $[\n]=\{0<...<\n\}$.

By \cite[Notation 5.2.12.]{gepner2026homotopy} there is a canonical functor $\bDelta^\bullet: \Delta \to \infty\Cat$
whose composition with the functor $\tau_1: \infty\Cat \to \infty\Cat$ is the canonical embedding.

\end{remark}

	


\begin{definition}
	
Let $\mC$ be an $\infty$-category.
The Street nerve $\N(\mC)$ of $\mC$ is the stratified simplicial space whose underlying simplicial space is $ \Map_{\infty\Cat}(\bDelta^{\bullet},\mC)$ and whose thin 
$\n$-simplices for $\n \geq 1$ are precisely 
the functors $\bDelta^\n \to \mC$ that send the unique non-degenerate $\n$-morphism of $\bDelta^\n$ to a degenerate $\n$-morphism.

\end{definition}


By \cite[Theorem 1.25.]{ozornova2020model} there is a cartesian combinatorial model structure on $\Strat$ whose cofibrations are the monomorphisms in $\Strat$, the maps that are objectwise injective. The fibrant objects 
are called complicial sets \cite[Definition 1.21.]{ozornova2020model}. These are higher-categorical analogues of (weak) Kan complexes.

By \cite[Theorem 3.3.2.5.]{loubaton2022complicial} complicial sets are a model for $\infty$-categories: there is a canonical equivalence
\begin{equation}\label{equika}\Strat[\mathrm{weak} \ \mathrm{equivalences}] \simeq \infty\Cat.\end{equation}

\begin{definition}

A stratified simplicial commutative monoid is a pair $(X,\mE)$,
where $X$ is a simplicial commutative monoid and $\mE$ is a set of simplices of $X$ containing all degenerate simplices and for every $\n \geq 0$ the monoid structure on $X_\n$ restricts to $X_\n \cap \mE.$
A map of stratified simplicial commutative monoids $(X,\mE) \to (Y,\mF)$ is a map $X \to Y$ of simplicial commutative monoids sending $\mE $ to $\F$.

\end{definition}

\begin{remark}\emph{}

\begin{itemize}

\item A stratified simplicial commutative monoid is a pair $(X,\mE)$,
where $X$ is a simplicial commutative monoid such that $(X,\mE)$ is a stratified simplicial set and
for every $\n \geq 0$ the monoid structure on $X_\n$ preserves thinness.

\item A map of stratified simplicial commutative monoids $(X,\mE) \to (Y,\mF)$ is a map $X \to Y$ of simplicial commutative monoids that is also a map of stratified simplicial sets $(X,\mE) \to (Y,\mF).$

\end{itemize}

\end{remark}

\begin{remark}

A stratified simplicial commutative monoid is precisely a commutative monoid in $\Strat.$
A map of stratified simplicial commutative monoids is precisely a map of commutative monoids in $\Strat.$

\end{remark}

\begin{notation}
We write $\C\mon(\Strat)$ for the category of commutative monoid objects in $\Strat$, the category of stratified commutative monoids.	

\end{notation}

\begin{example}
The forgetful functor $\C\mon(\Strat) \to \Strat$ admits a left adjoint, the free stratified simplicial commutative monoid functor, that sends a stratified simplicial set $(X,\mE)$
to the stratified simplicial commutative monoid $(\N[X],\bN[\mE])$,
where $ \bN[X]$ is the free simplicial commutative monoid on $X$ satisfying 
$  \bN[X]_\n = \bN[X_\n] $ for every $\n \geq 0$ and $\N[\mE]$ is the graded set satisfying 
$  \bN[\mE]_\n = \bN[\mE_\n] $.

\end{example}

In the following we will prove the following theorem:

\begin{theorem}\label{Strattrans}

There is a right induced (and so combinatorial) model structure on the category $\C\mon(\mathrm{Strat})$ of stratified commutative monoids.	

\end{theorem}

The next definitions  are \cite[\S 2]{farjoun2006homotopy}.

\begin{definition}
	
Let $\mD$ be a category. A functor $\phi: \mD \to \Set$ is a free diagram if it is a coproduct 
of functors of the form $\mD(X,-)$ for some $X \in \mD.$	
	
\end{definition}

\begin{definition}Let $\mD$ be a category and $\phi: \mD \to \Set$ a functor. 
A collection $\{ (\zeta_\alpha, X_\alpha) \mid \alpha \in \mJ \}$, where $X_\alpha \in \mD$ and $\zeta_\alpha \in \phi(X_\alpha)$ for every $\alpha \in \mJ$ corresponds to a transformation
$\theta:  \coprod_{\alpha \in \mJ} \mD(X_\alpha,-) \to \phi.$ 
We say that $\{ (\zeta_\alpha, X_\alpha) \mid \alpha \in \mJ \}$ generates $\phi$ if $\theta$ is an isomorphism. In this case we call $\{ (\zeta_\alpha, X_\alpha) \mid \alpha \in \mJ \}$ a collection of generators of $\phi.$
	
\end{definition}

\begin{remark}
A functor $\mD \to \Set$ is a free diagram if and only if there is a collection of generators.
But in general the collection of generators is not unique.
	
\end{remark}

\begin{definition}
	
Let $\mD$ be a category. A functor $\phi: \mD \to \Set^{\Delta^\op}$ is a free diagram if the following two conditions hold:

\begin{enumerate}
\item For every $\n \geq 0$ the functor $  \mD \xrightarrow{\phi} \Set^{\Delta^\op} \xrightarrow{\ev_\n} \Set$ is a free diagram, where $\ev_\n$ evaluates at $n$.

\vspace{1mm}

\item For every $\n \geq 0$ there is a collection of generators $\mQ_\n= \{ (\zeta_\alpha, X_\alpha) \mid \alpha \in \mJ_\n \}$ of $\ev_\n \circ \phi$ such that $\cup_{\n \geq 0}\mQ_\n$ is closed under degeneracies, i.e. for any order preserving surjection $\sigma: [\m] \to [n]$ and $\alpha \in \mJ_\n$ the map $\sigma^*: \ev_\n(\phi(X_\alpha)) \to \ev_\m(\phi(X_\alpha)) $ sends $\zeta_\alpha$ to $\zeta_\beta$ for some $\beta \in \mJ_\m$ with $X_\beta=X_\alpha.$

\end{enumerate}

\end{definition}

\begin{definition}
	
Let $\mD$ be a category. A functor $\phi: \mD \to \Strat$ is a free diagram if the following two conditions hold:
	
\begin{enumerate}
\item The functor $\phi: \mD \to \Strat \xrightarrow{\mathrm{forget}} \Set^{\Delta^\op}	$ is a free diagram.
	
\item For every $\n \geq 0$ the functor $  \mD \xrightarrow{\phi} \Strat \to \Set$ is a free diagram, where the last functor assigns the set of thin $\n$-simplices.
			
\end{enumerate}
	
\end{definition}

Next we consider our example of interest.

\begin{notation}
	
Let $\G$ be a group and $\mC$ a category.

\begin{enumerate}

\item Let $\mC[\G]:= \Fun(B\G,\mC)$ be the category of objects in $\mC$ with $\G$-action.

\vspace{1mm}
		
\item The orbit category of $\G$ is the full subcategory $\mO(\G) \subset \Set[\G]$ spanned by the $\G$-sets of the form $\G/H$ for some subgroup $H$ of $\G.$

\end{enumerate}
			
\end{notation}

\begin{notation}
Let $\G$ be a group. The canonical embedding $B\G \subset \Set[\G]$ hitting $\G$ induces an embedding $B\G \subset \mO(\G).$ 
Let $\mC$ be a category that admits small limits. The induced functor $$ \Fun(\mO(\G),\mC) \to \Fun(B\G,\mC)=\mC[\G]$$ admits a right adjoint that sends $X \in \mC[\G]$ to the functor
$\widehat{X}: \mO(\G) \to \mC, \G/H \mapsto X^H.$ 	
	
\end{notation}

\begin{remark}
The functor $B\G \subset \mO(\G)$ is cofinal as functor of $(1,1)$-categories.
Hence for every $(1,1)$-category $\mC$ that admits the necessary colimits and object $X \in \mC[\G]$ the canonical map
$$  X_\G:=\colim(X) \to \colim(\widehat{X}) $$ is an equivalence.
	
\end{remark}

\begin{example}\label{Free}Let $\G$ be a group.
	
\begin{enumerate}
	
\item Let $X $ be a set with $\G$-action.
The functor $\widehat{X}: \mO(\G)^\op \to \Set, \G/ H \mapsto X^H $ is a free diagram.

\item Let $X $ be a simplicial set with $\G$-action.
The functor $\widehat{X}: \mO(\G)^\op \to \Set^{\Delta^\op}, \G/ H \mapsto X^H $ is a free diagram.

\item Let $X $ be a stratified simplicial set with $\G$-action.
The functor $\widehat{X}: \mO(\G)^\op \to \Strat, \G/ H \mapsto X^H $ is a free diagram.
	
\end{enumerate}

The first and second examples are \cite[\S 2.7.]{farjoun2006homotopy}.
The third example follows immediately from the first and second since the forgetful 
functors $\Strat \to \Set^{\Delta^\op}, \Strat \to \Set^{\bN}$ assigning the underlying simplicial set, the 
sets of thin arrows, respectively, preserve small limits.
	
\end{example}

\begin{lemma}\label{cell}
	
Let $\mD$ be a category.
Every free diagram $X: \mD \to \Strat$ is projectively cofibrant.	
	
\end{lemma}

\begin{proof}
	
For every $\n \geq 0$ and $X \in \Strat$ let $X_{\leq \n} \in \Strat$ be the stratified simplicial subset of $X$ such that the inclusion $ X_{\leq \n} \to X$ in $\Strat$ is the identity on underlying simplicial sets and on thin simplices of dimension smaller $\n+1$ and such that all thin simplices of  $ X_{\leq \n}$ of dimension larger $n$ are degenerate.
Then $X_{\leq 0}$ is the minimal stratification on the underlying simplicial set of $X.$ 
For every $\m \geq \n \geq 0$ we obtain a functor $ (-)_{\leq \n}: \Strat \to \Strat$
and transformations $ (-)_{\leq \n} \to (-)_{\leq \m}, (-)_{\leq \m} \to \id$.
Trivially, the inclusions $  X_{\leq \n} \to  X$ in $\Strat$ give rise to an isomorphism
$ \colim(X_{\leq0} \to X_{\leq1} \to ...  ) \to X$ in $\Strat.$ 

The functor $ (-)_{\leq \n}: \Strat \to \Strat$ gives rise to a functor $  (-)_{\leq \n}: \Fun(\mD,\Strat) \to \Fun(\mD, \Strat)$ which we denote by the same name.
So for every functor $X: \mD \to \Strat$ and $\n \geq 0$ there are inclusions
$ X_{\leq \n} \to X$ in $\Fun(\mD,\Strat)$ that give rise to an isomorphism
$ \colim(X_{\leq0} \to X_{\leq 1} \to ...  ) \to X$ in $\Fun(\mD,\Strat).$ 
We prove first that for every functor $X: \mD \to \Strat$, which is a free diagram, the functor $X_{\leq 0} : \mD \to \Strat$ is projectively cofibrant.

For that it is enough to prove the following: for every functor $Y: \mD \to \Set^{\Delta^\op}$, which is a free diagram, the functor $Y^\flat: \mD \to \Strat$ is projectively cofibrant.
By \cite[Proposition 2.5.]{farjoun2006homotopy} the functor $Y: \mD \to \Set^{\Delta^\op}$ is a $\mD$-CW-complex, i.e. $Y$ is the sequential colimit of a diagram $Y^{-1} := \emptyset \to Y^0 \to Y^1 \to ...$ and for every $\n \geq 0 $ the morphism $Y^{\n-1} \to Y^{n} $ is the cobase change of a coproduct of morphisms of the form 
\begin{equation*}
\Map_\mD(Z,-) \times \partial\bD^\n  \to \Map_\mD(Z,-) \times \bD^\n 
\end{equation*}
for $\n \geq 0$ and $ Z \in \mD. $
Since the functor $(-)^\flat: \Set^{\Delta^\op} \to \Strat$ is a left adjoint and so preserves colimits, the functor $Y^\flat: \mD \to \Strat$ is a sequential colimit of the diagram $(Y^{-1})^\flat = \emptyset \to (Y^0)^\flat \to (Y^1)^\flat \to ... $ in $\Strat$ and for every $\n \geq 0 $ the morphism $(Y^{\n-1})^\flat \to (Y^{n})^\flat $ is the cobase change of a coproduct of morphisms of the form 
\begin{equation}\label{morti}
\Map_\mD(Z,-)\times (\partial\bD^\n)^\flat  \to \Map_\mD(Z,-)\times (\bD^\n)^\flat
\end{equation}
for $ \n \geq 0$ and $ Z \in \mD.$
The functor $\Fun(\mD,\Strat) \to \Strat$ evaluating at $Z \in \mD$ is right adjoint to the functor
$ A \mapsto \Map_\mD(Z,-) \times A.$ Since fibrations and trivial fibrations for the projective model structure are objectwise, the functor $\Fun(\mD,\Strat) \to \Strat$ evaluating at $Z \in \mD$ is a right Quillen functor so that its left adjoint is a left Quillen functor. 
For every $\n \geq 0$ the map $ (\partial\bD^\n)^\flat  \to (\bD^\n)^\flat $ is a cofibration in $\Strat$
so that for every $Z \in \mD$ the morphism (\ref{morti}) in $\Fun(\mD,\Strat) $ is a projective cofibration. 
Since cofibrations in any model category are closed under cobase change and transfinite compositions, we find that for every $\n \geq 0$ the map $(Y^{\n-1})^\flat  \to (Y^\n)^\flat$
and so the map $\emptyset =(Y^{-1} )^\flat \to Y^\flat $ are projective cofibrations. In other words $Y^\flat$ is projectively cofibrant.

We have seen that for every functor $X: \mD \to \Strat$, which is a free diagram, the functor $X_{\leq 0}$ is projectively cofibrant.
Consequently, it suffices to show that for every $\n \geq 0$ the inclusion $ X_{\leq \n-1} \to X_{\leq \n}$ is a projective cofibration. For every $\n \geq 0$ the canonical map $ (\bD^\n)^\flat \to (\bD^\n)^\rt$ is a cofibration in $\Strat$ so that for every $Z \in \mD$ the morphism 
\begin{equation}\label{morti2}
\Map_\mD(Z,-)\times (\bD^\n)^\flat \to\Map_\mD(Z,-)\times (\bD^\n)^\rt \end{equation}
in $\Fun(\mD,\Strat)$ is a projective cofibration.

Since cofibrations in any model category are closed under cobase change and transfinite compositions, it is enough to see that for every $\n \geq 0$ the inclusion $ X_{\leq \n-1} \to X_{\leq \n}$ is a cobase change of a coproduct of morphisms of the form (\ref{morti2}) for $Z \in \mD$.

For every $\n \geq 0$ the functor $ \mD \xrightarrow{X} \Strat \xrightarrow{ } \Set$,
where the last functor takes the thin $\n$-simplices, is a free diagram and so there is a collection of generators $\mQ= \{ (\zeta_\alpha, T_\alpha) \mid \alpha \in \mJ \}$ of the latter functor.
Let $\mJ' \subset \mJ$ be the subset of elements $\alpha$ such that the thin $\n$-simplex $ \zeta_\alpha $ of $ X(T_\alpha)$ is non-degenerate.

For every $\alpha \in \mJ$ the thin $\n$-simplex $ \zeta_\alpha $ of $ X(T_\alpha)$ corresponds to a map $ \Map_{\mD}(T_\alpha,-) \times (\Delta^\n)^\rt \to X$ in $\Fun(\mD,\Strat)$
that induces a map $ \Map_{\mD}(T_\alpha,-) \times (\Delta^\n)^\rt \to X_{\leq \n}$ in $\Fun(\mD,\Strat)$.
We obtain a map $$\coprod_{\alpha \in \mJ'} \Map_\mD(T_\alpha,-) \times (\Delta^\n)^\rt \to X_{\leq \n}$$ in $\Fun(\mD,\Strat).$
We obtain a commutative square in $\Fun(\mD,\Strat)$:
$$\begin{xy}
 \xymatrix{
\coprod_{\alpha \in \mJ'} \Map_\mD(T_\alpha,-) \times (\Delta^\n)^\flat \ar[d] \ar[r]
& \coprod_{\alpha \in \mJ'} \Map_\mD(T_\alpha,-) \times (\Delta^\n)^\rt  \ar[d]
\\ 
X_{\leq \n-1} \ar[r]  & X_{\leq \n}.
}
 \end{xy}$$	
We complete the proof by showing that this commutative square is a pushout square in
$\Fun(\mD,\Strat)$.
Both horizontal maps in the latter square forget to the identities in $\Fun(\mD,\Set^{\Delta^\op})$
and induce objectwise the identities on thin simplices of dimension different from $n.$ 
In particular, the latter square lies over a pushout square in $\Fun(\mD,\Set^{\Delta^\op})$.
The forgetful functor $\Strat \to \Set^{\Delta^\op}$ admits a right adjoint taking the maximal stratification and so preserves small colimits. Hence the map in $\Fun(\mD,\Strat)$ from the pushout to $X_{\leq \n}$ forgets to the identity in $\Fun(\mD,\Set^{\Delta^\op})$, and 
induces objectwise the identities on thin simplices of dimension different from $n.$ 
So it remains to see that for every $Z \in \mD$ every thin $\n$-simplex of $X_{\leq \n}(Z)$, i.e. thin $\n$-simplex of $X(Z)$, belongs to the pushout evaluated at $Z.$
By definition every degenerate $\n$-simplex of $X(Z)$ belongs to $X_{\leq \n-1}(Z)= X(Z)_{\leq \n-1}$. Hence we need to see that every non-degenerate thin $\n$-simplex $\sigma$ of $X(Z)$ belongs to the pushout evaluated at $Z.$
By generation there is an $\alpha \in \mJ$ and a morphism $\kappa: T_\alpha \to Z$
such that the induced map $X(\kappa): X(T_\alpha) \to X(Z)$ sends $\zeta_\alpha$ to $\sigma.$ 
If $\zeta_\alpha$ were degenerate, also the image $\sigma$ of $\zeta_\alpha$ under $X(\kappa)$ were degenerate. So $\zeta_\alpha$ is non-degenerate, so that $\alpha \in \mJ'$ and
$\sigma$ is the image of $\kappa$ and the unique non-degenerate $\n$-simplex of $\Delta^\n$ under the right vertical map of the latter square.

\end{proof}

\begin{corollary}\label{projcof}
	
Let $\G$ be a group and $X $ an object in $\Strat$ with $\G$-action. 
The functor $ \mO(\G)^\op \to \Strat, 
\G/ H \mapsto X^H $ is projectively cofibrant. 
	
\end{corollary}

\begin{proof}
The diagram $ \widehat{X}: \mO(\G)^\op \to \Strat, \G/ H \mapsto X^H $ is free by \cref{Free}.
We apply \cref{cell}.
	
\end{proof}

The next proposition is an adaptation of \cite[Theorem 1.3.]{casacuberta2016localizations}. We are grateful to David White who made us aware of \cite{casacuberta2016localizations}.

\begin{proposition}\label{Symso} For every $\n \geq 0$ the functor
$ \Strat \to \Strat, X \mapsto (X^{\times \n})_{\Sigma_\n}$ preserves weak equivalences.
	
\end{proposition}

\begin{proof}
For every map $X \to Y$ in $\Strat$ the induced map
$ (X^{\times \n})_{\Sigma_\n} \to (Y^{\times \n})_{\Sigma_\n}$
is the induced map on colimits $\colim(\widehat{X^{\times \n}}) \to \colim(\widehat{Y^{\times \n}})$
and by \cref{projcof} the induced map on homotopy colimits.
Thus the induced map
$ (X^{\times \n})_{\Sigma_\n} \to (Y^{\times \n})_{\Sigma_\n}$
is a weak equivalence if for every $\Sigma_\n/H \in \mO(\Sigma_\n)$
the induced morphism 
$$\widehat{X^{\times \n}}(\Sigma_\n/H)= (X^{\times \n})^H \to \widehat{Y^{\times \n}}(\Sigma_\n/H) = (Y^{\times \n})^H $$ is a weak equivalence. 
By \cite[Remark 4.1.]{badzioch2001recognition} the latter map identifies with the map $X^{\times\m}  \to Y^{\times\m}$,
where $\m \leq \n$ is the number of $H$-orbits on $\{1,...,n \}$.
So the statement follows from the fact that weak equivalences in $\Strat$ are closed under finite products since the model structure on $\Strat$ is cartesian and every object is cofibrant
\cite[Theorem 1.25.]{ozornova2020model}.

\end{proof}

\begin{corollary}\label{Symsoi} The functor
$$\Sym: \Strat \to \Strat, X \mapsto \coprod_{\n \geq 0} (X^{\times \n})_{\Sigma_\n} $$ preserves weak equivalences.	
\end{corollary}

\begin{proof}
By \cref{Symso} for every $\n \geq 0$ the functor
$ \Strat \to \Strat, X \mapsto (X^{\times \n})_{\Sigma_\n}$ preserves weak equivalences.
In any model category arbitrary coproducts of trivial cofibrations are trivial cofibrations.
So by Ken-Brown's Lemma arbitrary coproducts of weak equivalences between cofibrant objects are weak equivalences. Since every object of $\Strat $ is cofibrant, the functor
$\Sym: \Strat \to \Strat$ preserves weak equivalences.	
	
\end{proof}

\begin{proof}[Proof of \cref{Strattrans}]
	
This follows from \cite[Theorem 3.2., Lemma A.3.]{white2017model} and \cref{Symsoi}.
	
\end{proof}

\subsection{A complicial model for strict symmetric monoidal $\infty$-categories}

Next we establish a presentation of strict symmetric monoidal $\infty$-categories by stratified simplicial commutative monoids (\cref{Sing}).
This leads to a complicial presentation of categorical homology via the Street nerve, which is a higher-categorical analogue of the simplicial presentation of homology via the singular complex.

By \cite[Theorem 1.25.]{ozornova2020model} the model structure on $\Strat$ is cartesian and every object is cofibrant. Therefore weak equivalences in $\mathrm{Strat} $ are closed under finite products
and so the canonical functor $\mathrm{Strat} \to \mathrm{Strat}[\mathrm{weak} \ \mathrm{equivalences}]$ preserves finite products. Hence it induces a functor
$$ \C\mon(\mathrm{Strat}) \to \C\mon(\mathrm{Strat}[\mathrm{weak} \ \mathrm{equivalences}])$$ that inverts weak equivalences and so gives rise to a functor 
\begin{equation*}\label{cancomp} \C\mon(\mathrm{Strat})[\mathrm{weak} \ \mathrm{equivalences}] \to \C\mon(\mathrm{Strat}[\mathrm{weak} \ \mathrm{equivalences}]) \simeq \C\mon(\infty\Cat),\end{equation*}
where the last equivalence is (\ref{equika}).

\vspace{1mm}

We will prove the following:

\begin{theorem}\label{Sing}\label{polut}
	
The induced functor
\begin{equation*}\label{cancomp} \C\mon(\mathrm{Strat})[\mathrm{weak} \ \mathrm{equivalences}] \to \C\mon(\mathrm{Strat}[\mathrm{weak} \ \mathrm{equivalences}]) \simeq \C\mon(\infty\Cat) \end{equation*}
is an equivalence.

	
\end{theorem}

	
	

We will use the following adjunction of \cite[Remark 2.13.]{ozornova2020model}:

\begin{notation}\label{realiz}

Let $|-|: \Fun(\Delta^\op, \Set^{\rt\Delta^\op}) \cong \sSet^{\rt\Delta^\op}\rightleftarrows \Set^{\rt\Delta^\op}: \Sing $ 
be the unique left adjoint $\Set^{\rt\Delta^\op}$-linear functor extending the functor $\Delta \to \Set^{\rt\Delta^\op}, [n] \mapsto (\Delta^\n)^{\#}$ along the embedding
$$\Delta \subset  \Fun(\Delta^\op, \Set) \to \Fun(\Delta^\op, \Set^{\rt\Delta^\op})$$ induced by the diagonal embedding $\iota: \Set \to \Set^{\rt\Delta^\op}$.

\end{notation}

\begin{lemma}\label{realprod}

The realization $|-|: \sSet^{\rt\Delta^\op}\cong \Fun(\Delta^\op, \Set^{\rt\Delta^\op}) \to \Set^{\rt\Delta^\op}$ preserves finite products.	
\end{lemma}
\begin{proof}

Let $\iota: \Set \to \Set^{\rt\Delta^\op}$ be the diagonal embedding, which preserves small colimits and small limits and $$\iota_*:  \sSet= \Fun(\Delta^\op, \Set) \to \Fun(\Delta^\op, \Set^{\rt\Delta^\op})\cong \sSet^{\rt\Delta^\op}$$ the induced functor. 
By uniqueness of Yoneda extension the left adjoint functor $$|-| \circ \iota_*: \sSet \to \sSet^{\rt\Delta^\op}\cong \Fun(\Delta^\op, \Set^{\rt\Delta^\op}) \to \Set^{\rt\Delta^\op}$$ and the left adjoint functor $(-)^{\#}: \sSet \to \Set^{\rt\Delta^\op}$ of \cref{exammp} are equivalent since they are equivalent after restriction along the Yoneda embedding.
By \cref{exammp} the functor $(-)^{\#}: \sSet \to \Set^{\rt\Delta^\op}$ preserves small limits and so finite products. Hence for every $n, \m \geq 0$ the canonical map $$|\iota_*(\Delta^\n) \times \iota_*(\Delta^\m) | \cong |\iota_*(\Delta^\n \times \Delta^\m) |  \to |\iota_*(\Delta^\n)|  \times |\iota_*(\Delta^\m)| $$ is an isomorphism.
Thus for every $X, Y \in \Set^{\rt\Delta^\op}$ the canonical map $$|\iota_*(\Delta^\n) \times X \times \iota_*(\Delta^\m) \times Y |   \cong |\iota_*(\Delta^\n) \times \iota_*(\Delta^\m) \times X \times Y | 
\cong |\iota_*(\Delta^\n) \times \iota_*(\Delta^\m)| \times X \times Y 
\to $$$$ | \iota_*(\Delta^\n) | \times | \iota_*(\Delta^\m) | \times X \times Y \cong| \iota_*(\Delta^\n) | \times X  \times | \iota_*(\Delta^\m) | \times Y \cong | \iota_*(\Delta^\n) \times X |  \times | \iota_*(\Delta^\m) \times Y | $$ is an isomorphism.
This implies the result since $ \Set^{\rt\Delta^\op}, \sSet^{\rt \Delta^\op} $ are cartesian closed being categories of presheaves and $\sSet^{\rt \Delta^\op} \cong \Fun(\Delta^\op, \Set^{\rt \Delta^\op})$ is generated under small colimits by objects of the form $\iota_*(\Delta^\n) \times X$ for $\n \geq 0$ and $X \in  \Set^{\rt\Delta^\op}.$

\end{proof}


	

\begin{proof}[Proof of \cref{polut}]
	
By \cite[Theorem 2.7.]{ozornova2020model} there is a left Bousfield localization of the injective model structure on $ \sSet_\inj^{\rt\Delta^\op}$,
where $\sSet$ carries the Kan model structure whose fibrant objects are the precomplicial spaces \cite[Definition 2.5.]{ozornova2020model}. We write $ \sSet_{\precomplicial}^{\rt\Delta^\op} $ for this model structure.

By \cite[Theorem 1.28.]{ozornova2020model} the category $ \Set^{\rt\Delta^\op} $ carries a model structure whose cofibrations are the monomorphisms and whose fibrant objects are the precomplicial sets \cite[Definition 1.23.]{ozornova2020model}.
We write $ \Set_{\precomplicial}^{\rt\Delta^\op} $ for this model structure.
By \cite[Theorem 2.14.]{ozornova2020model} the adjunction $$|-|: \sSet_{\precomplicial}^{\rt\Delta^\op} \rightleftarrows \Set_{\precomplicial}^{\rt\Delta^\op}: \Sing $$ of \cref{realiz} is a Quillen equivalence.

By \cite[Proposition 1.35.]{ozornova2020model} there is a Quillen equivalence $L: \Set_{\precomplicial}^{\rt\Delta^\op} \rightleftarrows \Strat : \iota$ where the right adjoint is the restricted Yoneda embedding and the left adjoint 
preserves finite products by \cref{Adjfinprod}. 
We obtain a Quillen equivalence $$ L \circ |-|:  \sSet_{\precomplicial}^{\rt\Delta^\op}  \rightleftarrows \Set_{\precomplicial}^{\rt\Delta^\op} \rightleftarrows \Strat: \Sing \circ \iota.$$ 
Composing with the Quillen adjunction $$\id: \sSet_{\inj}^{\rt\Delta^\op} \rightleftarrows  \sSet_{\precomplicial}^{\rt\Delta^\op}: \id$$ we obtain a Quillen adjunction $$ L \circ |-|:  \sSet_{\inj}^{\rt\Delta^\op}  \rightleftarrows \Set_{\precomplicial}^{\rt\Delta^\op} \rightleftarrows \Strat: \Sing \circ \iota $$ whose counit at any fibrant object is a weak equivalence.

\cref{realprod} guarantees that the left adjoint preserves finite products.

By \cite[Theorem 3.2., Lemma A.3.]{white2017model} the category $ \C\mon(\sSet_{\inj}^{\rt\Delta^\op})$ admits a model structure right induced from $\sSet_{\inj}^{\rt\Delta^\op}$
since cofibrations and trivial cofibrations in $\sSet_{\inj}^{\rt\Delta^\op}$ are objectwise.

By \cref{Strattrans} the category $ \C\mon(\Strat)$ admits a right induced model structure.

We obtain an induced adjunction $$ L \circ |-|:  \\
\Cmon(\sSet_{\inj}^{\rt\Delta^\op})  \rightleftarrows \Cmon(\Strat): \Sing \circ \iota $$ whose counit lifts the counit of the former adjunction, and so is a weak equivalence at any fibrant object.
So for the induced adjunction on localizations
\begin{equation}\label{adj01}
L \circ |-|:  \C\mon(\sSet_{\inj}^{\rt\Delta^\op})[\mathrm{weak} \ \mathrm{equivalences}]  \rightleftarrows \C\mon(\Strat)[\mathrm{weak} \  \mathrm{equivalences}]: \Sing \circ \iota \end{equation}
the counit is an equivalence and the unit lifts the unit of the adjunction
\begin{equation}\label{adj02} L \circ |-|:  (\sSet_{\inj}^{\rt\Delta^\op})[\mathrm{weak} \  \mathrm{equivalences}]  \rightleftarrows \Strat[\mathrm{weak} \ \mathrm{equivalences}]: \Sing \circ \iota. \end{equation}
Hence every $X \in \C\mon(\Strat)[\mathrm{weak}\ \mathrm{equivalences}]$ belongs to the essential image of the right adjoint of adjunction (\ref{adj01}) if and only if the unit of adjunction (\ref{adj01}) at $X$ is an equivalence if and only if the unit of adjunction (\ref{adj02}) at the image of $X \in \Strat[\mathrm{weak} \ \mathrm{equivalences}]$ is an equivalence if and only if the image of $X \in \Strat[\mathrm{weak} \  \mathrm{equivalences}]$ belongs to the essential image of the right adjoint of (\ref{adj02}).

By \cite[Theorem 3.1.]{schwede2001stable} the category $\C\mon(\sSet) $ admits a right induced (and so combinatorial) model structure from the Kan-model structure on $\sSet.$ We consider the injective model structure on $$ \Fun(\rt\Delta^\op, \C\mon(\sSet)) \simeq \C\mon(\sSet^{\rt\Delta^\op}),$$ which we denote by $ \Fun(\rt\Delta^\op, \C\mon(\sSet))_\inj$.

Since the classes of weak equivalences of $ \Fun(\rt\Delta^\op, \C\mon(\sSet))_\inj$ and $\C\mon(\sSet_{\inj}^{\rt\Delta^\op})$ are both the class of levelwise weak equivalences, 
there is a canonical equivalence $$  \Fun(\rt\Delta^\op, \C\mon(\sSet))_\inj \ [\mathrm{weak} \ \mathrm{equivalences}]  \simeq  \C\mon(\sSet_{\inj}^{\rt\Delta^\op})[\mathrm{weak} \ \mathrm{equivalences}].$$
By \cite[Proposition 4.2.4.4.]{lurie.HTT} the canonical functors $$  \Fun(\rt\Delta^\op, \C\mon(\sSet))_\inj \ [\mathrm{weak} \ \mathrm{equivalences}]  \to \Fun(\rt\Delta^\op, \C\mon(\sSet)[\mathrm{weak} \ \mathrm{equivalences}]), $$
$$  \sSet^{\rt\Delta^\op}_\inj[\mathrm{weak} \  \mathrm{equivalences}]  \to \Fun(\rt\Delta^\op, \sSet[\mathrm{weak} \ \mathrm{equivalences}])\simeq \Fun(\rt\Delta^\op, \infty\Grp)$$
are equivalences.

The canonical functor $ \sSet \to \infty\Grp$ preserves finite products since weak homotopy equivalences are closed under finite products and so induces a functor
$ \C\mon(\sSet) \to \C\mon(\infty\Grp)$ that inverts weak homotopy equivalences. So the latter gives rise to a functor $$ \C\mon(\sSet)[\mathrm{weak} \ \mathrm{equivalences}]) \to \C\mon(\infty\Grp), $$ which is an equivalence by \cite[Theorem 6.4.]{badzioch2002algebraic}.
Hence the canonical functor $$ \C\mon(\sSet^{\rt\Delta^\op}_\inj)[\mathrm{weak} \  \mathrm{equivalences}]  \simeq   \Fun(\rt\Delta^\op, \C\mon(\sSet))_\inj \ [\mathrm{weak} \ \mathrm{equivalences}] $$$$ \simeq \Fun(\rt\Delta^\op, \C\mon(\infty\Grp)) \simeq \C\mon(\Fun(\rt\Delta^\op,\infty\Grp))$$
is an equivalence.
Thus adjunction (\ref{adj02}) gives an adjunction \begin{equation}\label{adj03}
\Fun(\rt\Delta^\op,\infty\Grp) \rightleftarrows \Strat[\mathrm{weak} \  \mathrm{equivalences}] \end{equation}
whose right adjoint is fully faithful.
Adjunction (\ref{adj01}) gives an adjunction \begin{equation}\label{adj04}
\C\mon(\Fun(\rt\Delta^\op,\infty\Grp)) \rightleftarrows \C\mon(\Strat)[\mathrm{weak} \  \mathrm{equivalences}] \end{equation}
whose right adjoint is fully faithful and whose local objects are precisely the objects
whose image in $\Fun(\rt\Delta^\op,\infty\Grp)$ is local for the localization (\ref{adj03}).
Consequently, the right adjoint of adjunction (\ref{adj04}) induces an equivalence
$$ \C\mon(\Strat)[\mathrm{weak} \  \mathrm{equivalences}] \simeq \Strat[\mathrm{weak} \  \mathrm{equivalences}] \times_{\Fun(\rt\Delta^\op,\infty\Grp)} \C\mon(\Fun(\rt\Delta^\op,\infty\Grp)) $$$$ \simeq 
\C\mon(\Strat[\mathrm{weak} \  \mathrm{equivalences}]).$$ 
	
\end{proof}

\cref{Sing} implies the following:

\begin{corollary}\label{commodel}
	
Let $(X,\mE)$ be a stratified simplicial set. 
The categorical homology of the $\infty$-category modeled by $(X,\mE)$ is modeled by
$$(\bN[X], \bN[\mE]).$$

\end{corollary}

We obtain the following:

\begin{corollary}\label{Orientmodel} Let $n \geq 0.$
The stratified simplicial commutative monoid
$\bN[\Delta^n]$ is a model for the categorical homology of $\bDelta^n.$

\end{corollary}

\begin{proof}

The canonical map $\Delta^n \to \N(\bDelta^n)$ is an equivalence by \cite[Theorem 2.1.]{maehara2023orientals}. So it induces a weak equivalence
$$\bN[\Delta^n] \to \bN[\N(\bDelta^n)] $$
since $\bN[-]: \Strat \to \Cmon(\Strat)$
is a left Quillen functor by \cref{Strattrans}.
We apply \cref{commodel}.

\end{proof}

\subsection{Categorical homology of the globes}


In this section we compute the categorical homology of the disks (\cref{disk1}, \cref{disk2}). We obtain that categorical homology can distinguish disks of different dimension.

\begin{notation}
	
Let $\N: \infty\Cat \to \Fun(\Theta^\op, \infty\Grp)$ be the restricted Yoneda embedding.

\end{notation}

\begin{remark}

The functor $\N: \infty\Cat \to \Fun(\Theta^\op, \infty\Grp)$
preserves small filtered colimits by compactness of objects of $\Theta$ and trivially preserves small coproducts.

\end{remark}

\begin{lemma}\label{helpou}
	
Let $\bk, \n \geq 0$.
There is a canonical isomorphism of $\Theta$-sets
$$\N((\bD^\bk)^{\vee \n}) \to \N(\bD^\bk)_{\Sigma_\n}^{\times \n} .$$
	
\end{lemma}

\begin{proof}

For every $ 1 \leq \ell \leq \n$ let $$\alpha_\ell: \bD^\bk \simeq \{1\}^{\times \ell-1} \times \bD^\bk \times  \{0\}^{\times \n- \ell} \hookrightarrow (\bD^\bk)^{\times \n} $$ be the canonical functor.
Let $\alpha: (\bD^\bk)^{\vee \n} \to (\bD^\bk)^{\times \n}$ be the functor that is the functor $\alpha_\ell$ at the $\ell$-th summand.
We prove that the functor \begin{equation}\label{ertk}
\N((\bD^\bk)^{\vee \n}) \xrightarrow{\N(\alpha)} \N((\bD^\bk)^{\times \n}) \simeq  \N(\bD^\bk)^{\times \n} \to \N(\bD^\bk)_{\Sigma_\n}^{\times \n}\end{equation} is degreewise a bijection.
We first prove that the latter map admits a left inverse.

For every $ 1 \leq \ell \leq \n$ let $$\alpha'_\ell: \bD^\bk \simeq \{1\}^{\boxtimes \ell-1} \boxtimes \bD^\bk \boxtimes  \{0\}^{\boxtimes \n- \ell} \hookrightarrow (\bD^\bk)^{\boxtimes \n} $$ be the canonical functor.
Let $\alpha': (\bD^\bk)^{\vee \n} \to (\bD^\bk)^{\boxtimes \n}$ be the functor that is the functor $\alpha'_\ell$ at the $\ell$-th summand.
The functor $\alpha: (\bD^\bk)^{\vee \n} \to (\bD^\bk)^{\times \n}$ 
factors as $$ (\bD^\bk)^{\vee \n} \xrightarrow{\alpha'} (\bD^\bk)^{\boxtimes \n} \to (\bD^\bk)^{\times \n}$$
since the Gray-tensor product and product share the tensor unit.
For every $\bk \geq 0$ the functor $\bD^\bk \to \cube^\bk$ taking the unique non-degenerate $\bk$-morphisms is a section of the canonical functor $ \cube^\bk \to \bD^\bk$.
For every $ 1 \leq \ell \leq \n$ let $$\beta_\ell: \cube^\bk \simeq \{1\}^{\boxtimes \ell-1} \boxtimes \cube^\bk \boxtimes \{0\}^{\boxtimes \n- \ell} \hookrightarrow (\cube^\bk)^{\boxtimes \n} $$ be the canonical functor.
Let $\beta: (\cube^\bk)^{\vee \n} \to (\cube^\bk)^{\boxtimes \n}$ be the functor that is the functor
$\beta_\ell$ at the $\ell$-th summand.
The functor $\alpha': (\bD^\bk)^{\vee \n} \to (\bD^\bk)^{\boxtimes \n}$ 
is the restriction of $\beta.$ 
By \cite[Theorem 2.1.]{campion2022cubesdenseinftyinftycategories} the functor $\beta$ admits a left inverse. Thus also $\alpha'$ admits a left inverse $\phi: (\bD^\bk)^{\boxtimes \n} \to (\bD^\bk)^{\vee \n} $
given by the restriction of the left inverse of $\beta$ to $(\bD^\bk)^{\boxtimes \n}$ followed by 
the functor $(\cube^\bk)^{\vee \n} \to (\bD^\bk)^{\vee \n}.$ 
By construction and \cite[Corollary 4.5.6.]{gepner2025oriented} the functor $\phi$ factors through the canonical functor $(\bD^\bk)^{\boxtimes \n} \to (\bD^\bk)^{\times \n}$ by a functor $\phi': (\bD^\bk)^{\times \n} \to (\bD^\bk)^{\vee \n}.$
So $\phi'$ is a left inverse of the functor $$ \alpha: (\bD^\bk)^{\vee \n} \to (\bD^\bk)^{\boxtimes \n} \to (\bD^\bk)^{\times \n}.$$
The functor $\phi'$ is $\Sigma_\n$-equivariant for the permutation action on the source and trivial action on the target, and so induces a map 
$$\psi: \N(\bD^\bk)^{\times \n}_{\Sigma_\n} \to \N((\bD^\bk)^{\vee \n}).$$
So $\psi $ is a left inverse of the map (\ref{ertk}). We prove next that $\psi$ is degree-wise surjective.

Let $\Xi \subset (\bD^\bk)^{\times \n}$ be the subcategory whose $\ell$-morphisms for $0 \leq \ell \leq \bk$ are the $\n$-tuples $(X^1,...,X^\n)$ of $\ell$-morphisms in $\bD^\bk$ such that there are $ i_0 \leq  ... \leq i_\ell \leq j_\ell \leq ... \leq j_0 $  
such that $(X^{i_{\ell}+1}, ..., X^{j_{\ell}})$ consists of non-degenerate $\ell$-morphisms in $\bD^\bk$ corresponding to an increasing sequence in $\{0,1\}$ if $\ell < \bk$ and for every $ 0 \leq \m < \ell$ the family $$(X^{i_\m+1}, ..., X^{j_\m}) \setminus (X^{i_{\m+1}+1}, ..., X^{j_{\m+1}}) $$ consists of non-degenerate $\m$-morphisms (viewed as $\ell$-morphisms) and such that source and target of the family $(X^{i_\m+1}, ..., X^{j_\m})$ consist of non-degenerate $\m-1$-morphisms
(viewed as $\ell-1$-morphisms) that correspond to an increasing sequence in $\{0,1\}$.

For every $1 \leq i \leq \n$ the functor $\alpha_i$ lands in $\Xi. $ 
Moreover for every $ 0 \leq \ell \leq k$ and $\ell$-morphism in $\Xi$ corresponding to an $\n$-tuple $ (X^1,...,X^\n) $ of $\ell$-morphisms of $\bD^\bk$ we have
$$ (X^1,...,X^\n) = (X_1 , 1 ,..., 1)  \circ ... \circ (0, ...0, X^{\n-2}, 1, 1) \circ (0, ...0, X^{\n-1}, 1) \circ (0 , ...,0, X^\n) $$$$ = \alpha_1(X^1) \circ ... \circ  \alpha_\n(X^\n).$$
Thus every $\ell$-morphism of $\Xi$ is in the image of $(\bD^\bk)^{\vee \n}.$
Hence the inclusion $\alpha: (\bD^\bk)^{\vee \n} \to (\bD^\bk)^{\times \n}$ induces an equivalence
$ (\bD^\bk)^{\vee \n} \simeq \Xi. $
So it suffices to show that the degree-wise injection $$\N(\Xi) \subset \N((\bD^\bk)^{\times \n}) \simeq  \N(\bD^\bk)^{\times \n} \to \N(\bD^\bk)_{\Sigma_\n}^{\times \n}$$ is degree-wise surjective.

So we have to see that for every $\theta \in \Theta$ and $(X^1,...,X^\n) : \theta \to (\bD^\bk)^{\times \n}$ there is a $\sigma \in \mathrm{Aut}(\{1,...,\n\})$ such that $(X^{\sigma(1)},...,X^{\sigma(\n)}) : \theta \to (\bD^\bk)^{\times \n}$ lands in $\Xi.$
As the truncation of an object of $\Theta$ belongs to $\Theta$, we can assume that $\theta$ is a $\bk$-category.

By induction on the number of disks appearing in a pushout decomposition of objects of $\Theta$, it suffices to show the following:
for every $\theta \in \Theta$ or $\theta$ the empty category, naturals $a,\ell \geq 0$, inclusions $\bD^\ra \to \bD^\ell, \bD^\ra \to \theta$ and $(X^1,...,X^\n) : \theta \coprod_{\bD^\ra} \bD^\ell \to (\bD^\bk)^{\times \n}$ such that $(X^1_{\mid \theta},..., X^\n_{\mid \theta}) : \theta \to (\bD^\bk)^{\times \n}$
lands in $\Xi$, there is a $\sigma \in \mathrm{Aut}(\{1,...,\n\})$ such that
$(X^{\sigma(1)}_{\mid \bD^\ell},..., X^{\sigma(\n)}_{\mid \bD^\ell}) : \bD^\ell \to (\bD^\bk)^{\times \n}$
lands in $\Xi$ and $$(X^{\sigma(1)}_{\mid \theta},...,X^{\sigma(\n)}_{\mid \theta}) = (X^1_{\mid \theta},..., X^\n_{\mid \theta}) : \theta \to (\bD^\bk)^{\times \n}.$$
In this case $ (X^{\sigma(1)}, ..., X^{\sigma(\n)}) : \theta \coprod_{\bD^\ra} \bD^\ell \to (\bD^\bk)^{\times \n}$ lands in $\Xi$.
This implies the claim since every object of $\Theta$ has finitely many objects.
By induction on dimension it is enough to prove the following statement:
for every $\theta \in \Theta$ or $\theta$ the empty category, naturals $\ra,\ell \geq 0, 0 \leq c \leq \ell $, inclusions $\bD^\ra \to \bD^\ell, \bD^\rc \to \bD^\ell, \bD^\ra \to \theta$ and $(X^1,...,X^\n) : \theta \coprod_{\bD^\ra} \bD^\ell \to (\bD^\bk)^{\times \n}$ such that $$(X^1_{\mid \theta},..., X^\n_{\mid \theta}) : \theta \to (\bD^\bk)^{\times \n}, \ (X^1_{\mid\partial \bD^{\rc}},..., X^\n_{\mid \partial\bD^{c}}) : \partial\bD^{c} \to (\bD^\bk)^{\times \n}$$ land in $\Xi$, there is a $\sigma \in \mathrm{Aut}(\{1,...,\n\})$ such that the functor
$(X^{\sigma(1)}_{\mid \bD^\rc},..., X^{\sigma(\n)}_{\mid \bD^c}) : \bD^c \to (\bD^\bk)^{\times \n}$
lands in $\Xi$ and $$ (X^{\sigma(1)}_{\mid \theta},...,X^{\sigma(\n)}_{\mid \theta}) = (X^1_{\mid \theta},..., X^\n_{\mid \theta}), \ (X^{\sigma(1)}_{\mid \partial\bD^{c}},...,X^{\sigma(\n)}_{\mid \partial\bD^{c}}) = (X^1_{\mid \partial\bD^{c}},..., X^\n_{\mid \partial\bD^{c}}).$$
Consequently, we can assume that $c=\ell.$ We complete the proof by verifying the latter statement.
For $\ell=0$ there is nothing to show.

Let $(X^1,...,X^\n) : \theta \coprod_{\bD^\ra} \bD^\ell \to (\bD^\bk)^{\times \n}$ such that $$(X^1_{\mid \theta},..., X^\n_{\mid \theta}) : \theta \to (\bD^\bk)^{\times \n}, \ (X^1_{\mid\partial \bD^{\ell}},..., X^\n_{\mid \partial\bD^{\ell}}) : \partial\bD^{\ell} \to (\bD^\bk)^{\times \n}$$ land in $\Xi$.
Then there are $ i_0 \leq  ... \leq i_\ell \leq j_\ell \leq ... \leq j_0 $  such that $(X_{\mid \bD^\ell}^{i_{\ell}+1}, ..., X_{\mid \bD^\ell}^{j_{\ell}})$ consists of non-degenerate $\ell$-morphisms in $\bD^\bk$, which have all a common source and a common target, 
and for every $ 0 \leq \m < \ell$ the family $$(X^{i_\m+1}, ..., X^{j_\m}) \setminus (X^{i_{\m+1}+1}, ..., X^{j_{\m+1}}) $$ consists of non-degenerate $\m$-morphisms (viewed as $\ell$-morphisms) and such that the source and target of the family $(X^{i_\m+1}, ..., X^{j_\m})$ consists of non-degenerate $\m$-1-morphisms (viewed as $\ell-1$-morphisms) that correspond to an increasing sequence in $\{0,1\}$.
Then trivially there is a $\sigma \in \mathrm{Aut}(\{1,...,\n\})$ such that 
$\sigma$ is the identity on $ \{1,...,\n\}\setminus \{i_\ell +1, ..., j_\ell\}$ and 
$(X^{\sigma(i_{\ell}+1)}, ..., X^{\sigma(j_{\ell})})$ is an increasing (possibly empty) family.
Since $(X^{i_{\ell}+1}, ..., X^{j_{\ell}})$ have all a common source and common target, we find that $$(X_{\mid \partial\bD^{\ell}}^{\sigma(i_{\ell}+1)}, ..., X_{\mid \partial\bD^{\ell}}^{\sigma(j_{\ell})})= 
(X_{\mid \partial\bD^{\ell}}^{i_{\ell}+1}, ..., X_{\mid \partial\bD^{\ell}}^{j_{\ell}})$$ and so $$(X_{\mid \partial\bD^{\ell}}^{\sigma(1)}, ..., X_{\mid \partial\bD^{\ell}}^{\sigma(\n)})= 
(X_{\mid \partial\bD^{\ell}}^{1}, ..., X_{\mid \partial\bD^{\ell}}^{\n}).$$
 
Moreover note that we could have assumed that $ i_0 <  ... < i_\ell < j_\ell < ... < j_0 $:
indeed if $i_\m = j_\m$ for some $1 \leq \m \leq \ell$, then $i_t = j_t$ for every $m \leq \rt \leq \ell$
and so $i_\ell = j_\ell$ so that the set $ \{i_\ell+1, ...,j_\ell \}$ is empty and we could have chosen
$\sigma$ to be the identity.
It remains to see that 
$(X_{\mid\theta}^{\sigma(i_\ell+1)}, ..., X_{\mid\theta}^{\sigma(j_\ell)})= (X_{\mid\theta}^{i_\ell+1}, ..., X_{\mid\theta}^{j_\ell})$ and so $(X_{\mid\theta}^{\sigma(1)}, ..., X_{\mid\theta}^{\sigma(\n)})= (X_{\mid\theta}^{1}, ..., X_{\mid\theta}^{\n})$.
For this we can assume that $ i_0 <  ... < i_\ell < j_\ell < ... < j_0 $.
Otherwise, we could choose $\sigma$ to be the identity.
We have to prove that for every disk $\bD^\rt$ for $0 \leq \rt \leq k$ appearing in the pushout decomposition of $\theta$ restriction along the canonical inclusion $\bD^\rt \to \theta$ associated to that disk gives the following identity:
$$(X_{\mid\bD^\rt}^{\sigma(i_\ell+1)}, ..., X_{\mid\bD^\rt}^{\sigma(j_\ell)})= (X_{\mid\bD^\rt}^{i_\ell+1}, ..., X_{\mid\bD^\rt}^{j_\ell}).$$

Note that any canonical inclusion $\bD^\rt \to \theta$ sends $\{0,1\}$ to the set of two subsequent objects. Moreover note that for every $ 0 \leq \rt \leq \br$ the induced inclusion
$\bD^{\br-\rt} \to \Mor^{\circ \rt}(\theta)$ sends $\{0,1\}$ to the set of two subsequent objects.

Let $ 0 \leq \rt \leq \br$. If $a \leq \rt$ or if the induced inclusion $\bD^{\br-\rt} \to \Mor^{\circ \rt}(\theta)$ sends 1 to an object $ i$ different from the last object of $\Mor^{\circ \rt}(\theta)$, 
for every $ i_\ell +1 \leq \m \leq j_\ell $ the functor $ \Mor^{\circ \rt}(X^{\m}_{\mid \bD^\rt}) : \bD^{\br-\rt} \to \bD^{\bk-\rt} $ sends 0 and 1 both to 0 because the functor $ \Mor^{\circ \rt}(X^\m_{\mid \theta}) :  \Mor^{\circ \rt}(\theta) \to \bD^{\bk-\rt} $  sends the object before the last object to 0 and so sends every object smaller than the last object to 0.
Therefore for every $ i_\ell +1 \leq \m \leq j_\ell $ the functor $\Mor^{\circ \rt}(X^\m_{\mid \bD^\rt}) : \bD^{\br-\rt} \to \bD^{\bk-\rt} $ is the constant functor at 0. So all functors in the family $(X^{i_\ell+1}_{\mid \bD^\rt}, ..., X^{j_\ell}_{\mid \bD^\rt})$ agree and we have proven the statement if for every $\s < \rt$ the functor $\Mor^{\circ \s}(X^\m_{\mid \bD^\rt}) : \bD^{\br-\s} \to \bD^{\bk-\s} $
sends 0 to 0 and 1 to 1.
If the induced inclusion $\bD^{\br-\rt} \to \Mor^{\circ \rt}(\theta)$ sends 1 to the last object, the functor $\Mor^{\circ \rt}(X^\m_{\mid \bD^{\br}}) : \bD^{\br-\rt }\to \bD^{\bk-\rt} $ sends 0 to 0 and 1 to 1.

We use the latter to finish the proof: if $\ra=0$ or the induced inclusion $\bD^{\br} \to \theta$ sends 1 to an object different from the last object of $\theta$, we have proven the statement.
Otherwise the induced inclusion $\bD^{\br} \to \theta$ sends 1 to the last object and the functor $X^\m_{\mid \bD^{\br}} : \bD^{\br}\to \bD^{\bk} $ sends 0 to 0 and 1 to 1 and so induces
a functor $\Mor(X^\m_{\mid \bD^{\br}}) : \bD^{\br-1}\to \bD^{\bk-1} $.
If $\ra=1$ or the induced inclusion $\bD^{\br-1} \to \Mor(\theta)$ sends 1 to an object different from the last object of $\Mor(\theta)$, we have proven the statement.
Otherwise the induced inclusion $\bD^{\br-1} \to \Mor(\theta)$ sends 1 to the last object and the functor $\Mor(X^\m_{\mid \bD^{\br}}) : \bD^{\br-1}\to \bD^{\bk-1} $ sends 0 to 0 and 1 to 1 and so induces
a functor $\Mor^2(X^\m_{\mid \bD^{\br}}) : \bD^{\br-2}\to \bD^{\bk-2} $.
So we continue $\ra-1 $-many steps and obtain the result from $\ra=\ra$.
 
\end{proof}

\begin{corollary}
Let $\bk, \n \geq 0$. There is a canonical equivalence in $\Fun(\Theta^\op, \infty\Grp): $$$ \D^\n \circ \N(\bD^\bk)\simeq \N((\bD^\bk)^{\vee \n}).$$
	
\end{corollary}

\begin{corollary} \label{coufo}
Let $\bk \geq 0.$ There are canonical equivalences in $\Fun(\Theta^\op, \infty\Grp): $
	
\begin{enumerate}
\item $$ \coprod_{\n \geq 0} \D^\n \circ \N(\bD^\bk)\simeq \N(\coprod_{\n \geq 0} (\bD^\bk)^{\vee \n}) .$$
\item $$ \colim_{\n \geq 0} \D^\n \circ \N(\bD^\bk) \simeq\N(\colim_{\n \geq 0} (\bD^\bk)^{\vee \n}).$$
\end{enumerate}	
\end{corollary}

We obtain the following:

\begin{theorem}\label{disk1}

Let $\bk \geq 0.$ The following $\infty$-categories are equivalent:

\begin{enumerate}
\item The categorical homology of $\bD^\bk_+$, i.e. $ \Omega^\infty(H(\bN)\wedge \bD_+^\bk)$.

\item The free strict commutative monoid in $\Fun(\Theta^\op,\infty\Grp) $ generated by the nerve of $\bD^\bk$.

\item The free strict symmetric monoidal $\infty$-category generated by $\bD^\bk$.

\item The free strict symmetric monoidal strict $\infty$-category generated by $\bD^\bk$.

\item The coproduct $ \coprod_{\n \geq 0} (\bD^\bk)^{\vee \n}$

\end{enumerate}

\end{theorem}

\begin{proof}
By 	\cref{theuta} the objects of (1) and (3) are equivalent. By 	\cref{free} the categorical homology of $\bD^\bk_+$ is the coproduct $ \coprod_{\n \geq 0} \D^\n(\bD^\bk),$ which by definition is the localization of $ \coprod_{\n \geq 0} \D^\n \circ \N(\bD^\bk),$
where $$\N: \infty\Cat \to \Fun(\Theta^\op, \infty\Grp)$$ is the restricted Yoneda embedding. 

By \cref{coufo} the presheaf $ \coprod_{\n \geq 0} \D^\n \circ \N(\bD^\bk)$ is equivalent to $ \N(\coprod_{\n \geq 0} (\bD^\bk)^{\vee \n}) $ and so is local. Thus the localization of $$ \coprod_{\n \geq 0} \D^\n \circ \N(\bD^\bk) \simeq \N(\coprod_{\n \geq 0} (\bD^\bk)^{\vee \n})$$ is $\coprod_{\n \geq 0} (\bD^\bk)^{\vee \n} $ whose image under $\N$ is $ \coprod_{\n \geq 0} \D^\n \circ \N(\bD^\bk), $ which by \cref{polujk} is the free strict commutative monoid on $\N(\bD^\bk).$ So the objects of (1) and (5) and (2) are equivalent.

The free strict symmetric monoidal strict $\infty$-category on $\bD^\bk$ is 
the localization of $$ \coprod_{\n \geq 0} \D^\n \circ \N(\bD^\bk)\simeq \coprod_{\n \geq 0} \N(\bD^\bk)^{\times \n}_{\Sigma_\n}\simeq \coprod_{\n \geq 0} \N((\bD^\bk)^{\times \n})_{\Sigma_\n}  \in \Fun(\Theta^\op, \Set) $$
with respect to the localization $  \Fun(\Theta^\op, \Set)  \to \infty\Cat^\strict. $ 
By \cref{coufo} the presheaf $$ \coprod_{\n \geq 0} \D^\n \circ \N(\bD^\bk) \in \Fun(\Theta^\op, \Set) $$ is $ \N(\coprod_{\n \geq 0} (\bD^\bk)^{\vee \n})$ and so is local. Thus the free strict symmetric monoidal strict $\infty$-category on $\bD^\bk$ is $ \coprod_{\n \geq 0} (\bD^\bk)^{\vee \n} $. So the objects of (4) and (5) are equivalent.
 
\end{proof}

\begin{theorem}\label{disk2}
Let $\bk \geq 0.$ The following $\infty$-categories are equivalent:

\begin{enumerate}
\item The categorical homology of $\bD^\bk$, i.e. $ \Omega^\infty(H(\bN)\wedge \bD^\bk)$.

\item The free reduced strict commutative monoid in $\Fun(\Theta^\op,\infty\Grp) $ generated by the nerve of $\bD^\bk$.

\item The free reduced strict symmetric monoidal $\infty$-category generated by $\bD^\bk$.

\item The free reduced strict symmetric monoidal strict $\infty$-category generated by $\bD^\bk$.

\item The sequential colimit of the diagram $$ \bD^0 \to \bD^\bk \to ... \to (\bD^\bk)^{\vee \n} \to ... $$

\end{enumerate}

\end{theorem}

\begin{proof}
By 	\cref{theuta} the objects of (1) and (3) are equivalent. By 	\cref{redfree} the categorical homology of $\bD^\bk$ is the sequential colimit $ \colim_{\n \geq 0} \D^\n(\bD^\bk),$ which by definition is the localization of $ \colim_{\n \geq 0} \D^\n \circ \N(\bD^\bk),$
where $\N: \infty\Cat \to \Fun(\Theta^\op, \infty\Grp)$ is the restricted Yoneda embedding. 
By \cref{coufo} the presheaf $ \colim_{\n \geq 0} \D^\n \circ \N(\bD^\bk)$ is equivalent to $ \N( \colim_{\n \geq 0} (\bD^\bk)^{\vee \n})$  and so is local.
Hence the categorical homology of $\bD^\bk$ is $ \colim_{\n \geq 0} (\bD^\bk)^{\vee \n} $ whose image under $\N$ is $ \colim_{\n \geq 0} \D^\n \circ \N(\bD^\bk), $ which by \cref{polujk} is the free reduced strict commutative monoid on $\N(\bD^\bk).$ So the objects of (1) and (5) and (2) are equivalent.

The free reduced strict symmetric monoidal strict $\infty$-category on $\bD^\bk$ is 
the localization of $$ \colim_{\n \geq 0} \D^\n \circ \N(\bD^\bk) \simeq \colim_{\n \geq 0} \N(\bD^\bk)^{\times \n}_{\Sigma_\n}  \simeq \colim_{\n \geq 0} \N((\bD^\bk)^{\times \n})_{\Sigma_\n}  \in \Fun(\Theta^\op, \Set) $$
with respect to the localization $  \Fun(\Theta^\op, \Set) \to \infty\Cat^\strict. $ 
By \cref{coufo} the presheaf $$ \colim_{\n \geq 0} \D^\n \circ \N(\bD^\bk) \in \Fun(\Theta^\op, \Set) $$ is $ \N( \colim_{\n \geq 0} (\bD^\bk)^{\vee \n})$ and so is local. Hence the free strict symmetric monoidal strict $\infty$-category on $\bD^\bk$ is $ \colim_{\n \geq 0} (\bD^\bk)^{\vee \n} $. So the objects of (4) and (5) are equivalent.	
\end{proof}

\subsection{A categorical cup product}

In the following we construct a higher-categorical version of the cup product on the categorical cohomology of the oriented simplices.

\begin{proposition}\label{coalgos}

Let $n \geq 0.$
The oriented $n$-simplex $\bDelta^n$ carries a canonical coalgebra structure for $(\infty\Cat, \boxtimes),$ which is functorial in 
$[\n] \in \Delta.$


\end{proposition}

\begin{proof}

By \cite[Example 3.8.]{Steiner2004} the oriented $n$-simplex $\bDelta^n$ is a Steiner $\infty$-category (see \cite[Definition 1.27.]{ARA2023107313} for details on Steiner $\infty$-categories).
Let $\infty\Cat^{\Steiner}$ be the full subcategory of the category $\infty\Cat $ spanned by the Steiner $\infty$-categories. 
By \cite[Proposition 3.2.13.]{gepner2025oriented} the category $\infty\Cat^{\Steiner}$ is a $(1,1)$-category.
By \cite[Theorem 3.8.1.]{gepner2025oriented} the Gray tensor product of $\infty\Cat$ restricts to $\infty\Cat^{\Steiner}$. So it suffices to refine $\bDelta^n$ to a coalgebra in $\infty\Cat^{\Steiner}$. 

An augmented directed chain complex is an augmented connective chain complex of abelian groups together with a graded submonoid (see \cite[Definition 2.2.]{Steiner2004} for details). 
An augmented directed chain map is an augmented chain map preserving the graded submonoids.
Let $\A\D\C $ be the $(1,1)$-category of augmented directed chain complexes and augmented directed chain maps.
For every simplicial set $X$ the normalized chains $\C(X)$ refine to an augmented directed complex $ \bar{\C}(X)$ whose graded submonoid is free on $X.$

By \cite[Example 3.10.]{Steiner2004} the $(1,1)$-category $\A\D\C $ carries a canonical monoidal structure
such that the forgetful functor $\A\D\C \to \Ch_{\geq 0}$ is monoidal.
The graded submonoid of the tensor product is the image of the 
tensor product of graded commutative submonoids in the tensor product of graded abelian groups.
By \cite[Theorem 5.11., Proposition 7.2., Example 3.10.]{Steiner2004} there is a canonical monoidal embedding
$\lambda: \infty\Cat^{\Steiner} \to \A\D\C.$ 

Consequently, it suffices to refine the image $\lambda(\bDelta^n)$ in $\A\D\C$ to a coalgebra in $\A\D\C$. 
This image is $\bar{\C}(\Delta^n).$ The normalized chains functor $\C: \s\Set \to \Ch_{\geq 0}$
refines the composition of the free simplicial abelian group functor $\s\Set \to \s\Ab$ followed by the normalized chains functor $ \s\Ab \to \Ch_{\geq 0}.$
The free simplicial abelian group functor $\s\Set \to \s\Ab$ is monoidal and the normalized chains functor $\s\Ab \to \Ch_{\geq 0}$
is oplax monoidal via the Alexander-Whitney map and preserves the tensor unit.
Hence the composition $\s\Set \to \s\Ab \to \Ch_{\geq 0}$ is oplax monoidal
and preserves the tensor unit. So also the functor $\bar{\C}: \s\Set \to \A\D\C$ is oplax monoidal and preserves the tensor unit
since the Alexander-Whitney map is an augmented directed chain map. 
Hence for every simplicial set $X$, viewed as coalgebra for the cartesian product, the chains $\bar{\C}(X)$ is a coalgebra in $\A\D\C$, which is functorial in $X \in \sSet$.
So $\bar{\C}(\Delta^n)$ is functorial in $\bDelta^n \in \sSet$ for $[n] \in \Delta.$
The functor $ \Delta \to \A\D\C, [\n] \mapsto \bar{\C}(\Delta^n) $
identifies with the functor
$$ \bDelta \to \infty\Cat^{\Steiner} \xrightarrow{\lambda} \A\D\C, \ [\n] \mapsto \lambda(\bDelta^n).$$

\end{proof}

\begin{corollary}\label{homolalg}

Let $n \geq 0.$
The non-reduced categorical homology $$ H(\bN) \wedge \bDelta^n_+ $$ of $\bDelta^n$ refines to a coalgebra in $ H(\bN)\mathrm{-}\Mod(\Cat\Sp).$

\end{corollary}

\begin{proof}
By \cref{catmonoidal} the functor $H(\bN) \wedge (-)_+: \infty\Cat \to H(\bN)\mathrm{-}\Mod(\Cat\Sp)$ is monoidal and so preserves the coalgebra structure on $\bDelta^n$ of \cref{coalgos}.

\end{proof}

\begin{definition}Let $R$ be a rig and $X$ an $\infty$-category.
The categorical $R$-cohomology of $X$ is 
$$ \R\Mor_{\Cat\Sp}(\Sigma^\infty_+(X), H(R)) \simeq \R\Mor_{H(R)\mathrm{-}\Mod(\Cat\Sp)}(H(R) \wedge \Sigma^\infty_+(X), H(R)). $$
    
\end{definition}

We obtain the following cup product on categorical cohomology of an oriented simplex.

\begin{corollary}\label{homolcoalg}

Let $n \geq 0.$
The non-reduced categorical cohomology of $\bDelta^n$ refines to an algebra in $$H(\bN)\mathrm{-}\Mod(\Cat\Sp).$$

\end{corollary}

\begin{proof}

For every closed monoidal category $\mV$
the functor $\R\Mor_\mV(-,\tu_\mV): \mV^\op \to \mV^\rev$ is lax monoidal.
We apply this to the presentably monoidal category 
$H(\bN)\mathrm{-}\Mod(\Cat\Sp)$ of \cref{catmonoidal} whose tensor unit is $H(\bN).$
The resulting lax monoidal functor $$ \R\Mor_{H(\bN)\mathrm{-}\Mod(\Cat\Sp)}(-,H(\bN)): H(\bN)\mathrm{-}\Mod(\Cat\Sp)^\op \to H(\bN)\mathrm{-}\Mod(\Cat\Sp)^\rev$$ transports the coalgebra structure on $ H(\bN) \wedge \bDelta^n_+ $ in $H(\bN)\mathrm{-}\Mod(\Cat\Sp).$

\end{proof}




\bibliographystyle{plain}
\bibliography{ma}




\end{document}